\documentclass{amsart}
\pdfoutput=1

\usepackage{amsfonts, amssymb, amsmath, amsthm}                               
\usepackage{slashed}                                                          

\newtheorem{theorem}{Theorem} 	      	      	                              
\newtheorem{corollary}[theorem]{Corollary}     	      	      	      	      
\newtheorem{lemma}[theorem]{Lemma}     	       	      	      	      	      
\newtheorem{proposition}[theorem]{Proposition} 	      	      	      	      
\newtheorem{definition}[theorem]{Definition}                                  
\newtheorem*{remark}{Remark}                                                  
\newtheorem*{ackn}{Acknowledgments}                                           

\numberwithin{equation}{section}                                              
\numberwithin{theorem}{section}                                               

\newcommand{\ul}[1]{\underline{#1}}                                           
\newcommand{\mf}[1]{\mathfrak{#1}}                                            
\newcommand{\mc}[1]{\mathcal{#1}}                                             

\newcommand{\R}{\mathbb{R}}                                                   
\newcommand{\Sph}{\mathbb{S}}                                                 

\newcommand{\paren}[1]{\left(#1\right)}                                       
\newcommand{\brak}[1]{\left[#1\right]}                                        

\newcommand{\g}{{\bf g}}                                                      
\newcommand{\trase}{/\mspace{-14mu}\trace}                                    
\newcommand{\sorm}[1]{/\mspace{-11mu}#1}   

\DeclareMathOperator{\real}{Re}                                               
\DeclareMathOperator{\imag}{Im}                                               

\DeclareMathOperator{\trace}{tr}                                              

\newcommand{\nasla}{\slashed{\nabla}}                                         
\newcommand{\gras}{\slashed{\operatorname{grad}}}                             
\newcommand{\lapl}{\Delta}                                                    
\newcommand{\cint}{{\begingroup\textstyle\smallint\endgroup}}                 

\newcommand{\mind}{/\mspace{-9mu}\gamma}                                      
\newcommand{\vind}{/\mspace{-9mu}\epsilon}                                    
\newcommand{\Samma}{/\mspace{-9mu}\Gamma}                                     

\newcommand{\ass}[2]{{\bf (#1)}${}_{#2}$}                                     
\newcommand{\fol}[2]{#1^{#2}}                                                 

\allowdisplaybreaks

\begin{document}

\title{Bounds on the Bondi Energy by a Flux of Curvature}
\author{Spyros Alexakis and Arick Shao}

\address{Department of Mathematics, University of Toronto, Toronto, ON, Canada M5S 2E4}
\email{alexakis@math.utoronto.ca, ashao@math.utoronto.ca}

\subjclass[2010]{83C30 (Primary) 35Q75, 83C05, 53C12, 53C21 (Secondary)}

\begin{abstract}
We consider smooth null cones in a vacuum spacetime that extend to future null infinity.
For such cones that are perturbations of shear-free outgoing null cones in Schwarzschild spacetimes, we prove bounds for the Bondi energy, momentum, and rate of energy loss.
The bounds depend on the closeness between the given cone and a corresponding cone in a Schwarzschild spacetime, measured purely in terms of the differences between certain weighted $L^2$-norms of the spacetime curvature on the cones, and of the geometries of the spheres from which they emanate.
This paper relies on the results in \cite{alex_shao:nc_inf}, which uniformly control the geometry of the given null cone up to infinity, as well as those of \cite{shao:stt}, which establish machinery for dealing with low regularities.
A key step in this paper is the construction of a family of asymptotically round cuts of our cone, relative to which the Bondi energy is measured.
\end{abstract}

\maketitle

\tableofcontents

\section{Introduction} \label{sect:intro}

This paper deals with smooth null hypersurfaces extending to null infinity in a four-dimensional vacuum spacetime, $(M, \g)$.
Our primary aim is to control the Bondi mass associated with such a hypersurface $\mc{N}$ by the $L^2$-norms of certain suitably weighted spacetime curvature components over $\mc{N}$.
Furthermore, we control the rate of energy loss through $\mc{N}$, as well as the linear momentum associated with $\mc{N}$.
Our result applies when $\mc{N}$ is sufficiently close, at the above $L^2$-curvature level, 
to a standard shear-free outgoing null cone in a Schwarzschild exterior.

Our result has two essential features.
The first is the low regularity of our setting: we operate purely at the level of the (weighted) spacetime curvature lying in $L^2$ over $\mc{N}$, that is, at the same regularity as in \cite{kl_rod:bdc, kl_rod_szf:blc}.
Thus, we inherit some of the same difficulties present in these works.
The other feature is that our result depends only on the geometries of the null cone $\mc{N}$ and the sphere from which $\mc{N}$ emanates.
In other words, we make no assumptions on the global structure of the ambient spacetime $(M, \g)$, besides that it is vacuum.
In particular, we impose no conditions on the existence or the structure of a null infinity on the spacetime.

This paper relies heavily on \cite{alex_shao:nc_inf}, which proved, in the aforementioned setting, that the instrinsic and extrinsic geometry of $\mc{N}$ is controlled uniformly up to infinity.
Moreover, several techniques and estimates in this paper depend on \cite{shao:stt}, which developed many of the tools needed for working with $\mc{N}$ at the $L^2$-curvature level.

The main new ingredient that we introduce here is the construction of an appropriate 1-parameter family of spherical cuts of $\mc{N}$ which become asymptotically round near infinity.
We show that the Bondi energy, momenta, and rate of energy loss, defined relative to these spherical cuts, can be controlled by this weighted curvature flux through $\mc{N}$.
More specifically, to control the Bondi energy associated with $\mc{N}$, we show that the Hawking masses of these 
spherical cuts converge to a limit at infinity that remains close to the corresponding Schwarzschild mass.

\subsection{Main Quantities} \label{sect:intro.quantities}

For the reader's convenience, we first briefly present the definitions and the physical significance of the quantities under consideration.

\subsubsection{Curvature Flux}

Recall the Bel-Robinson tensor, a symmetric divergence-free 4-tensor which is quadratic in the Weyl curvature $W$ of $(M, \g)$:
\footnote{In fact one can form the Bel-Robinson tensor out of any Weyl field.}
\[ \mc{Q}_{abcd} = W_a{}^e{}_b{}^f W_{cedf} + {}^\star W_a{}^e{}_b{}^f {}^\star W_{cedf} \text{.} \]
Note since $(M, \g)$ is vacuum, $W$ coincides with the Riemann curvature tensor.
A well-known application of the Bel-Robinson tensor is toward energy estimates: by contracting $\mc{Q}_{abcd}$ against three future-directed causal vector fields $A, B, C$, one obtains a current, $\mc{J}_d := \mc{Q}_{abcd} A^a B^b C^c$, to which one can apply the divergence theorem on any bounded domain $\Omega \subset M$ with a partially smooth boundary. 
If $\partial \Omega$ contains a portion of a null hypersurface $\mc{N}$ whose affine tangent null vector field is $L$, then the corresponding boundary term, i.e., the \emph{curvature flux} through $\mc{N}$, is
\[ \mc{F} = \int_{\mc{N}} \mc{J}_d L^d d V_{ \tilde{\g} } \text{,} \]
where $d V_{ \tilde{\g} }$ is the canonical volume form on $\mc{N}$ associated to $L$.
One can then see that, for any choice of $A$, $B$, and $C$, the curvature flux $\mc{F}$ will be comparable to the $L^2$-norm on $\mc{N}$ of some, but not all, of the independent components of $W$.
\footnote{In particular, these components include all those listed in \eqref{list}, except for $\ul{\alpha}$.}

Here, we wish to consider not the curvature flux of $\mc{N}$ itself, but rather a measure of how much $\mc{N}$ deviates from a standard shear-free null cone, $\mc{N}_{\rm S}$, in a Schwarzschild spacetime $(M_{\rm S}, \g_{\rm S})$ of mass $m_{\rm S} \geq 0$.
Therefore, the relevant \emph{curvature flux deviation} will be a weighted $L^2$-norm on $\mc{N}$ of the \emph{difference between the flux components of $W$ and the corresponding components for $\mc{N}_{\rm S}$}.
\footnote{In terms of the curvature decomposition of \eqref{list}, all these components for the shear-free Schwarzschild null cone $\mc{N}_{\rm S}$ vanish, except for $\rho_{\rm S}$, which is precisely $-2 m_{\rm S} r^{-3}$.
Here, $r$ is the usual radial coordinate in $(M_{\rm S}, \g_{\rm S})$; note $r$ is also an affine parameter for $\mc{N}_{\rm S}$.}
In fact, our main assumption will be that this curvature flux deviation is sufficiently small.
Under such assumptions, we will show that the Bondi energy of $\mc{N}$ is close to $m_{\rm S}$, with the closeness being controlled by this curvature flux deviation.

We also mention that in the case $m_{\rm S}=0$ (that is, when $\mc{N}$ is close to the standard affine-parametrized null cones in Minkowski spacetime), our weighted curvature flux deviation arises naturally in asymptotically Minkowskian spacetimes, in the case $A={\bf K}$, $B={\bf K}$, and $C={\bf T}$.
Here, ${\bf K}$ and ${\bf T}$ are suitable adaptations of the Morawetz and the time-translation ($\partial_t$) vector fields in the Minkowksi spacetime $(M_0, \g_0)$.  
These weighted curvature fluxes appear in \cite{bie_zip:stb_mink, chr_kl:stb_mink, kl_nic:stb_mink}.
Moreover, this is the main motivation for our choice of weights in \cite{alex_shao:nc_inf} and in this paper.

\subsubsection{Bondi Mass and Energy}

The Bondi mass measures the amount of gravitational mass remaining in an isolated system as measured at null infinity, at a given retarded time.
To be more specific, in the context of a spacetime $(M, \g)$ with a smooth enough and complete future null infinity $\mc{I}^+$, with the topology of $\R \times \Sph^2$, we consider an infinite outgoing null cone $\mc{N}$ which intersects $\mc{I}^+$ along a spherical cut, $\mc{S}_\infty \subset \mc{I}^+$.
The Bondi mass $m_{\rm B} = m_{\rm B}( \mc{N} )$ then measures the amount of mass remaining in the system after radiation emitted through $\mc{I}^+$ up to this cut $\mc{S}_\infty$. 

This quantity was originally defined in \cite{bondi_burg_metz:bondi_mass} by stipulating the existence of a system of Bondi coordinates near $\mc{S}_\infty$.
An alternative definition of the Bondi momentum 4-vector, using a conformal compactification of spacetime, can be found in \cite{pen_rind:spinor}; the reader may also refer to \cite{sauter:penrose}, where the notation adopted here is presented. 

The Bondi momentum $4$-vector $( E_B^{\gamma_\infty}, \vec{P}_{\rm B}^{\gamma_\infty} )$, where $\gamma_\infty$ refers to the round metric induced on $\mc{S}_\infty$ by the above conformal compactification, is conformally covariant.
(See the discussion below by equation \eqref{Bondi.def} for a more precise description of $\gamma_\infty$.)
This reflects the action of the conformal group on $\gamma_\infty$.
Moreover, the $3$-vector $\vec{P}^{\gamma_\infty}_{\rm B}$ corresponds to the linear momentum at $\mc{S}_\infty$ (relative to $\gamma_\infty$), while the number $E^{\gamma_\infty}_{\rm B}$ corresponds to the Bondi energy.
Nonetheless, the Minkowski norm of this 4-vector is \emph{invariantly} defined and corresponds to the Bondi mass of $\mc{S}_\infty$:
\footnote{In other words, it is invariant under the action of the conformal group.}
\begin{equation} \label{Bondi.all} m_{\rm B}^{\gamma_\infty} = \sqrt{ ( E^{\gamma_\infty}_{\rm B} )^2 - | \vec{P}^{\gamma_\infty}_{\rm B} |^2 } \text{.} \end{equation}

In fact, the choice of round metric $\gamma_\infty$ corresponds to a choice of a family of asymptotically round $2$-spheres for which the area-normalized induced metrics converge to $\gamma_\infty$.
This can be thought of as a frame of reference relative to which the Bondi energy $E^{\gamma_\infty}_B$ and the Bondi linear momentum $\vec{P}^{\gamma_\infty}_B$ are measured.

Thus, to properly extract the Bondi energy, we consider a family $\fol{\Sigma}{y}$, $y \in [1, \infty)$, of spherical cuts in $\mc{N}$.
Let $\mind_{ \fol{\Sigma}{y} }$ denote the metric on $\fol{\Sigma}{y}$ induced by the spacetime metric $\g$, and let $r_{ \fol{\Sigma}{y} }$ denote the area radius of $( \fol{\Sigma}{y}, \mind_{ \fol{\Sigma}{y} } )$,
\[ r_{ \fol{\Sigma}{y} } := \sqrt{ \frac{{\rm Area} ( \fol{\Sigma}{y}, \mind_{ \fol{\Sigma}{y} } )}{ 4 \pi } } \text{.} \]
Suppose the corresponding Gauss curvatures $\mc{K} ( \fol{\Sigma}{y}, \mind_{ \fol{\Sigma}{y} } )$ satisfy
\begin{equation} \label{lim} \lim_{y \nearrow \infty} \mc{K} ( \fol{\Sigma}{y}, r_{ \fol{\Sigma}{y} }^{-2} \mind_{ \fol{\Sigma}{y} } ) = \lim_{y \nearrow \infty} r_{ \fol{\Sigma}{y} }^2 \mc{K} ( \fol{\Sigma}{y}, \mind_{ \fol{\Sigma}{y} } ) = 1 \text{,} \end{equation}
so that the area-normalized metrics $r_{ \fol{\Sigma}{y} }^{-2} \mind_{ \fol{\Sigma}{y} }$ converge to a round metric $\gamma_\infty$.
Then, according to \cite{pen_rind:spinor, sauter:penrose}, the Bondi energy $E^{\gamma_\infty}_B$ associated with $(\mc{S}_\infty, \gamma_\infty)$ corresponds to the limit of the Hawking masses of the $\fol{\Sigma}{y}$'s,
\begin{equation} \label{Bondi.def} E^{\gamma_\infty}_{\rm B} = \lim_{y \nearrow \infty} m_{\rm H} ( \fol{\Sigma}{y} ) \text{,} \end{equation}
where the Hawking mass of $\fol{\Sigma}{y}$ is defined as
\begin{equation} \label{hawk.mass} m_{\rm H} (\fol{\Sigma}{y}) = \frac{ r_{ \fol{\Sigma}{y} } }{2} \left( 1 + \frac{1}{16 \pi} \int_{ \fol{\Sigma}{y} } \trase{\chi} \cdot \trase\ul{\chi} \cdot d V_{ \mind_{ \fol{\Sigma}{y} } } \right) \text{.} \end{equation}
The functions $\trase \chi, \trase \ul{\chi}$ are the expansions of $\fol{\Sigma}{y}$ relative to two future-directed orthogonal null vector fields over $\fol{\Sigma}{y}$---see \eqref{sff} below.

\begin{remark}
In view of \eqref{Bondi.def}, a bound on the limit of the Hawking masses for a family of spherical cuts satisfying \eqref{lim} would yield a bound on the Bondi mass (which is invariantly defined) and Bondi linear momentum (defined relative to the foliation).
\end{remark}
 
\subsubsection{Angular Momentum}

As present, there exists no universally accepted notion of angular momentum for sections of null infinity; examples of proposed definitions include the works of Rizzi \cite{riz:ang_mom} and Moreschi \cite{mores:intr_angmom}.

In the context of our setting, it seems that to control any such reasonable notion of angular momentum on $(\mc{S}_\infty, \gamma_\infty)$, a quantity that must be controlled is the torsion $\zeta$ of the $\Sigma_y$'s; see \eqref{torsion} for the precise definition.
One consequence of our main result is that we can obtain bounds for the quantities
\begin{equation} \label{ang.mom} \mf{A}^{\gamma_\infty} (X) = \int_{ \mc{S}_\infty } \mc{Z}_a X^a \cdot d V_{ \gamma_\infty } \text{,} \end{equation}
where $\mc{Z}$ corresponds to the renormalization
\begin{equation} \label{lim.Z} \mc{Z} = \lim_{y \nearrow \infty} r_{ \fol{\Sigma}{y} } \cdot \zeta ( \fol{\Sigma}{y} ) \text{,} \end{equation}
of $\zeta$ on $\mc{S}_\infty$, and where $X$ is a rotational Killing vector field on $(\mc{S}_\infty, \gamma_\infty)$.

\subsubsection{Rate of Energy Loss}

In the specific settings discussed thus far, in particular that of \cite{chr_kl:stb_mink}, the Bondi energy is in fact a non-increasing function along $\mc{I}^+$.
More specifically, consider a foliation of $\mc{I}^+$ by a $1$-parameter family of $2$-spheres, $(\mc{S}_{\infty, u}, \gamma_{\infty,u})$, with $u$ increasing in the future direction.
These correspond to a local foliation of $M$ near $\mc{I}^+$ by a family of outgoing null cones $\mc{N}_u$.
Then, the Bondi energies of the sections $(\mc{S}_{\infty, u}, \gamma_{\infty, u})$ evolve according to the \emph{energy loss formula}:
\begin{equation} \label{mass.loss} \dot{E}^{\gamma_\infty}_{\rm B} := \frac{d}{d u} E_{\rm B}^{\gamma_\infty} (u) = - \frac{1}{8 \pi} \int_{ \mc{S}_{\infty, u} } | \hat{\Xi} |^2_{ \gamma_{\infty, u} } dV_{ \gamma_{\infty, u} } \text{.} \end{equation} 

$\hat{\Xi}$ is the ($\gamma_{\infty, u}$-)traceless part of the symmetric 2-tensor $\Xi$ over $(\mc{S}_{\infty, u}, \gamma_{\infty, u})$, and $\Xi$ is defined, for a family of asymptotically round spherical cuts $\fol{\Sigma}{y}_u$ of $\mc{N}_u$, by
\footnote{Again, the limit refers to components of $\ul{\chi}$ with respect to transported systems of coordinates; see Section \ref{SphCuts} and Definition \ref{def.conv_inf}.}
\[ \Xi = \lim_{y \nearrow \infty} r_{ \fol{\Sigma}{y} }^{-1} \cdot \ul{\chi} ( \fol{\Sigma}{y}_u ) \text{.} \]
Here, $\ul{\chi}$ denotes the  second fundamental forms of the $\fol{\Sigma}{y}_u$'s in the incoming null direction.
\footnote{$\ul{\chi}$ is defined more precisely in \eqref{sff} and in \cite{alex_shao:nc_inf}.}
In our main theorem, we will control on a null cone $\mc{N}$ the right-hand side of \eqref{mass.loss}, which describes the \emph{rate of energy loss} across $\mc{N}$.

\subsection{The Results} \label{intro:thm}

The next task is to describe more precisely the results that we wish to prove.
Throughout this discussion, we assume $(M, \g)$ to be an arbitrary Einstein-vacuum spacetime, and $\mc{N} \subseteq M$ an infinite smooth null hypersurface beginning from a Riemannian $2$-sphere $\mc{S} \subseteq M$.

\subsubsection{Geodesic Foliations}

While there are many natural foliations of null cones of $\mc{N}$ by $2$-spheres, in this paper, we will only be considering geodesic foliations. 
These correspond to arc-length parametrizations of the null geodesics that rule $\mc{N}$.

Consider any null vector field $L$ which is both tangent to $\mc{N}$ and parallel (i.e., $D_L L=0$, where $D$ is the spacetime Levi-Civita connection).
\footnote{Note that $L$ on $\mc{N}$ is uniquely determined by its values on the initial sphere $\mc{S}$.}
We let $s$ be the arc-length parameter along the the null geodesics that are the integral curves of $L$, normalized such that $s = 1$ on $\mc{S}$.
Our foliation is then by the level sets $\mc{S}_s$ of $s$.
Furthermore, we let $\ul{L}$ denote the null vector field on $\mc{N}$ that is conjugate to $L$, that is, satisfying the conditions $\ul{L} \perp \mc{S}_s$ and $\g (L, \ul{L}) \equiv -2$.

\begin{remark}
Note that different choices of $L$ (and hence $s$) yield different affine parameters and hence different foliations.
We will be making use of this freedom in choosing geodesic foliations extensively in later sections.
\end{remark}

Next, we recall the quantities that define the intrinsic and extrinsic geometry of $\mc{N}$, in terms of the above foliation $\mc{S}_s$, $s \geq 1$:
\begin{itemize}
\item Let $\mind_s$ be the induced metric on $\mc{S}_s$, i.e., the restriction of $\g$ to $\mc{S}_s$.

\item Let $\nasla$ be the Levi-Civita connection for $\mind_s$.

\item The \emph{null second fundamental forms} $\chi$ and $\ul{\chi}$ are defined
\begin{equation} \label{sff} \chi (X, Y) = \g (D_X L, Y) \text{,} \qquad \ul{\chi} (X, Y) = \g (D_X \ul{L}, Y) \text{,} \end{equation}
where $X$ and $Y$ are arbitrary vector fields tangent to the $\mc{S}_s$'s.

\item The \emph{torsion} $\zeta$ is defined, for $X$ as before, by
\begin{equation} \label{torsion} \zeta (X) = \frac{1}{2} g (D_X L, \ul{L}) \text{.} \end{equation}
\end{itemize}
We will be decomposing the 2-tensors $\chi$ and $\ul{\chi}$ into their trace ($\trase \chi$, $\trase \ul{\chi}$) and traceless ($\hat{\chi}$, $\hat{\ul{\chi}}$) parts. 
Here, $\trase$ refers the trace operator with respect to the metrics $\mind_s$.
Thus, $\trase \chi$ represents the \emph{expansions} of the $\mc{S}_s$'s, and $\hat{\chi}$ their \emph{shears}.

We also recall the independent components of the Riemann/Weyl curvature.
Letting $R$ denote the Riemann curvature tensor associated with $\g$, we recall that $R$ is fully determined by the following components:
\begin{align}
\label{list} \alpha (X, Y) = R (L, X, L, Y) \text{,} &\qquad \ul{\alpha} (X, Y) = R (\ul{L}, X, \ul{L}, Y) \text{,} \\
\notag \beta (X) = \frac{1}{2} R (L, X, L, \ul{L}) \text{,} &\qquad \ul{\beta} (X) = \frac{1}{2} R (\ul{L}, X, \ul{L}, L) \text{,} \\
\notag \rho = \frac{1}{4} R (L, \ul{L}, L, \ul{L}) \text{,} &\qquad \sigma = \frac{1}{4} {}^\star R (L, \ul{L}, L, \ul{L}) \text{,}
\end{align}
where $X$ and $Y$ are as before, and where ${}^\star R$ denotes the left Hodge dual of $R$.

Finally, we recall the \emph{mass aspect function}, $\mu$, on the $\mc{S}_s$'s, cf. \cite{chr_kl:stb_mink}.
This is a scalar-valued function on the $\mc{S}_s$'s, defined by
\begin{equation} \label{maf} \mu := - \mind^{ab} \nasla_a \zeta_b - \rho + \frac{1}{2} \mind^{ac} \mind^{bd} \hat{\chi}_{ab} \ul{\hat{\chi}}_{cd} \text{.} \end{equation}
This quantity is closely related to the Hawking masses of the $\mc{S}_s$'s, via the formula
\[ \int_{\mc{S}_s} \mu \cdot dV_{ \mind_s } = \frac{8 \pi}{ r_{ \fol{\Sigma}{y} } } \cdot m_{\rm H} (\mc{S}_s) \text{.} \]

\subsubsection{Analysis of Infinite Null Cones}

To properly state our main result, we must quantitatively capture the deviation of our null cone $\mc{N}$ from a corresponding Schwarzschild cone, $\mc{N}_{\rm S}$.
The first measure of this deviation translates to a weighted $L^2$-norm over $\mc{N}$ of \emph{suitably weighted} differences between the components in \eqref{list} (excluding $\ul{\alpha}$) and their corresponding Schwarzschild values.
With this intuition in mind, we define the \emph{weighted curvature flux deviation} of $\mc{N}$ from $\mc{N}_{\rm S}$ by
\begin{equation} \label{flux_dev} \mc{F} = \int_{ \mc{N} } \left[ | s^2 \alpha |_{\mind_s}^2 + | s^2 \beta |_{\mind_s}^2 + \left| s \left( \rho + \frac{2 m}{s^3} \right) \right|_{\mind_s}^2 + | s \sigma |_{\mind_s}^2 + | \ul{\beta} |_{\mind_s}^2 \right] dV_{\tilde{\g}} \text{,} \end{equation}
where $| \dots |_{\mind_s}$ denotes the pointwise tensor norm with respect to the metrics $\mind_s$.

The other meausre  of the deviation between $\mc{N}$ and $\mc{N}_{\rm S}$ involves the geometries of their initial spheres, $\mc{S}$ and $\mc{S}_{\rm S}$;
this is captured in  the differences between the connection coefficients \eqref{sff}, \eqref{torsion} on $\mc{S}$ and their corresponding Schwarzschild values.
Assuming for convenience that $\mc{S}$ has unit area radius, then the above translates to measuring the following quantities on $\mc{S} = \mc{S}_1$:
\footnote{See \cite[Sect. 4.3]{alex_shao:nc_inf}; in particular, we assume $r_{\mc{S}} = r_{\mc{S}_{\rm S}} = 1$.}
\[ \chi - \mind_1 \text{,} \qquad \zeta \text{,} \qquad \ul{\chi} + (1 - 2 m_{\rm S}) \mind_1 \text{,} \qquad \mu - 2 m_{\rm S} \text{.} \]
Moreover, the norms with respect to which we measure these quantities must be compatible with the $L^2$-curvature regularity level.
The specific norms we use are listed in Theorem \ref{the.thm} below and justified in detail in \cite{alex_shao:nc_inf, kl_rod:cg}.

The main result of this paper states that, if the deviation between $\mc{N}$ and $\mc{N}_{\rm S}$, as described above, is sufficiently small, then the Bondi energy of $\mc{N}$, expressed via \eqref{Bondi.def} with respect to a family of spherical cuts $\fol{\Sigma}{y}$ satisfying \eqref{lim}, will be comparably close to the Schwarzschild mass $m_{\rm S}$ associated with $\mc{N}_{\rm S}$.
Similar estimates hold for the rate of energy loss; see Theorem \ref{the.thm} below.

To obtain these conclusions, we rely crucially on the main results of our previous article \cite{alex_shao:nc_inf}, which can roughly summarized as follows.
\emph{Assume that the deviation between $\mc{N}$, with a given geodesic foliation, and $\mc{N}_{\rm S}$ are small}, in the sense that:
\begin{itemize}
\item The curvature flux deviation $\mc{F}$ defined in \eqref{flux_dev} is sufficiently small.

\item The deviation quantities $\chi - \mind_1$, $\zeta$, $\ul{\chi} + (1 - 2 m_{\rm S}) \mind_1$, and $\mu - 2 m_{\rm S}$, corresponding to the connection coefficients and the mass aspect function, are sufficiently small on the initial sphere $\mc{S}$ of $\mc{N}$ in the appropriate norms.
\footnote{See \eqref{initial.smooth} below for the precise norms that are required to be small.}
\end{itemize}
\emph{Then, the geometry of $\mc{N}$ remains uniformly close to that of $\mc{N}_{\rm S}$:}
\begin{itemize}
\item The deviations of the connection coefficients $\chi$, $\ul{\chi}$, and $\zeta$ from their values on $\mc{N}_{\rm S}$ remain uniformly small (in suitable weighted norms) on all of $\mc{N}$.

\item Similarly, the deviation of the mass aspect function $\mu$ from its value $2 m_{\rm S} s^{-3}$ on $\mc{N}_{\rm S}$ also remains uniformly small on all of $\mc{N}$.
\end{itemize}
The above comprise the contents of \cite[Thm.~1.1, Thm.~5.3]{alex_shao:nc_inf}.
A precise statement of these results is introduced later in Theorem \ref{thm.nc_renorm}.

For our present paper, the most important conclusion from \cite{alex_shao:nc_inf} is that \emph{the uniform estimates for $\chi$, $\ul{\chi}$, $\zeta$, and $\mu$ imply that 
suitable renormalizations of these  quantities have limits at infinity.}
To be more precise, in \cite[Thm.~1.2, Cor.~5.2]{alex_shao:nc_inf} we derived:
\begin{itemize}
\item The renormalized metrics $s^{-2} \mind_s$ converge to a limiting metric $\gamma_\infty$ as $s \nearrow \infty$.

\item There exist limits for $\chi$, $\ul{\chi}$, $\zeta$, and $\mu$, in the appropriate spaces and with the appropriate weights, as $s \nearrow \infty$.
In particular, the limits as $s \nearrow \infty$ of
\begin{equation} \label{limits} \int_{ \mc{S}_s } s \mu \cdot d V_{ \mind_s } \text{,} \qquad \int_{ \mc{S}_s } s^{-1} \zeta (X) \cdot d V_{ \mind_s } \text{,} \qquad \int_{ \mc{S}_s } | \hat{\ul{\chi}} |_{ \mind_s }^2 \cdot d V_{ \mind_s } \text{,} \end{equation}
where $X$ is a rotational Killing field on $\Sph^2$, exist and are uniformly small.
\end{itemize}
The specific limits of \eqref{limits} are directly related to the limits \eqref{Bondi.def}, \eqref{lim.Z}, and \eqref{mass.loss} for the physical quantities of interest.

These results in particular imply that under the above assumptions, the limit of the Hawking masses $m_{\rm H} (\mc{S}_s)$ as $s \nearrow \infty$ exists and is close to $m_{\rm S}$.
However, this limit does \emph{not} yield a bound on the Bondi energy of $\mc{N}$ (nor for the other quantities of interest), since the spheres $(\mc{S}_s, \mind_s)$ need not become asymptotically round, in the sense of \eqref{lim}.
The main novel challenge of this paper, then, is to extract the Bondi energy and the other physical quantities in a manner such that they can be controlled using the results of \cite{alex_shao:nc_inf} described above.

To accomplish this, we note an additional degree of freedom in our setup: \emph{the results from \cite{alex_shao:nc_inf} outlined above hold for any geodesic foliation of $\mc{N}$ for which the deviation from $\mc{N}_{\rm S}$ is sufficiently small.}
In particular, {\it any other} geodesic foliation of $\mc{N}$ that is ``sufficiently close" to the current one would satisfy the small deviation condition.
The idea, then, is to \emph{find}  a {\it new}  geodesic foliation that is ``nearby'' the original one and moreover fulfils the asymptotic roundness property  \eqref{lim}.
We will then be able to apply the main result of \cite{alex_shao:nc_inf} to this new geodesic foliation 
to control the physical quantities of interest.
\footnote{In practice, we obtain the asymptotically round family of spheres using not a single change of foliation, but rather a one-parameter family of new foliations; see the discussion in Section \ref{sect:intro.outline}.}

\subsubsection{The Main Theorem}
 
The main theorem of this paper is the following: 

\begin{theorem} \label{the.thm}
Let $0 \leq m_{\rm S} < 1/2$, and let $\mc{N} \subset M$ be an infinite smooth null hypersurface emanating from $\mc{S}$, with ${\rm Area} (\mc{S}) = 4 \pi$.
\footnote{In particular, if $M$ is the Schwarzschild spacetime with mass $m_{\rm S}$, and if $\mc{N}$ is a canonical shear-free null cone in $M$, then the initial sphere $\mc{S}$ would lie in the outer region.}
Also, fix a geodesic foliation of $\mc{N}$, with associated affine parameter $s$.
Assume the following hold on $\mc{N}$:
\begin{itemize}
\item Curvature flux deviation bound on $\mc{N}$:
\begin{equation} \label{CFsmall} \int_{ \mc{N} } \left[ | s^2 \alpha |_{\mind_s}^2 + | s^2 \beta |_{\mind_s}^2 + \left| s \left( \rho + \frac{2 m_{\rm S}}{s^3} \right) \right|_{\mind_s}^2 + | s \sigma |_{\mind_s}^2 + | \ul{\beta} |_{\mind_s}^2 \right] dV_{\tilde{\g}} \leq \Gamma^2 \text{.} \end{equation}

\item Initial value bounds on $\mc{S}$:
\footnote{The $\sorm{L}^q$-, $\sorm{H}^a$-, and $\sorm{B}^a$-norms refer to Lebesgue, Sobolev, and Besov norms on $(\mc{S}, \mind_1)$.
For more precise definitions of these norms, see Section \ref{sect:prelim.norms}, as well as \cite{alex_shao:nc_inf, shao:stt}.}
\begin{align}
\label{initial.smooth} \| \trase \chi - 2 \|_{ \sorm{L}^\infty_x (\mc{S}, \mind_1) } + \| \chi - \mind_1 \|_{ \sorm{H}^{1/2}_x (\mc{S}, \mind_1) } + \| \zeta \|_{ \sorm{H}^{1/2}_x (\mc{S}, \mind_1) } &\leq \Gamma \text{,} \\
\notag \| \ul{\chi} + (1 - 2 m_{\rm S}) \mind_1 \|_{ \sorm{B}^0_x (\mc{S}, \mind_1) } + \| \nasla ( \trase \chi ) \|_{ \sorm{B}^0_x (\mc{S}, \mind_1) } + \| \mu - 2 m_{\rm S} \|_{ \sorm{B}^0_x (\mc{S}, \mind_1) } &\leq \Gamma \text{.}
\end{align}
\end{itemize}
Then, if $\Gamma$ is sufficiently small, there is a family of spherical cuts $\fol{\Sigma}{y}$, $y \in [1, \infty)$, of $\mc{N}$ going to infinity, with corresponding induced metrics $\mind_{ \fol{\Sigma}{y} }$ and areas
\[ {\rm Area} (\fol{\Sigma}{y}, \mind_{ \fol{\Sigma}{y} }) = 4 \pi r^2_{ \fol{\Sigma}{y} } \text{,} \]
such that:
\begin{itemize}
\item The $(\fol{\Sigma}{y}, r_{ \fol{\Sigma}{y} }^{-2} \mind_{ \fol{\Sigma}{y} })$'s, that is, the $\fol{\Sigma}{y}$'s with the area-normalized induced metrics, become round in a weak sense.
More specifically,
\begin{equation} \label{rounding} \lim_{y \nearrow \infty} \| r_{ \fol{\Sigma}{y} }^2 \mc{K} ( \fol{\Sigma}{y}, \mind_{ \fol{\Sigma}{y} } ) - 1 \|_{ \sorm{H}^{-1/2} (\fol{\Sigma}{y}, r_{ \fol{\Sigma}{y} }^{-2} \mind_{ \fol{\Sigma}{y} } ) } = 0 \text{.} \end{equation}

\item The Bondi energy, defined in \eqref{Bondi.def} in terms of the $\fol{\Sigma}{y}$'s, exists, and
\begin{equation} \label{conclns.bondi} \lim_{y \nearrow \infty} | m_H ( \fol{\Sigma}{y} ) - m_{\rm S} | \lesssim \Gamma \text{.} \end{equation}

\item The rate of energy loss, defined in \eqref{mass.loss} in terms of the $\fol{\Sigma}{y}$'s, exists, and
\begin{equation} \label{conclns.loss} | \dot{E}_{\rm B}^{\gamma_\infty} | \lesssim \Gamma \text{.} \end{equation}
\end{itemize}
\end{theorem}

\begin{remark}
We note that the requirement of Theorem \ref{the.thm} with $m_{\rm S} = 0$ will hold for truncated null cones in the perturbations of the Minkowski spacetime in \cite{chr_kl:stb_mink}.
This follows from the decay of the curvature components stated in the last chapter of \cite{chr_kl:stb_mink}.
For $m_{\rm S} > 0$, Theorem \ref{the.thm} is designed to apply to truncated null surfaces in perturbations of Schwarzschild spacetimes.
In all settings, these perturbations are required to be small at the $L^2$-curvature level.
\end{remark}

\begin{remark}
Although Theorem \ref{the.thm} deals only with the specific case $m_{\rm S} < 1/2$ and ${\rm Area} (\mc{S}) = 4 \pi$, it can be extended to general cases of arbitrary mass and initial area due to the dilation invariance of the Einstein-vacuum equations.
To see how the assumptions and norms transform under such rescalings, see \cite[Thm. 5.1]{alex_shao:nc_inf}.
\end{remark}

\begin{remark}
That the convergence of the Gauss curvatures in \eqref{rounding} to $1$ is in weak $\sorm{H}^{-1/2}$-norms is a consequence of the low regularity of our setting.
One cannot expect a stronger norm of convergence without additional assumptions, in view of the Sobolev trace theorem.
If one were to assume extra regularity for $\mc{N}$ (for example, analogous control for derivatives of $R$), then the convergence for the Gauss curvatures would be in correspondingly smooth norms.
\end{remark}

\begin{remark}
In Theorem \ref{the.thm} and in \cite{alex_shao:nc_inf}, we elected to work with shear-free null hypersurfaces in Schwarzschild spacetimes primarily because the values of the connection and curvature quantities on these hypersurfaces are explicitly given and well-known.
It is possible that analogous results can be proved for null cones near Kerr spacetimes, or for some other null hypersurfaces near Schwarzschild spacetimes.

On the other hand, Theorem \ref{the.thm} includes null hypersurfaces in Kerr spacetimes with small angular momentum, which admit a smooth conformal compactification up to $\mc{I}^+$ and (in this compactified setting) can be continuously deformed to Schwarzschild solutions in the $\mc{C}^2$-norm.
Since the weights in Theorem \ref{the.thm} are weaker than required by the Sachs peeling (which holds for the Kerr solutions), it follows that Theorem \ref{the.thm} applies to these slowly rotating Kerr spacetimes as well as to their perturbations (in the weak norms of Theorem \ref{the.thm}).
\end{remark}

\begin{remark}
In view of recent works by Luk and Rodnianski, \cite{luk_rod:grav_wave, luk_rod:grav_wave_inter}, on gravitational waves, one may ask whether a version of Theorem \ref{the.thm} can be established without assumptions on the curvature component $\alpha$ in \eqref{CFsmall}.
To our knowledge, this condition on $\alpha$ seem to be necessary, since the lack of regularity in our setting forces us to utilize \emph{all} of the structure equations available for $\mc{N}$.
\end{remark}

\begin{remark}
The proof of Theorem \ref{the.thm} also yields bounds for the torsion $\zeta$ at infinity.
More specifically, we can obtain a bound for the quantity \eqref{ang.mom} of the form
\begin{equation} \label{conclns.ang} | \mf{A}^{\gamma_\infty} (X) | \lesssim \Gamma \text{.} \end{equation}
\end{remark}
 
\subsection{Ideas of the Proof} \label{sect:intro.outline}

We conclude the introduction with a brief and informal discussion of the main ideas of the proof of Theorem \ref{the.thm}.

\subsubsection{The Renormalized Setting}

The following intuitions regarding $\mc{N}$ arise from the assumed closeness of $\mc{N}$ to the Schwarzschild null cone $\mc{N}_{\rm S}$:
\begin{itemize}
\item The chosen affine parameter $s$, from the assumptions of Theorem \ref{the.thm}, should approximate the area radii $r (s)$ of the level spheres $\mc{S}_s$.

\item The rescaled metrics $s^{-2} \mind_s$ should be close to the round metric.
\end{itemize}
These heuristics suggest it may be more natural to work with the metrics $s^{-2} \mind_s$, rather than $\mind_s$ itself.
This was an essential idea in the analysis throughout \cite{alex_shao:nc_inf} and will play the same fundamental role in this paper.
Furthermore, these intuitions were rigorously justified within the main results of \cite{alex_shao:nc_inf}.

For even more convenience, we make an additional change of parameter: from the affine parameter $s \in [1, \infty)$ to a finite parameter $t = 1 - s^{-1} \in [0, 1)$.
Combined, the above two transformations result in the so-called \emph{renormalized system}, on which all of our serious analysis will take place.
Because of the near-uniform geometries of the level spheres in this setting, it is much easier to consider limits in terms of the renormalized picture.
In particular, we will generate the limiting metric $\gamma_\infty$ at infinity as limits of these rescaled metrics $s^{-2} \mind_s$.

\subsubsection{Changes of Foliations}

As noted before, different choices of the affine vector field $L$ lead to different affine parameters and geodesic foliations.
Given our original geodesic foliation, we can rescale our null tangent vector field $L$ by
\[ L^\prime = e^v L \text{,} \]
where $v$ is a smooth function on $\mc{N}$ that is constant on the null geodesics that generate $\mc{N}$.
Notice in particular that $L^\prime$ is once again parallel.
Thus, we can consider an affine parameter $s^\prime$ associated with $L^\prime$ (again with the normalization $s^\prime = 1$ on $\mc{S}$).
One can proceed from here to compute how the metrics, Ricci coefficients, and curvature components transform under this change of geodesic foliation.
The results are discussed in detail in Section \ref{sect:prelim.cfol} and Appendix \ref{sect:cfol}.

Since our goal is to construct asymptotically round spherical cuts of $\mc{N}$, our greatest interest lies in how the limiting metric
\[ \gamma_\infty = \lim_{s \nearrow \infty} s^{-2} \mind_s \]
transforms.
Without delving into details (these are presented in Section \ref{sect:prelim.cfol}), we can guess the result via a heuristic argument.
From the results in \cite{alex_shao:nc_inf}, we know $\mc{N}$ is asymptotic at infinity to the cone $\mc{N}_\infty \simeq [1, \infty) \times \Sph^2$, with the degenerate metric
\[ \g_\infty = 0 \cdot ds^2 + s^2 \gamma_\infty \text{.} \]
Furthermore, we observe that for a spherical cut of $\mc{N}_\infty$ given by
\[ \omega \in \Sph^2 \mapsto ( \lambda (\omega), \omega ) \text{,} \]
the metric induced by $\g_\infty$ on this cut is precisely $\lambda^2 \gamma_{\infty}$.

This suggests that \emph{a change of geodesic foliation results in a corresponding conformal transformation of the limiting metric at infinity.}
\footnote{In fact, this equivalence between conformal transformations and cuts of conical pseudo-Riemannian metrics has been used in many contexts; see, for instance, \cite{fef_gra:conf_inv}.}
Consequently, the \emph{problem of generating asymptotically round spherical cuts should be closely related to the uniformization problem for the limiting sphere $\mc{S}_\infty$ at infinity.}
The rigorous implementation of this idea will be presented in Section \ref{sect:distortion}.

In practice, the asymptotically round cuts will be obtained by constructing a special $1$-parameter family of changes of foliations, characterized by the functions $\fol{v}{y}$, $y \in [0, 1)$.
The $\fol{v}{y}$'s will be uniformly small, in particular so that the new foliations will also be controlled via the results of \cite{alex_shao:nc_inf}.
Moreover, the $\fol{v}{y}$'s can be chosen to converge in the appropriate norms to a limiting function $\fol{v}{1}$ that satisfies a uniformization-type equation at infinity.
Finally, we construct the cuts $\fol{\Sigma}{y} := \{ \fol{s}{y} = (1 - y)^{-1} \} \subset \mc{N}$ from these $\fol{v}{y}$-foliations and prove that they are asymptotically round.
This construction of the cuts $\fol{\Sigma}{y}$ is the crux of this paper.

We remark that certain extra layers of complexity required by our proof are  due to the low regularity of our setting.
If one assumed a priori the limiting function $\fol{v}{1}$ is smooth, then one could bypass entirely the approximating $\fol{v}{y}$'s.
Indeed, in this case, the $\fol{v}{1}$-foliation of $\mc{N}$ would be the desired asymptotically round spheres.
However, since $\fol{v}{1}$ is only in $H^2$ in our setting and the convergence of the Gauss curvatures is in $H^{-1/2}$, 
this approach would lead to undesirable technical difficulties.

\subsubsection{Limits at Infinity}

With the family $\fol{\Sigma}{y}$ in hand, it remains to show that the limits from \eqref{conclns.bondi}-\eqref{conclns.loss} exist and are sufficiently controlled.
While the results of \cite{alex_shao:nc_inf} suggest that this is indeed the case, they unfortunately do not directly apply here, since we are now working with an infinite family of foliations of $\mc{N}$.
Thus, in order to generate the desired limits, we have the additional task of comparing quantities on different foliations with each other.

While this adds its share of technical baggage to the process (for example, it is a priori unclear how to compare tensor fields living in different foliations), the issues are not fundamental.
The problem of comparing objects in separate foliations can be resolved by a natural identification of frames in these foliations; see Section \ref{sect:prelim.cfol}.
Once this convention is clear, the ensuing analysis resembles that found in \cite{alex_shao:nc_inf} (although the estimates are messier due to the changes of foliations).
The bulk of this argument will be carried out in Section \ref{sect:convergence}.

\begin{remark}
We note that this family $\fol{\Sigma}{y}$ is not the unique one with the asymptotic roundness property.
Indeed, any other construction obtained from functions $v^y$ solving the appropriate uniformization problem at infinity would also suffice and would yield another bound on the Bondi mass.
While it is not clear which refoliation results in the best bound, we do note that, up to the universal constant implied by the ``$\lesssim$", the bound \eqref{conclns.bondi} is in fact optimal in terms of the powers of $\Gamma$, in view of nearly Schwarzschild solutions with mass $m \neq m_{\rm S}$.
\end{remark}

\begin{ackn}
The first author was supported by NSERC grants 488916 and 489103, as well as a Sloan Fellowship.
The authors also wish to thank Mihalis Dafermos, Sergiu Klainerman, and Lydia Bieri for helpful discussions and insights.
\end{ackn}

\section{Preliminaries} \label{sect:prelim}

We discuss various preliminary notions needed to prove Theorem \ref{the.thm}:
\begin{itemize}
\item A brief discussion of our basic setting of analysis.

\item Changes of (geodesic) foliations of null cones.

\item The main theorem of \cite{alex_shao:nc_inf}, which uniformly controls the geometry of an infinite null cone by its curvature flux (with respect to a geodesic foliation).
\end{itemize}
As before, we assume $(M, g)$ to be a vacuum spacetime and $\mc{N} \subset M$ to be an infinite null cone emanating from a Riemannian $2$-sphere $\mc{S} \subset M$.

\subsection{Spherical Foliations} \label{sect:prelim.fol}

We briefly summarize the basic formalism, developed in \cite{alex_shao:nc_inf, shao:stt}, that we will use in this paper; see \cite[Sect. 3]{alex_shao:nc_inf} and \cite{shao:stt} for more detailed discussions.
The abstract setting is a one-parameter spherical foliation of $\mc{N}$,
\[ \mc{N} = \bigcup_{ \tau_- \leq \tau < \tau_+ } \Sigma_\tau \text{,} \qquad \Sigma_\tau \simeq \Sph^2 \text{,} \]
with $\Sigma_{\tau_-} = \mc{S}$.
The basic objects of analysis are \emph{horizontal tensor fields}, i.e., tensor fields on $\mc{N}$ which are everywhere tangent to the $\Sigma_\tau$'s.

On each sphere $\Sigma_\tau$, we impose a Riemannian metric $\eta_\tau$, so that the $\eta_\tau$'s vary smoothly with respect to $\tau$.
Let $\eta$ denote the resulting \emph{horizontal metric} on $\mc{N}$, representing the aggregation of all the $\eta_\tau$'s on $\mc{N}$.
\footnote{In particular, $\eta$ is a horizontal tensor field.}
Similarly, the volume forms $\upsilon_\tau$ associated with the $\eta_\tau$'s can be combined into the \emph{horizontal volume form} $\upsilon$ for $\eta$.
Combined, these objects form what we refer to as a \emph{horizontal covariant system},
\[ (\mc{N}, \eta) = \bigcup_{ \tau_- \leq \tau < \tau_+ } ( \Sigma_\tau, \eta_\tau ) \text{.} \]

In addition, let $\nabla$ denote the usual Levi-Civita connections for the $\eta_\tau$'s, which represent covariant derivatives of horizontal tensor fields in directions tangent to the $\Sigma_\tau$'s.
We can also define an analogous covariant derivative operator $\nabla_\tau$ in the remaining $\tau$-direction.
First, given a horizontal tensor field field $\Psi$, we let $\mf{L}_\tau \Psi$ denote the Lie derivative of $\Psi$ in the direction $d/d\tau$, along the null generators of $\mc{N}$.
\footnote{See \cite[Sect. 4.1]{shao:stt} for a more detailed characterization of $\mf{L}_\tau \Psi$.}
Of particular interest is the \emph{second fundamental form},
\[ k = \frac{1}{2} \mf{L}_\tau \eta \text{,} \]
which indicates how the geometries of the $\Sigma_\tau$'s evolve.
We then define
\[ \nabla_\tau \Psi_{u_1 \dots u_l}^{v_1 \dots v_r} = \mf{L}_\tau \Psi_{u_1 \dots u_l}^{v_1 \dots v_r} - \sum_{i = 1}^l \eta^{cd} k_{u_i c} \Psi_{u_1 \hat{d}_i u_l}^{v_1 \dots v_r} + \sum_{j = 1}^r \eta^{c v_j} k_{cd} \Psi_{u_1 \dots u_l}^{v_1 \hat{d}_j v_r} \text{,} \]
where the notation $u_1 \hat{d}_i u_l$ means $u_1 \dots u_l$, but with $u_i$ replaced by $d$.

In particular, one can show that both $\nabla$ and $\nabla_\tau$ annihilate the metric $\eta$, its dual $\eta^{-1}$, and the associated volume form $\upsilon$.
Moreover, we note the following commutation formula, which played a crucial role in many of the estimates in \cite{shao:stt}:
\begin{align}
\label{eq.comm} [ \nabla_\tau, \nabla_a ] \Psi_{u_1 \dots u_l}^{v_1 \dots v_r} &= - \eta^{cd} k_{ac} \nabla_d \Psi_{u_1 \dots u_l}^{v_1 \dots v_r} - \sum_{i = 1}^l \eta^{cd} ( \nabla_{u_i} k_{ac} - \nabla_c k_{a u_i} ) \Psi_{u_1 \hat{d}_i u_l}^{v_1 \dots v_r} \\
\notag &\qquad + \sum_{j = 1}^r \eta^{c v_j} ( \nabla_d k_{ac} - \nabla_c k_{ad} ) \Psi_{u_1 \dots u_l}^{v_1 \hat{d}_j v_r} \text{.}
\end{align}

\subsubsection{The Physical Setting}

We now formulate the geometry of $\mc{N}$, as a null hypersurface of $(M, \g)$, in terms of the above general framework.
Assume a \emph{geodesic foliation} of $\mc{N}$, via an \emph{affine parameter} $s: \mc{N} \rightarrow \R$, that is, $s$ acts as an affine parameter of every null geodesic generator of $\mc{N}$.
In addition, we normalize $s$ so that the initial sphere $\mc{S}$ is precisely the level set $s = 1$.
Using the symbol $\mc{S}_\varsigma$ to denote the level sphere $s = \varsigma$, we can describe our geodesic foliation as $\mc{N} = \bigcup_{s \geq 1} \mc{S}_s$.

On each $\mc{S}_s$, we have the (Riemannian) metric $\mind_s$ induced by the spacetime metric $\g$, as well as the Levi-Civita connection $\nasla$ associated with $\mind_s$.
The aggregation $\mind$ of all the $\mind_s$'s defines the horizontal metric for this system.
To maintain consistency with \cite{alex_shao:nc_inf}, we will use the symbol $\vind$ to denote the horizontal volume form associated with $\mind$.
We will refer to the resulting horizontal covariant system,
\[ (\mc{N}, \mind) \simeq \bigcup_{s \geq 1} (\mc{S}_s, \mind_s) \text{,} \]
as the \emph{physical system}.
For further details regarding this setting, see \cite[Sect. 4]{alex_shao:nc_inf}.

As in the introduction, we let $L$ be the tangent null vector field on $\mc{N}$ satisfying $L s \equiv 1$, that is, $L$ is the normalized tangent field for the $s$-parametrized null generators of $\mc{N}$.
One can show that the associated $s$-covariant derivative on horizontal fields, $\nasla_s$, is precisely the projection of the \emph{spacetime} covariant derivative $D_L$ to the $\mc{S}_s$'s.
In fact, this was the definition used for null covariant derivatives in \cite{kl_rod:cg}.

The objects of analysis in the physical setting---the connection coefficients $\chi$, $\zeta$, $\ul{\chi}$ and the curvature components $\alpha$, $\beta$, $\rho$, $\sigma$, $\ul{\rho}$ defined in the introduction---can now be treated as horizontal tensor fields in the physical system.
Moreover, these quantities are related to each other via the \emph{null structure equations}, which, in this foliation, can be found in \cite[Prop. 4.1]{alex_shao:nc_inf}.
Finally, we note that in the physical setting, the second fundamental form $\slashed{k}$ is precisely $\chi$.

\subsubsection{The Renormalized Setting}

Both in this paper and in \cite{alex_shao:nc_inf}, it is easier to work with a different horizontal covariant system in which:
\begin{itemize}
\item The metrics on the level spheres $\mc{S}_s$ are nearly identical.

\item The null parameter ranges over a finite, rather than infinite, interval.
\end{itemize}
As a result, we transform our physical system into one which achieves the above.

First, one rescales the $\mind_s$'s by defining $\gamma | \mc{S}_s = s^{-2} \mind_s$.
Since heuristically, $s$ corresponds roughly to the area radii of the $\mc{S}_s$'s, this has the effect of transforming the infinite near-cone $(\mc{N}, \mind)$ to an infinite near-cylinder.
Next, we apply the change of parameter $t = 1 - s^{-1}$, which transforms the infinite interval $s \in [1, \infty)$ to a finite interval $t \in [0, 1)$.
In particular, the initial sphere $s = 1$ corresponds now to $t = 0$, while the limit $s \nearrow \infty$ at infinity corresponds to $t \nearrow 1$.

We will use the symbol $S_\tau$ to denote the level set $t = \tau$.
\footnote{Note in particular that $S_\tau = \mc{S}_{ (1 - \tau)^{-1} }$.}
Moreover, we let $\gamma_t$ denote the rescaled metric $\gamma$ on $S_t$, and we let $\epsilon_t$ denote the volume form associated with $\gamma_t$.
We will refer to resulting horizontal covariant system,
\[ (\mc{N}, \gamma) = \bigcup_{0 \leq t < 1} (S_t, \gamma_t) \text{,} \]
as the \emph{renormalized system}.
For analytical purposes, this renormalized setting is often the more natural structure to work with.
For example, the general estimates developed in \cite{shao:stt} apply to the renormalized, but not the physical, system.
\footnote{However, estimates in the renormalized system (in particular, the main renormalized estimates in \cite{alex_shao:nc_inf}) can be directly translated to corresponding estimates in the physical system.}
For a more detailed construction of this renormalized setting, see \cite[Sect. 4.4]{alex_shao:nc_inf}.

Remaining with the conventions from \cite{alex_shao:nc_inf}, we define the following:
\begin{itemize}
\item \emph{Renormalized Ricci coefficients:}
\begin{equation} \label{eq.renorm_ricci} H = \chi - s^{-1} \mind \text{,} \qquad Z = s \zeta \text{,} \qquad \ul{H} = s^{-1} \ul{\chi} + s^{-2} \paren{ 1 - \frac{2m}{s} } \mind \text{.} \end{equation}

\item \emph{Renormalized curvature components:}
\begin{equation} \label{eq.renorm_curv} A = s^2 \alpha \text{,} \qquad B = s^3 \beta \text{,} \qquad R = s^3 (\rho + i \sigma) + 2 m_{\rm S} \text{,} \qquad \ul{B} = s \ul{\beta} \text{.} \end{equation}

\item \emph{Renormalized mass aspect function:}
\begin{equation} \label{eq.renorm_maf} M = s^3 \mu - 2 m_{\rm S} = - \gamma^{ab} \nabla_a Z_b - R + \frac{1}{2} \gamma^{ac} \gamma^{bd} \hat{H}_{ab} \hat{\ul{H}}_{cd} \text{.} \end{equation}
Here, $\hat{H}$ and $\hat{\ul{H}}$ represent the ($\gamma$-)traceless parts of $H$ and $\ul{H}$.
\end{itemize}
These quantities will be treated as horizontal fields in the renormalized system.

Let $\nabla$, $\lapl$, and $\mc{K}$ denote the Levi-Civita connections, the (Bochner) Laplacians, and the Gauss curvatures, respectively, for the $(S_t, \gamma_t)$'s.
Following earlier conventions, we let $\nabla_t$ denote the corresponding $t$-covariant derivative.
Note that since $\mf{L}_t = s^2 \mf{L}_s$ by our defined relation between $s$ and $t$, then a direct computation shows that the second fundamental form $k$ for the renormalized system is given by
\[ k = \frac{1}{2} \mf{L}_t \gamma = H \text{.} \]
As a result, we can write
\begin{align*}
\nabla_t \Psi_{u_1 \dots u_l}^{v_1 \dots v_r} &= \mf{L}_t \Psi_{u_1 \dots u_l}^{v_1 \dots v_r} - \sum_{i = 1}^l \gamma^{cd} H_{u_i c} \Psi_{u_1 \hat{d}_i u_l}^{v_1 \dots v_r} + \sum_{j = 1}^r \gamma^{c v_j} H_{cd} \Psi_{u_1 \dots u_l}^{v_1 \hat{d}_j v_r} \text{,}
\end{align*}

\begin{remark}
In contrast to the physical system, $\nabla_t$ is not characterized as a projection of a spacetime covariant derivative to the $S_t$'s.
But, in both \cite{alex_shao:nc_inf} and this paper, $\nabla_t$ is a more natural evolution operator to consider than the projection $\nasla_s$.
\end{remark}

From a series of rather tedious computations, one can convert the null structure equations in the physical setting to corresponding equations in the renormalized settings (in terms of $H$, $A$, $\gamma$, $\nabla$, $\nabla_t$, etc.).
For the full list of renormalized structure equations, the reader is referred to \cite[Prop. 4.2]{alex_shao:nc_inf}.

Finally, as we will be working exclusively with renormalized settings in our analysis, we make the following assumptions regarding notations:
\begin{itemize}
\item From now on, objects will be stated in terms of the renormalized rather than the physical setting.
The lone exception is that we may sometimes refer to the affine parameter $s$ when convenient.

\item By ``$\trace$", we mean the trace with respect to $\gamma$, e.g., $\trace H = \gamma^{ab} H_{ab}$.

\item For a horizontal tensor field $\Psi$ on $\mc{N}$, we will generally use the symbol $\Psi_\tau$ to refer to the restriction of $\Psi$ to the level set $S_\tau$.
\end{itemize}

\subsubsection{Spherical Cuts}
\label{SphCuts}
In addition to the above foliations of $\mc{N}$, we will occasionally deal with more general spherical hypersurfaces of $\mc{N}$.
Here, we introduce the relevant terminology that will prove to be useful throughout the paper.

First, we will use the term \emph{spherical cut} of $\mc{N}$ to refer to any codimension-$1$ submanifold $\Sigma$ of $\mc{N}$ that intersects each null generator of $\mc{N}$ exactly once; note that any such $\Sigma$ is necessarily spacelike.
Basic examples of cuts include the level sets of $t$ and $s$.
Moreover, for each spherical cut $\Sigma$, we define the \emph{transport map} $\Phi_\Sigma: \Sigma \rightarrow S_0$, which sends each $P \in \Sigma$ to the corresponding point on $S_0 = \mc{S}$ along the same null generator as $P$.
Since $\Phi_\Sigma$ is a diffeomorphism, it also induces a push-forward $\Phi_\Sigma^\ast$, which identifies tensor fields on $\Sigma$ with tensor fields on $S_0$.

A basic construction that will be useful on occasion is \emph{transported coordinates}.
Given a coordinate system $x^1, x^2$ on the initial sphere $S_0$, we define corresponding coordinates on a spherical cut $\Sigma$ of $\mc{N}$ by transport along the null generators of $\mc{N}$.
In other words, we define these transported coordinates on $\Sigma$ by $x^a \circ \Phi_\Sigma$.
In particular, this construction can be done with $\Sigma$ being any level sphere $S_t$.

Finally, as we will deal with ``limits at infinity" in future sections, it will be convenient to treat this more concretely.
For this purpose, we formally introduce a limiting sphere $S_1$ ``at infinity", that is, we attach an upper spherical boundary $S_1$ to $\mc{N}$, which we can think of as the level set $t = 1$.
Like for finite cuts, we can once again define a transport map $\Phi_{S_1}: S_1 \rightarrow S_0$.
At a heuristic level, $S_1$ represents the spherical cut of future null infinity created by $\mc{N}$.

\subsection{Changes of Foliation} \label{sect:prelim.cfol}

In order to obtain the relevant physical limits for Theorem \ref{the.thm}, we will need to consider transformations from our system in Section \ref{sect:prelim.fol} to other geodesic foliations of $\mc{N}$.
Such a change of foliation is generally described by a \emph{constant} rescaling of the tangent vector field of each null generator of $\mc{N}$ (though different null generators may be scaled by different factors).
These rescalings can be represented by a \emph{distortion function} $v: \mc{S} \rightarrow \R$, with $e^v$ as the rescaling factor for the null generators.
$v$ is then extended to all of $\mc{N}$ so that it is constant on each null generator of $\mc{N}$ (i.e., $\nabla_t v \equiv 0$).
In particular, whenever $v$ is small, one obtains a new foliation that is very close to the original.

Here, we will adopt the following convention: objects defined with respect to the new geodesic foliation will be denoted with a $\prime$; objects without this $\prime$ are presumed to be with respect to the original foliation.
By definition, the tangent vector field $L^\prime$ for the new foliation is related to the original vector field $L$ via the relation
\begin{equation} \label{eq.cf_L} L^\prime = e^v L \text{.} \end{equation}
Furthermore, because of our normalization, $\mc{S}^\prime_1$ (i.e., the set $s^\prime = 1$) should coincide with the initial sphere $\mc{S} = \mc{S}_1$ of $\mc{N}$.
Consequently, we have that
\begin{equation} \label{eq.cf_s} s^\prime - 1 = e^{-v} (s - 1) \text{.} \end{equation}

\subsubsection{Identification of Horizontal Fields}

If $X$ and $Y$ are vector fields on $\mc{N}$ tangent to the $\mc{S}_s$'s (i.e., horizontal in the $s$-foliation), then we define
\begin{equation} \label{eq.cf_X} X^\prime = X + (s - 1) \nasla_X v \cdot L \text{,} \end{equation}
and analogously for $Y^\prime$ and $Y$.
If $\mind^\prime$ denotes the induced metrics on the level sets $\mc{S}^\prime_{s^\prime}$ of $s^\prime$, then $X^\prime$ and $Y^\prime$ are everywhere tangent to the $\mc{S}^\prime_{s^\prime}$'s, and
\begin{equation} \label{eq.cf_mind} \mind^\prime ( X^\prime, Y^\prime ) = \mind (X, Y) \text{,} \qquad \g ( X^\prime, L^\prime ) \equiv 0 \text{.} \end{equation}
To be more precise, given any $P \in \mc{N}$, then $X^\prime |_P$ and $Y^\prime |_P$ are the (metric-preserving) natural projections of $X |_P$ and $Y |_P$, which are tangent to the level set of $s$ through $P$, to the corresponding level set of $s^\prime$ through $P$.

The correspondence \eqref{eq.cf_X} provides a natural method for identifying and comparing horizontal tensors from different foliations of $\mc{N}$.
Suppose $\Psi^\prime$ is an $s^\prime$-horizontal covariant tensor field, i.e., $\Psi^\prime$ is tangent to the $\mc{S}^\prime_{s^\prime}$'s.
Then, $\Psi^\prime$ naturally induces an $s$-horizontal field $\Psi^\dagger$ in the following manner: given $t$-horizontal vector fields $X_1, \dots, X_k$, with $k$ the rank of $\Psi^\prime$, we define
\[ \Psi^\dagger (X_1, \dots, X_l) := \Psi^\prime (X_1^\prime, \dots, X_l^\prime) \text{.} \]
In other words, at each $P \in \mc{N}$, one projects via \eqref{eq.cf_X} from the $s$-tangent space at $P$ to the corresponding $s^\prime$-tangent space at $P$.
To reduce notational baggage, when the context is clear, the induced $t$-horizontal field $\Psi^\dagger$ will also be denoted by $\Psi^\prime$.

\begin{remark}
The above also allows us to make sense of the difference of two horizontal fields living in different foliations.
This point will become important in Section \ref{sect:convergence}.
\end{remark}

It will be convenient to adjust our index notations to reflect the above correspondence.
Henceforth, \emph{given an equation with quantities in both foliations, identical indices for primed and unprimed quantities will always refer to frame elements related via \eqref{eq.cf_X}.}
With this convention, the first identity in \eqref{eq.cf_mind} can be restated
\[ \mind^\prime_{ab} = \mind_{ab} \text{,} \qquad \mind^{\prime ab} = \mind^{ab} \text{.} \]
More generally, with $\Psi^\prime$ as before, the induced $t$-horizontal field $\Psi^\dagger$ is defined
\[ \Psi^\dagger_{ u_1 \dots u_l } := \Psi^\prime_{ u_1 \dots u_l } \text{.} \]

\begin{remark}
This indexing convention is also compatible with the transported coordinate systems described in Section \ref{sect:prelim.fol}.
Consider a coordinate system on $S_0$, which yields transported coordinates on both the $\mc{S}_s$'s and the $\mc{S}^\prime_{s^\prime}$'s.
In this case, the associated coordinate vector fields on the $\mc{S}_s$'s and $\mc{S}^\prime_s$'s are related to each other via \eqref{eq.cf_X}.
Thus, we can equivalently define $\Psi^\dagger$ by requiring that $\Psi^\dagger$ and $\Psi^\prime$ act the same way on corresponding transported coordinate vector fields.
\end{remark}

\subsubsection{Changes of Physical Systems}

We now have two physical systems,
\[ \mc{N} = \bigcup_{s \geq 1} (\mc{S}_s, \mind_s) = \bigcup_{s^\prime \geq 1} (\mc{S}^\prime_{s^\prime}, \mind^\prime_{s^\prime}) \text{,} \]
in the sense of Section \ref{sect:prelim.fol}.
In the transformed $s^\prime$-foliation, we again have the usual Ricci coefficients $\chi^\prime$, $\zeta^\prime$, $\ul{\chi}^\prime$ and the curvature components $\alpha^\prime$, $\beta^\prime$, $\rho^\prime$, $\sigma^\prime$, $\ul{\beta}^\prime$ on $\mc{N}$.
Moreover, we can derive \emph{change of foliation formulas} relating these quantities to the corresponding quantities $\chi$, $\alpha$, etc., in the original $s$-foliation.

These transformation laws are listed in the following proposition.
Throughout, we always use the indexing conventions described above.

\begin{proposition} \label{thm.cfol.phys}
Consider the geodesic $s$- and $s^\prime$-foliations of $\mc{N}$, related via the distortion function $v$, as described above.
Then:
\begin{itemize}
\item The following relations hold for the Ricci coefficients:
\begin{align}
\label{eq.cf_chi} \chi^\prime_{ab} &= e^v \chi_{ab} \text{,} \\
\label{eq.cf_zeta} \zeta^\prime_a &= \zeta_a + (s - 1) \mind^{bc} \nasla_b v \cdot \chi_{ac} - \nasla_a v \text{,} \\
\label{eq.cf_chibar} \ul{\chi}^\prime_{ab} &= e^{-v} \ul{\chi}_{ab} + 2 (s - 1) e^{-v} \nasla_{ab} v - 2 (s - 1) e^{-v} ( \nasla_a v \cdot \zeta_b + \nasla_b v \cdot \zeta_a ) \\
\notag &\qquad + (s - 1)^2 \mind^{cd} e^{-v} \nasla_c v ( \nasla_d v \cdot \chi_{ab} - 2 \nasla_a v \cdot \chi_{bd} - 2 \nasla_b v \cdot \chi_{ad} ) \\
\notag &\qquad + 2 (s - 1) e^{-v} \nasla_a v \nasla_b v \text{.}
\end{align}

\item The following relations hold for the curvature coefficients:
\footnote{The symbol $\star$, in \eqref{eq.cf_betabar}, refers to the Hodge dual, i.e., ${}^\star \nasla_a v = \vind_{ac} \mind^{cb} \nasla_b v$.}
\begin{align}
\label{eq.cf_alpha} \alpha^\prime_{ab} &= e^{2v} \alpha_{ab} \text{,} \\
\label{eq.cf_beta} \beta^\prime_a &= e^v \beta_a + (s - 1) \mind^{bc} e^v \nasla_b v \cdot \alpha_{ac} \text{,} \\
\label{eq.cf_rho} \rho^\prime &= \rho + 2 (s - 1) \mind^{ab} \nasla_a v \cdot \beta_b + (s - 1)^2 \mind^{ac} \mind^{bd} \nasla_a v \nasla_b v \cdot \alpha_{cd} \text{,} \\
\label{eq.cf_sigma} \sigma^\prime &= \sigma - 2 (s - 1) \vind^{ab} \nasla_a v \cdot \beta_b - (s - 1)^2 \vind^{ac} \mind^{bd} \nasla_a v \nasla_b v \cdot \alpha_{cd} \text{,} \\
\label{eq.cf_betabar} \ul{\beta}^\prime_a &= e^{-v} \ul{\beta}_a + 3 (s - 1) e^{-v} \nasla_a v \cdot \rho + (s - 1) e^{-v} {}^\star \nasla_a v \cdot \sigma \\
\notag &\qquad + (s - 1)^2 e^{-v} \nasla_b v ( 4 \mind^{bc} \nasla_a v \cdot \beta_c - \mind^{bc} \nasla_c v \cdot \beta_a + 2 \vind^{bc} {}^\star \nasla_a v \cdot \beta_c ) \\
\notag &\qquad + (s - 1)^3 e^{-v} \mind^{bd} \mind^{ce} \nasla_b v \nasla_c v ( 2 \nasla_a v \cdot \alpha_{de} - \nasla_e v \cdot \alpha_{ad} ) \text{.}
\end{align}

\item Suppose $\Psi^\prime$ is a horizontal tensor field in the $s^\prime$-foliation.
Then,
\begin{align}
\label{eq.cf_nasla} \nasla^\prime_a \Psi^\prime_{ u_1 \dots u_l } &= \nasla_a \Psi^\prime_{u_1 \dots u_l} + (s - 1) \nasla_a v \cdot \nasla_s \Psi^\prime_{u_1 \dots u_l} \\
\notag &\qquad - (s - 1) \mind^{cd} \sum_{i = 1}^l ( \nasla_{u_i} v \cdot \chi_{da} - \nasla_d v \cdot \chi_{a u_i} ) \Psi^\prime_{u_1 \hat{c}_i u_l} \text{,}
\end{align}
where in the right-hand side, $\Psi^\prime$ refers to the induced $s$-horizontal field.
\end{itemize}
\end{proposition}

\begin{proof}
See Appendix \ref{sect:cfol.phys}.
\end{proof}

\subsubsection{Changes of Renormalized Systems}

Next, we apply the renormalization from Section \ref{sect:prelim.fol} to both physical systems to produce two renormalized systems,
\[ \mc{N} = \bigcup_{0 \leq t < 1} (S_t, \gamma_t) = \bigcup_{0 \leq t^\prime < 1} (S^\prime_{t^\prime}, \gamma^\prime_{t^\prime}) \text{,} \]
with respect to the finite parameters $t = 1 - s^{-1}$ and $t^\prime = 1 - s^{\prime -1}$.

For convenience, we define, for any integer $k$, the coefficients
\footnote{Intuitively, the $\mc{B}_k$'s will remain uniformly close to $1$, while the $\mc{C}_k$'s will remain uniformly small; for specifics, see the remark after Proposition \ref{thm.renorm_cfol}.}
\begin{equation} \label{eq.Ck} \mc{B}_k = [ 1 + s^{-1} ( e^v - 1 ) ]^k \text{,} \qquad \mc{C}_k = s (\mc{B}_k - 1) \text{,} \end{equation}
which will be present in several upcoming computations.
For example, by \eqref{eq.cf_s},
\begin{equation} \label{eq.cf_sk} s^{\prime k} = e^{-k v} \mc{B}_k \cdot s^k \text{,} \qquad \mc{B}_k = (1 + s^{-1} \mc{C}_k) \text{.} \end{equation}
Recalling the definitions of $t$ and $t^\prime$, we also have the identity
\begin{equation} \label{eq.cf_t} t^\prime = t - s^{-1} ( e^v - 1 ) - s^{-2} e^v \mc{C}_{-1} \text{.} \end{equation}
Moreover, since $\gamma^\prime = s^{\prime -2} \mind^\prime$ and $\gamma = s^{-2} \mind$, then \eqref{eq.cf_sk} implies
\begin{equation} \label{eq.cf_gammainv} \gamma^{\prime ab} = e^{-2v} \mc{B}_2 \gamma^{ab} \text{,} \qquad \gamma_{ab} = e^{-2v} \mc{B}_2 \gamma^\prime_{ab} \text{,} \end{equation}
where we use the same indexing conventions mentioned earlier.

Next, the quantities $H$, $Z$, $\ul{H}$, $A$, $B$, $R$, $\ul{B}$, $M$ also have counterparts in the $t^\prime$-foliation.
Using Proposition \ref{thm.cfol.phys}, we can derive identities comparing them.

\begin{proposition} \label{thm.cfol.renorm}
Consider the renormalized $t$- and $t^\prime$-foliations of $\mc{N}$, related via the distortion function $v$, as described above.
Then:
\begin{itemize}
\item The following relations hold for the renormalized Ricci coefficients:
\begin{align}
\label{eq.cf_H} H^\prime_{ab} &= e^v ( H_{ab} - \mc{C}_{-1} \gamma_{ab} ) \text{,} \\
\label{eq.cf_Z} Z^\prime_a &= e^{-v} \mc{B}_1 ( Z_a + t \gamma^{bc} \nabla_b v \cdot H_{ac} - \nabla_a v ) \text{,} \\
\label{eq.cf_Hbar} \ul{H}^\prime_{ab} &= \mc{B}_{-1} [ \ul{H}_{ab} + 2 t \nabla_{ab} v + ( e^{2v} - 1 + s^{-1} e^{2v} \mc{C}_{-1} ) \gamma_{ab} ] \\
\notag &\qquad + 2 m_{\rm S} s^{-1} e^{3v} \mc{B}_{-1} [ (e^{-3v} - 1) - s^{-1} \mc{C}_{-1} (1 + \mc{B}_{-1}) ] \gamma_{ab} \\
\notag &\qquad + \mc{B}_{-1} [ t^2 \gamma^{cd} \nabla_c v \nabla_d v \cdot \gamma_{ab} - 2 t (1 - 3 s^{-1}) \nabla_a v \nabla_b v ] \\
\notag &\qquad + \mc{B}_{-1} s^{-1} t^2 \gamma^{cd} \nabla_c v ( \nabla_d v \cdot H_{ab} - 2 \nabla_a v \cdot H_{bd} - 2 \nabla_b v \cdot H_{ad} ) \\
\notag &\qquad - 2 \mc{B}_{-1} s^{-1} t ( \nabla_a v \cdot Z_b + \nabla_b v \cdot Z_a ) \text{.}
\end{align}

\item The following relations hold for the renormalized curvature coefficients:
\begin{align}
\label{eq.cf_A} A^\prime_{ab} &= \mc{B}_2 A_{ab} \text{,} \\
\label{eq.cf_B} B^\prime_a &= e^{-2v} \mc{B}_3 ( B_a + t \gamma^{bc} \nabla_b v \cdot A_{ac} ) \text{,} \\
\label{eq.cf_R} R^\prime &= e^{-3v} \mc{B}_3 R + 2 m e^{-3v} ( e^{3v} - 1 - s^{-1} \mc{C}_3 ) \\
\notag &\qquad + e^{-3v} \mc{B}_3 s^{-1} t ( \gamma^{ab} - i \epsilon^{ab} ) \nabla_a v ( 2 B_b + t \gamma^{cd} \nabla_c v \cdot A_{bd} ) \text{,} \\
\label{eq.cf_Bbar} \ul{B}^\prime_a &= e^{-2v} \mc{B}_1 \ul{B}_a + e^{-2 v} \mc{B}_1 s^{-1} t [ 3 \nabla_a v \cdot ( \real R - 2 m ) + {}^\star \nabla_a v \cdot \imag R ] \\
\notag &\qquad + e^{-2 v} \mc{B}_1 s^{-2} t^2 \nabla_b v [ \gamma^{bc} ( 4 \nabla_a v \cdot B_c - \nabla_c v \cdot B_a ) + 2 \epsilon^{bc} {}^\star \nabla_a v \cdot B_c ] \\
\notag &\qquad + e^{-2 v} \mc{B}_1 s^{-2} t^3 \gamma^{bd} \gamma^{ce} \nabla_b v \nabla_c v ( 2 \nabla_a v \cdot A_{de} - \nabla_e v \cdot A_{ad} ) \text{.}
\end{align}

\item Suppose $\Psi^\prime$ is a $t^\prime$-horizontal tensor field.
Then,
\begin{align}
\label{eq.cf_nabla} \nabla^\prime_a \Psi^\prime_{ u_1 \dots u_l } &= \nabla_a \Psi^\prime_{u_1 \dots u_l} + s^{-1} t \nabla_a v \cdot \nabla_t \Psi^\prime_{u_1 \dots u_l} + l t \nabla_a v \cdot \Psi^\prime_{u_1 \dots u_l} \\
\notag &\qquad - s^{-1} t \gamma^{cd} \sum_{i = 1}^l ( \nabla_{u_i} v \cdot H_{da} - \nabla_d v \cdot H_{a u_i} ) \Psi^\prime_{u_1 \hat{c}_i u_l} \\
\notag &\qquad + t \sum_{i = 1}^l ( \nabla_{u_i} v \cdot \Psi^\prime_{u_1 \hat{a}_i u_l} - \gamma^{cd} \nabla_d v \cdot \gamma_{a u_i} \Psi^\prime_{u_1 \hat{c}_i u_l} ) \text{,}
\end{align}
where $\Psi^\prime$ on the right-hand side refers to the induced $t$-horizontal field.
\end{itemize}
\end{proposition}

\begin{proof}
See Appendix \ref{sect:cfol.renorm}.
\end{proof}

In particular, we examine these formulas on the initial sphere $S_0 = S^\prime_0$.

\begin{corollary} \label{thm.cfol.init}
Consider the renormalized $t$- and $t^\prime$-foliations of $\mc{N}$, related via the distortion function $v$, as described above.
Then, on $S_0 = S^\prime_0$:
\begin{align}
\label{eq.cf_H_init} H^\prime_{ab} &= e^v H_{ab} + ( e^v - 1 ) \gamma_{ab} \text{,} \\
\label{eq.cf_Z_init} Z^\prime_a &= Z_a + \nabla_a v \text{,} \\
\label{eq.cf_Hbar_init} \ul{H}^\prime_{ab} &= e^{-v} \ul{H}_{ab} + ( 1 - 2 m_{\rm S} ) ( 1 - e^{-v} ) \gamma_{ab} \text{,} \\
\label{eq.cf_H_deriv_init} \nabla^\prime_a ( \trace^\prime H^\prime ) &= e^v \nabla_a ( \trace H ) + e^v \nabla_a v \cdot \trace H + 2 e^v \nabla_a v \text{,} \\
\label{eq.cf_M_init} M^\prime &= M - \lapl v \text{.}
\end{align}
\end{corollary}

\begin{proof}
See Appendix \ref{sect:cfol.renorm}.
\end{proof}

\subsection{Norms} \label{sect:prelim.norms}

We now describe the norms we will use throughout the paper.
Fix first a spherical cut $\Sigma \subseteq \mc{N}$ and a Riemannian metric $h$ on $\Sigma$.
Given a tensor field $F$ on $\Sigma$, we define the following geometric norms for $(\Sigma, h)$:
\begin{itemize}
\item Given any $1 \leq q \leq \infty$, we let $\| F \|_{ L^q_x (\Sigma, h) }$ denote the usual Lebesgue $L^q$-norm of $F$ over $\Sigma$, with respect to $h$.

\item In a few instances, we will need to refer to the geometric (fractional) Sobolev and Besov norms used in \cite{alex_shao:nc_inf, shao:stt}.
In particular, we let
\begin{align*}
\| F \|_{ H^s_x (\Sigma, h) }^2 &= \sum_{k \geq 0} 2^{2sk} \| P_k F \|_{ L^2_x (\Sigma, h) }^2 + \| P_{< 0} F \|_{ L^2_x (\Sigma, h) }^2 \text{,} \\
\| F \|_{ B^s_x (\Sigma, h) } &= \sum_{k \geq 0} 2^{sk} \| P_k F \|_{ L^2_x (\Sigma, h) } + \| P_{< 0} F \|_{ L^2_x (\Sigma, h) } \text{,}
\end{align*}
where the $P_k$'s and $P_{< 0}$ are \emph{geometric Littlewood-Paley operators} on $S_\tau$, based on spectral decompositions of the Laplacian.
For precise definitions and discussions on these operators, see \cite[Sect. 2.2]{alex_shao:nc_inf} or \cite[Sect. 2.3]{shao:stt}.
\footnote{Alternatively, one could also utilize the geometric Littlewood-Paley operators of \cite{kl_rod:glp}, based instead on the geometric tensorial heat flow.}
\end{itemize}

\begin{remark}
For various technical reasons, the above geometric Sobolev and Besov norms were essential to the results of \cite{alex_shao:nc_inf, shao:stt}.
Here, though, we will only require, in a few instances, some isolated facts regarding these norms.
These primarily include certain product and elliptic estimates found in \cite{shao:stt}.
\end{remark}

\begin{remark}
When $\Sigma = S_t$, then unless otherwise stated, the norms will by default be with respect to the renormalized metric $\gamma_t$, that is, $\| F \|_{ L^q_x (S_t) } := \| F \|_{ L^q_x (S_t, \gamma_t) }$.
Similarly, given a change of foliations as in Section \ref{sect:prelim.cfol}, then whenever $\Sigma = S^\prime_{t^\prime}$, our conventions dictate that $\| F \|_{ L^q_x (S^\prime_{t^\prime}) } := \| F \|_{ L^q_x (S^\prime_{t^\prime}, \gamma^\prime_{t^\prime}) }$.
\end{remark}

Next, we consider analogous iterated norms over all of $\mc{N}$:
\begin{itemize}
\item Given a horizontal tensor field $\Psi$ on $\mc{N}$, along with $1 \leq p, q \leq \infty$, we let
\[ \| \Psi \|_{ L^{p, q}_{t, x} } = \| \Psi \|_{ L^{p, q}_{t, x} (\gamma) } \text{,} \qquad \| \Psi \|_{ L^{q, p}_{x, t} } = \| \Psi \|_{ L^{q, p}_{x, t} (\gamma) } \]
be the iterated Lebesgue norms over $\mc{N}$.
In general, the subscript ``$x$" indicates integrals with respect to the spheres $(S_t, \gamma_t)$, while ``$t$" refers to integration over the parameter $t$, relative to the measure $dt$.
In an $L^{p, q}_{t, x}$-norm, one takes first the $L^q_x$-norms on the $S_t$'s and then the $L^p$-norm in $t$.
For an $L^{q, p}_{x, t}$-norm, one integrates first in $t$ and then over the spheres.
\footnote{For more explicit formulas, see \cite[Sect. 3.3]{alex_shao:nc_inf}, as well as \cite{shao:stt}.}

\item In a few instances, we will need to consider iterated integrals over only a part of $\mc{N}$.
Given spherical cuts $\Sigma$ and $\Sigma^\prime$ of $\mc{N}$, we let
\[ \| \cdot \|_{ L^{p, q}_{t, x} (\Sigma, \Sigma^\prime) } = \| \cdot \|_{ L^{p, q}_{t, x} (\Sigma, \Sigma^\prime, \gamma) } \text{,} \qquad \| \cdot \|_{ L^{q, p}_{x, t} (\Sigma, \Sigma^\prime) } = \| \cdot \|_{ L^{q, p}_{x, t} (\Sigma, \Sigma^\prime, \gamma) } \]
denote the aforementioned $L^{p, q}_{t, x}$- and $L^{q, p}_{x, t}$-norms, but only over the region of $\mc{N}$ that lies between $\Sigma$ and $\Sigma^\prime$.
Similarly, we define
\[ \| \cdot \|_{ L^{p, q}_{t, x} (\Sigma, S_1) } = \| \cdot \|_{ L^{p, q}_{t, x} (\Sigma, S_1, \gamma) } \text{,} \qquad \| \cdot \|_{ L^{q, p}_{x, t} (\Sigma, S_1) } = \| \cdot \|_{ L^{q, p}_{x, t} (\Sigma, S_1, \gamma) } \text{,} \]
representing norms over the region of $\mc{N}$ that lies above $\Sigma$.

\item We will also require the iterated Besov norms
\[ \| \Psi \|_{ B^{p, 0}_{t, x} } = \| \Psi \|_{ B^{p, 0}_{t, x} (\gamma) } = \sum_{k \geq 0} \| P_k \Psi \|_{ L^{p, 2}_{t, x} } + \| P_{< 0} \Psi \|_{ L^{p, 2}_{t, x} } \text{,} \]
where $1 \leq p \leq \infty$, and where the $P_k$'s and $P_{< 0}$ are the Littlewood-Paley operators on the $S_t$'s.
In particular, we have for any $t$ that
\[ \| \Psi \|_{ B^0_x (S_t) } \leq \| \Psi \|_{ B^{\infty, 0}_{t, x} } \text{.} \]
Again, this norm was used more extensively in \cite{alex_shao:nc_inf, shao:stt}.
\end{itemize}

\begin{remark}
Given a change of foliations as in Section \ref{sect:prelim.cfol}, then by our conventions, the iterated Lebesgue norms with respect to the new foliation are denoted
\[ \| \Psi \|_{ L^{p, q}_{t^\prime, x} } = \| \Psi \|_{ L^{p, q}_{t^\prime, x} (\gamma^\prime) } \text{,} \qquad \| \Psi \|_{ L^{q, p}_{x, t^\prime} } = \| \Psi \|_{ L^{q, p}_{x, t^\prime} (\gamma^\prime) } \text{.} \]
\end{remark}

Finally, we recall that given two norms on vector spaces $X$ and $Y$, we can define a corresponding ``intersection" norm on $X \cap Y$ by
\[ \| A \|_{ X \cap Y } := \| A \|_X + \| A \|_Y \text{.} \]
For example, we will consider norms of the form $L^{\infty, 2}_{x, t} \cap L^{4, \infty}_{x, t}$ in this paper.

\subsection{Control of the Null Geometry} \label{sect:prelim.ncinf}

It is useful to recall at this point the parts of the main theorem in \cite{alex_shao:nc_inf} which are relevant to this work.
For this purpose, we state here the following abridged version of \cite[Thm. 5.1]{alex_shao:nc_inf}:

\begin{theorem} \label{thm.nc_renorm}
Fix $0 \leq m_{\rm S} < 1/2$, and assume the following on $\mc{N} = \bigcup_{0 \leq t < 1} S_t$:
\begin{itemize}
\item The area of $(S_0, \gamma_0)$ is $4 \pi$.

\item The following curvature flux bounds hold on $\mc{N}$:
\begin{equation} \label{eq.renorm_ass_flux} \| A \|_{ L^{2, 2}_{t, x} } + \| B \|_{ L^{2, 2}_{t, x} } + \| R \|_{ L^{2, 2}_{t, x} } + \| \ul{B} \|_{ L^{2, 2}_{t, x} } \leq \Gamma \text{,} \end{equation}

\item The following initial value bounds hold on $S_0$:
\begin{align}
\label{eq.renorm_ass_init} \| \trace H \|_{ L^\infty_x (S_0) } + \| H \|_{ H^{1/2}_x (S_0) } + \| Z \|_{ H^{1/2}_x (S_0) } &\leq \Gamma \text{,} \\
\notag \| \ul{H} \|_{ B^0_x (S_0) } + \| \nabla ( \trace H ) \|_{ B^0_x (S_0) } + \| M \|_{ B^0_x (S_0) } &\leq \Gamma \text{.}
\end{align}
\end{itemize}
Then, for sufficiently small $\Gamma \ll 1$, depending on the geometry of $(S_0, \gamma_0)$,
\begin{align}
\label{eq.renorm_est} \| \trace H \|_{ L^{\infty, \infty}_{t, x} } + \| H \|_{ L^{\infty, 2}_{x, t} \cap L^{4, \infty}_{x, t} } + \| Z \|_{ L^{\infty, 2}_{x, t} \cap L^{4, \infty}_{x, t} } &\lesssim \Gamma \text{,} \\
\notag \| \nabla_t H \|_{ L^{2, 2}_{t, x} } + \| \nabla H \|_{ L^{2, 2}_{t, x} } + \| \nabla_t Z \|_{ L^{2, 2}_{t, x} } + \| \nabla Z \|_{ L^{2, 2}_{t, x} } &\lesssim \Gamma \text{,} \\
\notag \| \nabla_t \nabla (\trace H) \|_{ L^{2, 1}_{x, t} } + \| \nabla_t M \|_{ L^{2, 1}_{x, t} } + \| \nabla_t \ul{H} \|_{ L^{2, 2}_{t, x} } &\lesssim \Gamma \text{,} \\
\notag \| \nabla (\trace H) \|_{ L^{2, \infty}_{x, t} \cap B^{\infty, 0}_{t, x} } + \| M \|_{ L^{2, \infty}_{x, t} \cap B^{\infty, 0}_{t, x} } + \| \ul{H} \|_{ L^{2, \infty}_{x, t} \cap B^{\infty, 0}_{t, x} } &\lesssim \Gamma \text{,}
\end{align}
where the constants of the inequalities depend on the geometry of $(S_0, \gamma_0)$.
\footnote{See \cite[Thm. 5.1]{alex_shao:nc_inf} for the precise dependence on the geometry of $(S_0, \gamma_0)$.}
Also,
\begin{equation} \label{eq.renorm_est_curv} \| \mc{K} - 1 \|_{ H^{-1/2}_x (S_t) } \lesssim \| \trace \ul{H} \|_{ L^2_x (S_t) } + (1 - t) \Gamma \text{.} \end{equation}
\end{theorem}

One important consequence of Theorem \ref{thm.nc_renorm} is that, under these assumptions, certain regularity properties of the geometry of the initial sphere $(S_0, \gamma_0)$ are propagated to all the $(S_t, \gamma_t)$'s.
\footnote{In \cite{alex_shao:nc_inf, shao:stt}, this phenomenon was made precise via the \ass{r0}{}, \ass{r1}{}, and \ass{r2}{} conditions.}
Here, we briefly describe some of these properties, and we summarize their most important consequences.

Fix now a coordinate system on $S_0$, and consider the induced transported coordinates on the $S_t$'s.
In these coordinates, we can make the following observations:
\begin{itemize}
\item For a component $\gamma_{ab}$ of $\gamma$ in these coordinates, its rate of change with respect to $t$ is precisely $\partial_t \gamma_{ab} = 2 H_{ab}$ (see Section \ref{sect:prelim.fol} or \cite[Sect. 4.4]{alex_shao:nc_inf}).
Since $H$ remains small by \eqref{eq.renorm_est}, the metrics $\gamma_t$ vary little with respect to $t$, and hence the $\gamma_{ab}$'s are uniformly bounded.

\item Similarly, letting $\partial_c \gamma_{ab}$ denote a \emph{coordinate} derivative of $\gamma_{ab}$, one can show that $\partial_t \partial_c \gamma_{ab}$ is controlled by $\nabla H$, along with less dangerous lower-order terms.
Since $\nabla H$ is controlled in $L^{2, 2}_{t, x}$, it follows that the Christoffel symbols associated with transported coordinates are uniformly controlled in $L^2_x$.

\item A similar analysis (see, e.g., \cite[Proposition 4.5]{shao:stt}) yields analogous control for the volume forms $\epsilon_t$ associated with the $\gamma_t$'s.
\end{itemize}
By standard methods, we can use the above estimates to derive the following:

\begin{corollary} \label{thm.nc_renorm.norms}
Let $\Phi_{S_t}: S_t \rightarrow S_0$ be the transport map defined in Section \ref{sect:prelim.fol}.
Under the assumptions of Theorem \ref{thm.nc_renorm}, if $F$ is a tensor field on $S_t$, then
\begin{equation} \label{eq.norm_compare} \| F \|_{ L^q_x (S_t) } \simeq \| \Phi_{S_t}^\ast (F) \|_{ L^q_x (S_0) } \text{,} \qquad 1 \leq q \leq \infty \text{.} \end{equation}
In particular, if $| S_t |$ is the area of $(S_t, \gamma_t)$, then
\begin{equation} \label{eq.area_compare} | S_t | \simeq 4 \pi \text{.} \end{equation}
\end{corollary}

In other words, for Lebesgue norms, the choice of metric with respect to which we take these norms is unimportant.
For additional details, see \cite{alex_shao:nc_inf, kl_rod:cg, shao:stt}.

\subsubsection{Limits at Infinity}

In addition to uniform control of the connection coefficients on $\mc{N}$, one can also show via Theorem \ref{thm.nc_renorm} that limits of these same quantities exist at infinity.
First, we must make precise what such limits mean, as we are comparing tensor fields on different spheres with different geometries.
We say that a family of smooth spherical cuts $\fol{\Sigma}{y} \subseteq \mc{N}$, $y \in [0, 1)$, is \emph{going to infinity} iff
\[ \lim_{ y \nearrow 1 } \inf_{ \omega \in \fol{\Sigma}{y} } t (\omega) = 1 \text{.} \]

\begin{definition} \label{def.conv_inf}
Consider a family $\fol{A}{y}$, $y \in [0, 1)$, of tensor fields over a corresponding family of spherical cuts $\fol{\Sigma}{y}$ going to infinity, as well as a tensor field $A^1$ over $S_1$.
We say that the $\fol{A}{y}$'s converge in $L^q_x$ to $A^1$ iff
\footnote{The choice of the initial metric $\gamma_0$ is extraneous, as one obtains an equivalent definition if $\gamma_0$ is replaced by another Riemannian metric on $S_0$.}
\[ \lim_{ y \nearrow 1 } \| \Phi_{ \fol{\Sigma}{y} }^\ast (\fol{A}{y}) - \Phi_{ S_1 }^\ast (A^1) \|_{ L^q_x (S_0, \gamma_0) } = 0 \text{.} \]
\end{definition}

We can now adapt the discussion following Theorem \ref{thm.nc_renorm} to show that $\gamma_t$ has a limit as $t \nearrow 1$.
Due to the $L^{\infty, 2}_{x, t}$-bound for $H$ and the observation that $H$ captures the variation of $\gamma_t$ in the $t$-direction, it follows that the $\gamma_t$'s are Cauchy in $L^\infty_x$ as $t \nearrow 1$.
By a similar argument with $\nabla H$, we can see that the first (coordinate) derivatives of $\gamma$ are Cauchy in $L^2_x$ as $t \nearrow 1$.
As a result, we conclude that the $\gamma_t$'s converge to a limiting metric $\gamma_1$, both ``in $L^\infty_x$ and in $H^1_x$".

By similar arguments, one obtains limits for $H$, $Z$, $\ul{H}$, and $M$.
These follow from the integral bounds for $\nabla_t H$, $\nabla_t Z$, $\nabla_t \ul{H}$, and $\nabla_t M$ in \eqref{eq.renorm_est}.

\begin{corollary} \label{thm.nc_renorm.limit.rc}
Under the assumptions of Theorem \ref{thm.nc_renorm}, we have that $\trace H_t$ converges in $L^\infty_x$ to a function on $S_1$.
Furthermore, the quantities $H_t$, $Z_t$, $\ul{H}_t$, $\nabla (\trace H_t)$ and $M_t$ converge in $L^2_x$ to tensor fields on $S_1$.
\end{corollary}

For further details and proofs regarding limits at infinity, see \cite[Cor. 5.2]{alex_shao:nc_inf}.

\subsection{Small Changes of Foliations} \label{sect:prelim.cfol.small}

We now connect Theorem \ref{thm.nc_renorm} to the notions introduced in Section \ref{sect:prelim.cfol}.
Assume, as usual, two (renormalized) foliations of $\mc{N}$,
\[ \mc{N} = \bigcup_{0 \leq t < 1} (S_t, \gamma_t) = \bigcup_{0 \leq t^\prime < 1} (S^\prime_{t^\prime}, \gamma^\prime_{t^\prime}) \text{,} \]
where the latter foliation is obtained from the former via the distortion function $v$.
Suppose we are in the setting of Theorem \ref{thm.nc_renorm}, so that the geometry of $\mc{N}$ in the $t$-foliation is uniformly controlled.
If $v$ is similarly small, then the geometry in the new $t^\prime$-foliation should be similarly controlled.
The goal here is to make this statement precise through a number of estimates.

For instance, the following proposition states that for small enough $v$, Theorem \ref{thm.nc_renorm} also applies directly to the $t^\prime$-foliation of $\mc{N}$.

\begin{proposition} \label{thm.renorm_cfol}
Assume the setting of Theorem \ref{thm.nc_renorm}, and consider a change of foliation corresponding to the distortion function $v$.
Assume, moreover, that
\begin{equation} \label{eq.renorm_cfol_ass} \| \nabla^2 v \|_{ B^{\infty, 0}_{t, x} \cap L^{2, \infty}_{x, t} } + \| \nabla v \|_{ L^{\infty, \infty}_{t, x} } + \| v \|_{ L^{\infty, \infty}_{t, x} } \lesssim \Gamma \text{.} \end{equation}
If $\Gamma$ is sufficiently small, then all the conclusions of Theorem \ref{thm.nc_renorm} also hold with respect to the $t^\prime$-foliation, that is, in the $\gamma^\prime$-$t^\prime$-covariant system.
In particular,
\begin{align}
\label{eq.renorm_est_cfol} \| \trace^\prime H^\prime \|_{ L^{\infty, \infty}_{t^\prime, x} } + \| H^\prime \|_{ L^{\infty, 2}_{x, t^\prime} \cap L^{4, \infty}_{x, t^\prime} } + \| Z^\prime \|_{ L^{\infty, 2}_{x, t^\prime} \cap L^{4, \infty}_{x, t^\prime} } &\lesssim \Gamma \text{,} \\
\notag \| \nabla^\prime_{t^\prime} H^\prime \|_{ L^{2, 2}_{t^\prime, x} } + \| \nabla^\prime H^\prime \|_{ L^{2, 2}_{t^\prime, x} } + \| \nabla^\prime_{t^\prime} Z^\prime \|_{ L^{2, 2}_{t^\prime, x} } + \| \nabla^\prime Z^\prime \|_{ L^{2, 2}_{t^\prime, x} } &\lesssim \Gamma \text{,} \\
\notag \| \nabla^\prime_{t^\prime} \nabla^\prime (\trace^\prime H^\prime) \|_{ L^{2, 1}_{x, t^\prime} } + \| \nabla^\prime_{t^\prime} M^\prime \|_{ L^{2, 1}_{x, t^\prime} } + \| \nabla^\prime_{t^\prime} \ul{H}^\prime \|_{ L^{2, 2}_{t^\prime, x} } &\lesssim \Gamma \text{,} \\
\notag \| \nabla^\prime (\trace^\prime H^\prime) \|_{ L^{2, \infty}_{x, t^\prime} \cap B^{\infty, 0}_{t^\prime, x} } + \| M^\prime \|_{ L^{2, \infty}_{x, t^\prime} \cap B^{\infty, 0}_{t^\prime, x} } + \| \ul{H}^\prime \|_{ L^{2, \infty}_{x, t^\prime} \cap B^{\infty, 0}_{t^\prime, x} } &\lesssim \Gamma \text{.}
\end{align}
Furthermore, if $\mc{K}^\prime$ denotes the Gauss curvatures of the $( S^\prime_{t^\prime}, \gamma^\prime_{t^\prime} )$'s, then
\begin{equation} \label{eq.renorm_est_curv_cfol} \| \mc{K}^\prime - 1 \|_{ H^{-1/2}_x ( S^\prime_{t^\prime} ) } \lesssim \| \trace^\prime \ul{H}^\prime \|_{ L^2_x ( S^\prime_{t^\prime} ) } + (1 - t^\prime) \Gamma \text{.} \end{equation}
\end{proposition}

\begin{proof}
See Appendix \ref{sect:estimates.cfol.small}.
\end{proof}

\begin{remark}
Let $k$ be an integer satisfying $|k| \leq 3$, and suppose $|v|$ is sufficiently small everywhere on $\mc{N}$.
Then, we have the trivial bounds
\begin{equation} \label{eq.renorm_cfol_exp} e^{k v} \simeq 1 \text{,} \qquad | e^{k v} - 1 | \lesssim | v | \text{.} \end{equation}
In particular, this implies bounds for the coefficients $\mc{B}_k$ and $\mc{C}_k$ defined in \eqref{eq.Ck}:
\begin{equation} \label{eq.renorm_cfol_C} \mc{B}_k \simeq 1 \text{,} \qquad | \mc{C}_k | \lesssim | v | \text{.} \end{equation}
We will use these observations repeatedly in various upcoming estimates.
\end{remark}

Next, recall any $t^\prime$-horizontal field $\Psi^\prime$ induces a corresponding $t$-horizontal field, also denoted $\Psi^\prime$, via the projection \eqref{eq.cf_X}.
Thus, we can make sense of measuring $\Psi^\prime$ with respect to the $t$-foliation by taking a $\gamma$-$t$-norm of the $t$-horizontal induced field.
In particular, we can consider $\gamma$-$t$-norms of $H^\prime$, $Z^\prime$, $B^\prime$, $\nabla^\prime Z^\prime$, etc.

Assuming \eqref{eq.renorm_cfol_ass} for the moment, we observe the following:
\begin{itemize}
\item By the identity \eqref{eq.cf_gammainv}, along with \eqref{eq.renorm_cfol_exp} and \eqref{eq.renorm_cfol_C}, we see that corresponding horizontal frames in the $t$- and $t^\prime$-foliations are comparable.
\footnote{More specifically, a $\gamma^\prime$-orthonormal frame corresponds to a $\gamma$-orthogonal frame that is ``almost" orthonormal.
The same observation also holds in the reverse direction.}

\item As a result of the above, we conclude that $L^{p, q}_{t, x}$- and $L^{p, q}_{t^\prime, x}$-norms (i.e., with respect to the $t$- and $t^\prime$-foliations) of corresponding horizontal fields are comparable.
The same observation holds for $L^{q, p}_{x, t}$ and $L^{q, p}_{x, t^\prime}$-norms.
\end{itemize}
These observations link Proposition \ref{thm.renorm_cfol} to the subsequent proposition, which estimate the usual $t^\prime$-horizontal quantities ($H^\prime$, $A^\prime$, etc.) in terms of the $t$-foliation.
Such estimates will be especially useful throughout Section \ref{sect:convergence}.

\begin{proposition} \label{thm.renorm_cfol_bound}
Assume the hypotheses of Theorem \ref{thm.nc_renorm}, and consider a change of foliation corresponding to the distortion function $v$.
Assume in addition that \eqref{eq.renorm_cfol_ass} holds, with $\Gamma$ sufficiently small.
If we consider $H^\prime$, $Z^\prime$, $A^\prime$, $B^\prime$, $\nabla^\prime H^\prime$, $\nabla^\prime_{t^\prime} Z^\prime$ as $t$-horizontal fields, then we have the following bounds:
\begin{align}
\label{eq.renorm_cfol_bound} \| \trace^\prime H^\prime \|_{ L^{\infty, \infty}_{t, x} } + \| H^\prime \|_{ L^{\infty, 2}_{x, t} \cap L^{4, \infty}_{x, t} } + \| Z^\prime \|_{ L^{\infty, 2}_{x, t} \cap L^{4, \infty}_{x, t} } + \| \ul{H}^\prime \|_{ L^{2, \infty}_{x, t} } &\lesssim \Gamma \text{,} \\
\notag \| \nabla^\prime_{t^\prime} H^\prime \|_{ L^{2, 2}_{t, x} } + \| \nabla^\prime H^\prime \|_{ L^{2, 2}_{t, x} } + \| \nabla^\prime_{t^\prime} Z^\prime \|_{ L^{2, 2}_{t, x} } + \| \nabla^\prime Z^\prime \|_{ L^{2, 2}_{t, x} } + \| \nabla^\prime_{t^\prime} \ul{H}^\prime \|_{ L^{2, 2}_{t, x} } &\lesssim \Gamma \text{,} \\
\notag \| A^\prime \|_{ L^{2, 2}_{t, x} } + \| B^\prime \|_{ L^{2, 2}_{t, x} } + \| R^\prime \|_{ L^{2, 2}_{t, x} } + \| \ul{B}^\prime \|_{ L^{2, 2}_{t, x} } &\lesssim \Gamma \text{.}
\end{align}
\end{proposition}

\begin{proof}
See Appendix \ref{sect:estimates.cfol.small}.
\end{proof}

Propositions \ref{thm.renorm_cfol} and \ref{thm.renorm_cfol_bound} will be proved simultaneously in Appendix \ref{sect:estimates.cfol.small}.

\subsection{Strategy of the Proof} \label{sect:prelim.thm}

We close this section with an outline of the proof of the main result of this paper, Theorem \ref{the.thm}.
In particular, we relate parts of Theorem \ref{the.thm} to the renormalized settings discussed in this section.

The first observation is that the hypotheses of Theorem \ref{the.thm} are equivalent to those of Theorem \ref{thm.nc_renorm}.
This is a consequence of the transformation from the physical to the renormalized setting (see Section \ref{sect:prelim.fol}), as well as its inverse.
Thus, one can replace the assumptions of Theorem \ref{the.thm} with those of Theorem \ref{thm.nc_renorm}.
For more details on relating physical and renormalized versions of estimates, see \cite[Sect. 5]{alex_shao:nc_inf}.

From here, the proof of Theorem \ref{the.thm} consists of two main components.

\subsubsection{Construction of Asymptotically Flat Spheres}

The first component is that of constructing a family of spheres $\fol{\Sigma}{y} \subset \mc{N}$, $y \in [0, 1)$, going to infinity as $y \nearrow 1$, for which the area-normalized induced metrics become asymptotically round.
(Later, we will control on these $\fol{\Sigma}{y}$'s the physically relevant quantities, i.e., those related to the Bondi energy and the rate of energy loss.)
The $\fol{\Sigma}{y}$'s are defined as level spheres of a corresponding family of  geodesic foliations of $\mc{N}$, as described in Section \ref{sect:prelim.cfol}.
In addition, we establish estimates for the distortion functions $\fol{v}{y}$ associated with these refoliations, which will be essential later for demonstrating the convergence of various physically relevant quantities.

As we will be dealing with objects in the $\fol{v}{y}$-foliation of $\mc{N}$ for various $y \in [0, 1)$, we will adopt in the remainder of the paper the following convention.
Objects defined in the $\fol{v}{y}$-foliation of $\mc{N}$, either in the physical or the renormalized setting, will be denoted with a superscript $y$.
Objects without such a superscript $y$ will be understood to be with respect to the original foliation of $\mc{N}$.
For example, the finite parameter corresponding to the change of foliation given by $\fol{v}{y}$ is denoted $\fol{t}{y}$.
\footnote{In general, superscripts arguments in the notation will refer to a specific choice of a refoliation of $\mc{N}$ (e.g., $\fol{H}{y}$ is a renormalized Ricci coefficient in the $\fol{v}{y}$-foliation), while subscript arguments will refer to a restriction to a level sphere (e.g. $\fol{H}{y}_\tau$ is the restriction of $\fol{H}{y}$ to $\fol{t}{y} = \tau$).}

The results of this first component of the proof of Theorem \ref{the.thm} are summarized in the subsequent lemma, which will be proved in Section \ref{sect:distortion}.

\begin{lemma} \label{thm.distortion}
Assume the hypotheses of Theorem \ref{thm.nc_renorm} hold, with $\Gamma$ sufficiently small.
There is a 1-parameter family of (smooth) distortion functions $\fol{v}{y}$, $y \in [0, 1)$, such that if $\fol{s}{y}$ is the new affine parameter defined via \eqref{eq.cf_s}, then considering the renormalized $\fol{t}{y}$-foliation relative to $\fol{s}{y}$, the following conclusions hold:
\begin{itemize}
\item For each $y \in [0, 1)$, let $\fol{\Sigma}{y} = \fol{S}{y}_y$, i.e., the level sphere $\fol{t}{y} = y$, and let $\fol{h}{y} = \fol{\gamma}{y}_y$ denote the restriction of $\fol{\gamma}{y}$ to $\fol{\Sigma}{y}$.
Then, the spheres $(\fol{\Sigma}{y}, \fol{h}{y})$ become asymptotically round, in the sense that
\begin{equation} \label{eq.asympt_round} \lim_{y \nearrow 1} \operatorname{Area} (\fol{\Sigma}{y}, \fol{h}{y}) = 4 \pi \text{,} \qquad \lim_{y \nearrow 1} \| \fol{\mc{K}}{y} - 1 \|_{ H^{-1/2}_x (\fol{\Sigma}{y}, \fol{h}{y}) } = 0 \text{,} \end{equation}
where $\fol{\mc{K}}{y}$ denotes the Gauss curvature of $(\fol{\Sigma}{y}, \fol{h}{y})$.

\item \emph{In the original $t$-foliation}, the $\fol{v}{y}$'s satisfy the following uniform bounds:
\begin{align}
\label{eq.distortion_bound} \| \nabla_t \nabla^2 \fol{v}{y} \|_{ L^{2, 2}_{t, x} } + \| \nabla_t \nabla \fol{v}{y} \|_{ L^{\infty, 2}_{x, t} } &\lesssim \Gamma \text{,} \\
\notag \| \nabla^2 \fol{v}{y} \|_{ B^{\infty, 0}_{t, x} \cap L^{2, \infty}_{x, t} } + \| \nabla \fol{v}{y} \|_{ L^{\infty, \infty}_{t, x} } + \| \fol{v}{y} \|_{ L^{\infty, \infty}_{t, x} } &\lesssim \Gamma \text{,}
\end{align}

\item Fix arbitrary exponents $1 \leq p < 2$ and $2 \leq q < \infty$.
Then, \emph{in the original $t$-foliation}, the $\fol{v}{y}$'s satisfy the following convergence properties:
\begin{align}
\label{eq.distortion_conv} \lim_{y_1, y_2 \nearrow 1} [ \| \nabla_t \nabla^2 ( \fol{v}{y_2} - \fol{v}{y_1} ) \|_{ L^{p, 2}_{x, t} } + \| \nabla_t \nabla ( \fol{v}{y_2} - \fol{v}{y_1} ) \|_{ L^{q, 2}_{x, t} } ] &= 0 \text{,} \\
\notag \lim_{y_1, y_2 \nearrow 1} [ \| \nabla^2 ( \fol{v}{y_2} - \fol{v}{y_1} ) \|_{ L^{p, \infty}_{x, t} } + \| \nabla ( \fol{v}{y_2} - \fol{v}{y_1} ) \|_{ L^{q, \infty}_{x, t} } + \| \fol{v}{y_2} - \fol{v}{y_1} \|_{ L^{\infty, \infty}_{t, x} } ] &= 0 \text{.}
\end{align}

\item In addition, the following convergence property holds:
\begin{equation} \label{eq.distortion_conv_ex} \lim_{ y_1, y_2 \nearrow 1 } \| \nabla^2 ( \fol{v}{y_2} - \fol{v}{y_1} ) \|_{ L^{2, \infty}_{x, t} ( \fol{\Sigma}{y_1}, S_1 ) } = 0 \text{.} \end{equation}
\end{itemize}
\end{lemma}

Returning to the condition \eqref{rounding} in Theorem \ref{the.thm}, we must consider the \emph{physical} metric on $\fol{\Sigma}{y}$, i.e., the level set $\fol{s}{y} = (1 - y)^{-1}$.
Letting $\mind_{\fol{\Sigma}{y}}$ be the induced metric on $\fol{\Sigma}{y}$, and $r_{\fol{\Sigma}{y}}$ its area radius, then the first part of \eqref{eq.asympt_round} implies that
\begin{equation} \label{eq.asympt_round_cor} \lim_{y \nearrow 1} \frac{ \operatorname{Area} (\fol{\Sigma}{y}, \mind_{\fol{\Sigma}{y}} ) }{ ( \fol{s}{y} )^2 } = 4 \pi \text{,} \qquad \lim_{y \nearrow 1} \frac{ r_{\fol{\Sigma}{y}} }{ \fol{s}{y} } = 1 \text{.} \end{equation}
Moreover, since
\[ r_{\fol{\Sigma}{y}}^{-2} \cdot \mind_{\fol{\Sigma}{y}} = (\fol{s}{y})^2 r_{\fol{\Sigma}{y}}^{-2} \cdot \fol{h}{y} \text{,} \qquad r_{\fol{\Sigma}{y}}^2 \cdot \mc{K} (\fol{\Sigma}{y}, \mind_{\fol{\Sigma}{y}}) = r_{\fol{\Sigma}{h}}^2 ( \fol{s}{y} )^{-2} \cdot \mc{K} (\fol{\Sigma}{y}, \fol{h}{y}) \text{,} \]
then it follows from \eqref{eq.asympt_round} and \eqref{eq.asympt_round_cor} that
\[ \lim_{y \nearrow 1} \| r_{\fol{\Sigma}{y}}^2 \mc{K} (\fol{\Sigma}{y}, \mind_{\fol{\Sigma}{y}}) - 1 \|_{ H^{-1/2}_x (\fol{\Sigma}{y}, r_{\fol{\Sigma}{y}}^{-2} \cdot \mind_{\fol{\Sigma}{y}}) } = 1 \text{.} \]
This proves the asymptotic roundness property of \eqref{rounding}.

\subsubsection{Convergence of Physical Quantities}

It remains to show that the physically relevant quantities---namely, $\fol{M}{y}$, $\fol{Z}{y}$, and $\fol{\ul{H}}{y}$ on $\fol{\Sigma}{y}$---converge at infinity; this is the second main component of the proof of Theorem \ref{the.thm}.
The basic ideas are simple, as they are analogous to that of establishing limits at infinity in a single foliation of $\mc{N}$ as a consequence of Theorem \ref{thm.nc_renorm} (see Section \ref{sect:prelim.ncinf} and \cite[Cor. 5.2]{alex_shao:nc_inf} for details).
In practice, though, the process is complicated by the fact that we are now comparing objects from different foliations of $\mc{N}$.
\footnote{Again, much of the technical difficulty arises from the need to work with a \emph{$1$-parameter family} of changes of foliations.
See the discussion in Section \ref{sect:intro.outline} for details.}

The results of this part of the proof of Theorem \ref{the.thm} are summarized in the subsequent lemma, which will be proved in Section \ref{sect:convergence}.

\begin{lemma} \label{thm.convergence}
Let the $\fol{v}{y}$'s and $\fol{\Sigma}{y}$'s, $y \in [0, 1)$, be from Lemma \ref{thm.distortion}.
\footnote{In fact, we only require a family of $\fol{v}{y}$'s satisfying \eqref{eq.distortion_bound}-\eqref{eq.distortion_conv_ex}; the asymptotic roundness property \eqref{eq.asympt_round} plays no role in the proof of Lemma \ref{thm.convergence}.}
As before, consider the $\fol{s}{y}$- and $\fol{t}{y}$-foliations of $\mc{N}$ with respect to the $\fol{v}{y}$'s.
Then:
\begin{itemize}
\item The renormalized mass aspect functions $\fol{M}{y}$, restricted to $(\fol{\Sigma}{y}, \fol{h}{y})$, converge in $L^1_x$ to a limit $M^1$ as $y \nearrow 1$, in the sense of Definition \ref{def.conv_inf}.
Moreover,
\footnote{Recall $\Phi_{S_1}$, defined in Section \ref{sect:prelim.ncinf}, is the transport map along the null generators of $\mc{N}$.}
\begin{equation} \label{eq.limit_M} \| \Phi_{S_1}^\ast (M^1) \|_{ L^1_x (S_0) } \lesssim \Gamma \text{.} \end{equation}

\item The renormalized torsions $\fol{Z}{y}$ and conjugate null second fundamental forms $\fol{\ul{H}}{y}$, restricted to $(\fol{\Sigma}{y}, \fol{h}{y})$, converge in $L^2_x$ to limits $Z^1$ and $\ul{H}^1$, respectively, as $y \nearrow 1$, in the  sense of Definition \ref{def.conv_inf}.
Furthermore,
\begin{equation} \label{eq.limit_ZHbar} \| \Phi_{S_1}^\ast (Z^1) \|_{ L^2_x (S_0) } \lesssim \Gamma \text{,} \qquad \| \Phi_{S_1}^\ast (\ul{H}^1) \|_{ L^2_x (S_0) } \lesssim \Gamma \text{.} \end{equation}
\end{itemize}
\end{lemma}

The conclusions \eqref{eq.limit_M} and \eqref{eq.limit_ZHbar} can be connected to \eqref{conclns.bondi}-\eqref{conclns.loss} by inverting the renormalization process.
First of all, by letting $\vind_{\fol{\Sigma}{y}}$ and $\fol{\omega}{y}$ denote the volume forms associated with $\mind_{\fol{\Sigma}{y}}$ and $\fol{h}{y}$, respectively, we obtain
\begin{align*}
| m_{\rm H} (\fol{\Sigma}{y}) - m_{\rm S} | &\leq \int_{ \fol{\Sigma}{y} } \left| \frac{ r_{\fol{\Sigma}{y}} }{8 \pi} \fol{\mu}{y} - \frac{1}{ 4 \pi r_{\fol{\Sigma}{y}}^2 } m_{\rm S} \right| d \vind_{ \fol{\Sigma}{y} } \\
&= \int_{ \fol{\Sigma}{y} } \left| \frac{ r_{\fol{\Sigma}{y}} }{ 8 \pi \fol{s}{y} } \fol{M}{y} + \frac{ r_{\fol{\Sigma}{y}} }{ 4 \pi \fol{s}{y} } m_{\rm S} - \frac{ (\fol{s}{y})^2 }{ 4 \pi r_{\fol{\Sigma}{y}}^2 } m_{\rm S} \right| d \fol{\omega}{y} \text{.}
\end{align*}
Recall that by \eqref{eq.norm_compare} and Proposition \ref{thm.renorm_cfol}, integrals with respect to $\fol{h}{y}$ and $\gamma_0$ differ very little.
Thus, combining the above with \eqref{eq.asympt_round_cor} and \eqref{eq.limit_M} results in \eqref{conclns.bondi}.

Finally, the second limit in \eqref{eq.asympt_round_cor} implies that the limits of the $\fol{Z}{y}$'s and $\fol{\ul{H}}{y}$'s along the $\fol{\Sigma}{y}$'s coincide with the quantities $\mc{Z}$ and $\Xi$ at $S_1$ from \eqref{ang.mom} and \eqref{mass.loss}.
Thus, \eqref{conclns.loss} and \eqref{conclns.ang} follow immediately from \eqref{eq.limit_ZHbar}.

This completes the proof of Theorem \ref{the.thm}.

\section{Convergence Estimates} \label{sect:convergence}

This section is dedicated to the proof of Lemma \ref{thm.convergence}, the second main component in the proof of Theorem \ref{the.thm}.
We show that given a family $\fol{v}{y}$, $y \in [0, 1)$, of distortion functions, corresponding to changes of foliations, satisfying the properties \eqref{eq.distortion_bound}-\eqref{eq.distortion_conv_ex}, the physically relevant quantities---$\fol{M}{y}$ and $\fol{\ul{H}}{y}$, restricted to $\fol{\Sigma}{y} = \fol{S}{y}_y$---converge in the appropriate norms as $y \nearrow 1$.
As mentioned before, these limits are related to the Bondi energy and the rate of mass loss.

\subsection{Difference Estimates} \label{sect:convergence.diff}

Assume now the hypotheses of Lemma \ref{thm.convergence}.
The main analytical tools we will need in proving Lemma \ref{thm.convergence} are the following:
\begin{itemize}
\item Estimates for connection and curvature quantities in the $\fol{t}{y}$-foliations, uniform in $y$ and in terms of the original foliation.

\item Cauchy estimates for the \emph{differences} between corresponding connection and curvature quantities in two \emph{different} refoliations of $\mc{N}$.
\end{itemize}
Note the first class of estimates are consequences of Proposition \ref{thm.renorm_cfol_bound}.
The upcoming development hence focuses on the remaining difference estimates.
The techniques involved are analogous to those in the proof of Proposition \ref{thm.renorm_cfol_bound}, with the main difference being that we must compare two refoliations of $\mc{N}$ simultaneously.

Before stating and proving the main estimates, let us first clarify the meaning of these aforementioned differences.
Fix $y_1, y_2 \in [0, 1)$, and consider two tensor fields $\fol{\Psi}{y_1}$ and $\fol{\Psi}{y_2}$, horizontal with respect to the $\fol{t}{y_1}$- and $\fol{t}{y_2}$-foliations, respectively.
From the development in Section \ref{sect:prelim.cfol}, we can, using \eqref{eq.cf_X}, consider both $\fol{\Psi}{y_1}$ and $\fol{\Psi}{y_2}$ as \emph{$t$-horizontal} tensor fields.
With this identification, we can now make sense of the difference $\fol{\Psi}{y_2} - \fol{\Psi}{y_1}$, as a $t$-horizontal field.
One important example will be the difference $\fol{\ul{H}}{y_2} - \fol{\ul{H}}{y_1}$ between corresponding $\ul{H}$'s in two refoliations.

With the above conventions in mind, we state our main Cauchy estimates:

\begin{lemma} \label{thm.limits_lem}
The following Cauchy properties hold:
\footnote{Recall $\fol{\epsilon}{y}_t$ represents the volume form associated with $\fol{\gamma}{y}_t$.}
\begin{align}
\label{eq.limits_lem_conv_pre} \lim_{y_1, y_2 \nearrow 1} \| \fol{t}{y_2} - \fol{t}{y_1} \|_{ L^{\infty, \infty}_{t, x} } &= 0 \text{,} \\
\notag \lim_{y_1, y_2 \nearrow 1} [ \| \fol{\gamma}{y_2} - \fol{\gamma}{y_1} \|_{ L^{\infty, \infty}_{t, x} } + \| \fol{\epsilon}{y_2} - \fol{\epsilon}{y_1} \|_{ L^{\infty, \infty}_{t, x} } ] &= 0 \text{,} \\
\notag \lim_{y_1, y_2 \nearrow 1} [ \| \fol{(\gamma^{-1})}{y_2} - \fol{(\gamma^{-1})}{y_1} \|_{ L^{\infty, \infty}_{t, x} } + \| \fol{(\epsilon^{-1})}{y_2} - \fol{(\epsilon^{-1})}{y_1} \|_{ L^{\infty, \infty}_{t, x} } ] &= 0 \text{.}
\end{align}
In the above, $\fol{(\gamma^{-1})}{y}$ and $\fol{(\epsilon^{-1})}{y}$ refer to the contravariant metric duals of $\fol{\gamma}{y}$ and $\fol{\epsilon}{y}$, respectively.
Furthermore, given any $1 \leq p < 2$ and $2 \leq q < \infty$, we have that
\begin{align}
\label{eq.limits_lem_conv} \lim_{y_1, y_2 \nearrow 1} [ \| \fol{H}{y_2} - \fol{H}{y_1} \|_{ L^{q, 2}_{x, t} \cap L^{2, \infty}_{x, t} } + \| \fol{Z}{y_2} - \fol{Z}{y_1} \|_{ L^{q, 2}_{x, t} \cap L^{2, \infty}_{x, t} } ] &= 0 \text{,} \\
\notag \lim_{y_1, y_2 \nearrow 1} \| \fol{\ul{H}}{y_2} - \fol{\ul{H}}{y_1} \|_{ L^{p, \infty}_{x, t} } &= 0 \text{,} \\
\notag \lim_{y_1, y_2 \nearrow 1} [ \| \fol{A}{y_2} - \fol{A}{y_1} \|_{ L^{p, 2}_{x, t} } + \| \fol{B}{y_2} - \fol{B}{y_1} \|_{ L^{p, 2}_{x, t} } + \| \fol{R}{y_2} - \fol{R}{y_1} \|_{ L^{p, 2}_{x, t} } ] &= 0 \text{,} \\
\notag \lim_{y_1, y_2 \nearrow 1} [ \| \fol{\nabla}{y_2} \fol{H}{y_2} - \fol{\nabla}{y_1} \fol{H}{y_1} \|_{ L^{p, 2}_{x, t} } + \| \fol{\nabla}{y_2} \fol{Z}{y_2} - \fol{\nabla}{y_1} \fol{Z}{y_1} \|_{ L^{p, 2}_{x, t} } ] &= 0 \text{.}
\end{align}
Also, recalling that $\fol{\Sigma}{y} = \fol{S}{y}_y$, we have the refined limit
\begin{equation} \label{eq.limits_lem_conv_ex} \lim_{y_1, y_2 \nearrow 1} \| \fol{\ul{H}}{y_2} - \fol{\ul{H}}{y_1} \|_{ L^{2, \infty}_{x, t} (\fol{\Sigma}{y_1}, S_1) } = 0 \text{.} \end{equation}
\end{lemma}

\begin{remark}
All the estimates in Lemma \ref{thm.limits_lem} are in terms of the $t$-foliation.
\end{remark}

\begin{remark}
Note that if $y_1$ and $y_2$ are sufficiently close to $1$, so that $| \fol{v}{y_2} - \fol{v}{y_1} |$ is small, then for any integer $k$ satisfying $|k| \leq 3$, we have
\begin{equation} \label{eq.renorm_cfol_exp_diff} | e^{k \fol{v}{y_2}} - e^{ k \fol{v}{y_1} } | \lesssim | \fol{v}{y_2} - \fol{v}{y_1} | \text{.} \end{equation}
Moreover, if $\fol{\mc{B}}{y}_k$ and $\fol{\mc{C}}{y}_k$ denote the coefficients in \eqref{eq.Ck}, with $v$ replaced by $\fol{v}{y}$, then
\begin{equation} \label{eq.renorm_cfol_C_diff} | \fol{\mc{B}}{y_2}_k - \fol{\mc{B}}{y_1}_k | \lesssim | \fol{v}{y_2} - \fol{v}{y_1} | \text{,} \qquad | \fol{\mc{C}}{y_2}_k - \fol{\mc{C}}{y_1}_k | \lesssim | \fol{v}{y_2} - \fol{v}{y_1} | \text{.} \end{equation}
We will use these observations repeatedly throughout the proof of Lemma \ref{thm.limits_lem}.
\end{remark}

For convenience, we will also adopt in the proof of Lemma \ref{thm.limits_lem} the abbreviations
\[ \mf{d} \Psi = \fol{\Psi}{y_2} - \fol{\Psi}{y_1} \text{,} \qquad \mf{d} \nabla \Psi = \fol{\nabla}{y_2} \fol{\Psi}{y_2} - \fol{\nabla}{y_1} \fol{\Psi}{y_1} \text{.} \]
For example, by these conventions,
\[ \mf{d} H = \fol{H}{y_2} - \fol{H}{y_1} \text{,} \qquad \nabla \mf{d} v = \nabla \fol{v}{y_2} - \nabla \fol{v}{y_1} \text{,} \qquad \mf{d} \nabla Z = \fol{\nabla}{y_2} \fol{Z}{y_2} - \fol{\nabla}{y_1} \fol{Z}{y_1} \text{.} \]

\subsubsection{Proof of Lemma \ref{thm.limits_lem}}

Throughout, we can assume both $y_1$ and $y_2$ to be sufficiently close to $1$.
The key is to use the change of foliation formulas of Appendix \ref{sect:cfol.renorm}, along with \eqref{eq.renorm_est} and the limits \eqref{eq.distortion_bound}-\eqref{eq.distortion_conv_ex}.

Recalling \eqref{eq.cf_t} and applying \eqref{eq.distortion_bound}, \eqref{eq.distortion_conv}, \eqref{eq.renorm_cfol_exp_diff}, and \eqref{eq.renorm_cfol_C_diff}, we estimate
\[ \lim_{y_1, y_2 \nearrow 1} \| \mf{d} t \|_{ L^{\infty, \infty}_{t, x} } \lesssim \lim_{y_1, y_2 \nearrow 1} \| \mf{d} v \|_{ L^{\infty, \infty}_{t, x} } = 0 \text{.} \]
Similarly, for the metric, we can use \eqref{eq.cf_gammainv} in order to write
\[ ( \fol{\gamma}{y_2} )_{ab} - ( \fol{\gamma}{y_1} )_{ab} = f_1 \cdot \gamma_{ab} \text{,} \qquad ( \fol{\gamma}{y_2} )^{ab} - ( \fol{\gamma}{y_1} )^{ab} = f_{-1} \cdot \gamma^{ab} \text{,} \]
where $f_1$ and $f_{-1}$ are scalar functions on $\mc{N}$ satisfying $| f_1 | \lesssim | \mf{d} v |$ and $| f_{-1} | \lesssim | \mf{d} v |$.
As a result, the limits for $\mf{d} \gamma$ and $\mf{d} \gamma^{-1}$ follow.
As the estimates for the volume forms are analogous, this completes the proof of \eqref{eq.limits_lem_conv_pre}.

The proof for \eqref{eq.limits_lem_conv} is similar.
First, we expand $\fol{Z}{y_1}$ and $\fol{Z}{y_2}$'s using \eqref{eq.cf_Z}, and we note that the only difference between these expansions is the presence of $\fol{v}{y_1}$ and $\fol{v}{y_2}$.
Thus, each term in the expansion of $\mf{d} Z$ must be a product of the following:
\begin{itemize}
\item A difference of $v$: either $\mf{d} v$ or $\nabla \mf{d} v$.

\item A quantity in the original foliation, i.e., either $Z$ or $H$.

\item Instances of $\fol{v}{y_1}$ and $\fol{v}{y_2}$, not appearing as a difference.
\end{itemize}
These isolated instances of $\fol{v}{y_i}$'s can be controlled using \eqref{eq.distortion_bound}.
To be more specific, a more careful look at \eqref{eq.cf_Z}, along with H\"older's inequality, yields
\begin{align*}
\| \mf{d} Z \|_{ L^{q, 2}_{x, t} \cap L^{2, \infty}_{x, t} } &\lesssim \| \mf{d} v \|_{ L^{\infty, \infty}_{t, x} } \| Z \|_{ L^{\infty, 2}_{x, t} \cap L^{4, \infty}_{x, t} } + ( \| \nabla \mf{d} v \|_{ L^{q^\prime, \infty}_{x, t} } + \| \mf{d} v \|_{ L^{\infty, \infty}_{t, x} } ) \\
&\qquad + ( \| \nabla \mf{d} v \|_{ L^{q^\prime, \infty}_{x, t} } + \| \mf{d} v \|_{ L^{\infty, \infty}_{t, x} } ) \| H \|_{ L^{\infty, 2}_{x, t} \cap L^{4, \infty}_{x, t} } \text{,}
\end{align*}
where $2 < q^\prime < \infty$ is sufficiently large.
Recalling \eqref{eq.renorm_est} and \eqref{eq.distortion_conv} results in the limit for $\mf{d} Z$ in \eqref{eq.limits_lem_conv}.
The limit for $\mf{d} H$ is proved similarly using \eqref{eq.cf_H}, but is easier.

For $\mf{d} \ul{H}$, we apply \eqref{eq.cf_Hbar} and \eqref{eq.distortion_bound} to obtain, with $q^\prime$ as before,
\begin{align*}
\| \mf{d} \ul{H} \|_{ L^{p, \infty}_{x, t} } &\lesssim \| \mf{d} v \|_{ L^{\infty, \infty}_{t, x} } \| \ul{H} \|_{ L^{2, \infty}_{x, t} } + ( \| \nabla^2 \mf{d} v \|_{ L^{p, \infty}_{x, t} } + \| \nabla \mf{d} v \|_{ L^{q^\prime, \infty}_{x, t} } + \| \mf{d} v \|_{ L^{\infty, \infty}_{t, x} } ) \\
&\qquad + ( \| \nabla \mf{d} v \|_{ L^{q^\prime, \infty}_{x, t} } + \| \mf{d} v \|_{ L^{\infty, \infty}_{t, x} } ) ( \| Z \|_{ L^{4, \infty}_{x, t} } + \| H \|_{ L^{4, \infty}_{x, t} } ) \text{.}
\end{align*}
The curvature coefficients can be similarly bounded.
For example, by \eqref{eq.cf_B},
\begin{align*}
\| \mf{d} B \|_{ L^{p, 2}_{x, t} } &\lesssim \| \mf{d} v \|_{ L^{\infty, \infty}_{t, x}} \| B \|_{ L^{2, 2}_{t, x} } + ( \| \nabla \mf{d} v \|_{ L^{q^\prime, \infty}_{x, t} } + \| \mf{d} v \|_{ L^{\infty, \infty}_{t, x} } ) \| A \|_{ L^{2, 2}_{t, x} } \text{.}
\end{align*}
Analogous bounds can be derived for $A$ and $R$.
Applying \eqref{eq.renorm_est} and \eqref{eq.distortion_conv} in the same manner as before, we obtain the desired limits for $\ul{H}$, $A$, $B$, and $R$.

It remains to establish the limits for $\mf{d} \nabla H$ and $\mf{d} \nabla Z$; we prove the latter here, as the former is similar but easier.
By \eqref{eq.cf_nabla}, for sufficiently large $2 < q^\prime < \infty$,
\begin{align*}
\| \mf{d} \nabla Z \|_{ L^{p, 2}_{x, t} } &\lesssim \| \nabla \mf{d} Z \|_{ L^{p, 2}_{x, t} } + \| \nabla \fol{v}{y_1} \|_{ L^{\infty, \infty}_{t, x} } \| \nabla_t \mf{d} Z \|_{ L^{p, 2}_{x, t} } + \| \nabla \mf{d} v \|_{ L^{q^\prime, \infty}_{x, t} } \| \nabla_t \fol{Z}{y_2} \|_{ L^{2, 2}_{x, t} } \\
&\qquad + \| \nabla \fol{v}{y_1} \|_{ L^{\infty, \infty}_{t, x} } ( 1 + \| H \|_{ L^{\infty, 2}_{x, t} } ) \| \mf{d} Z \|_{ L^{2, \infty}_{x, t} } \\
&\qquad + \| \nabla \mf{d} v \|_{ L^{q^\prime, \infty}_{x, t} } ( 1 + \| H \|_{ L^{\infty, 2}_{x, t} } ) \| \fol{Z}{y_2} \|_{ L^{4, \infty}_{x, t} } \text{.}
\end{align*}
From \eqref{eq.renorm_est}, \eqref{eq.renorm_cfol_bound}, and \eqref{eq.distortion_bound}, we see that the last two terms on the right-hand side tend to zero as $y_1, y_2 \nearrow 1$.
As a result, we need only prove that
\begin{equation} \label{eql.limits_lem_final} \lim_{y_1, y_2 \nearrow 1} [ \| \nabla \mf{d} Z \|_{ L^{p, 2}_{x, t} } + \| \nabla_t \mf{d} Z \|_{ L^{p, 2}_{x, t} } ] = 0 \text{,} \qquad \| \nabla_t \fol{Z}{y_2} \|_{ L^{2, 2}_{t, x} } \lesssim \Gamma \text{.} \end{equation}

For this, we again expand $\mf{d} Z$ using \eqref{eq.cf_Z}, and we apply $\nabla$ and $\nabla_t$ to the result.
Applying H\"older's inequality and \eqref{eq.distortion_bound} to eliminate isolated $\fol{v}{y_i}$'s, we have
\begin{align*}
\| \nabla \mf{d} Z \|_{ L^{p, 2}_{x, t} } &\lesssim \| \mf{d} v \|_{ L^{\infty, \infty}_{t, x} } \| \nabla Z \|_{ L^{2, 2}_{t, x} } + ( \| \nabla \mf{d} v \|_{ L^{q^\prime, \infty}_{x, t} } + \| \mf{d} v \|_{ L^{\infty, \infty}_{t, x} } ) \| Z \|_{ L^{\infty, 2}_{x, t} } \\
&\qquad + ( \| \nabla \mf{d} v \|_{ L^{q^\prime, \infty}_{x, t} } + \| \mf{d} v \|_{ L^{\infty, \infty}_{t, x} } ) \| \nabla H \|_{ L^{2, 2}_{t, x} } \\
&\qquad + ( \| \nabla^2 \mf{d} v \|_{ L^{p, 2}_{x, t} } + \| \nabla \mf{d} v \|_{ L^{q^\prime, \infty}_{x, t} } + \| \mf{d} v \|_{ L^{\infty, \infty}_{t, x} } ) \| H \|_{ L^{\infty, 2}_{x, t} } \\
&\qquad + \| \nabla^2 \mf{d} v \|_{ L^{p, \infty}_{x, t} } + \| \nabla \mf{d} v \|_{ L^{q^\prime, \infty}_{x, t} } + \| \mf{d} v \|_{ L^{\infty, \infty}_{t, x} } \text{,} \\
\| \nabla_t \mf{d} Z \|_{ L^{p, 2}_{x, t} } &\lesssim \| \mf{d} v \|_{ L^{\infty, \infty}_{t, x} } \| \nabla_t Z \|_{ L^{2, 2}_{t, x} } + ( \| \nabla \mf{d} v \|_{ L^{q^\prime, \infty}_{x, t} } + \| \mf{d} v \|_{ L^{\infty, \infty}_{t, x} } ) \| \nabla_t H \|_{ L^{2, 2}_{t, x} } \\
&\qquad + ( \| \nabla_t \nabla \mf{d} v \|_{ L^{q^\prime, 2}_{x, t} } + \| \mf{d} v \|_{ L^{\infty, \infty}_{t, x} } ) \| H \|_{ L^{4, \infty}_{x, t} } \\
&\qquad + \| \nabla_t \nabla \mf{d} v \|_{ L^{q^\prime, 2}_{x, t} } + \| \mf{d} v \|_{ L^{\infty, \infty}_{t, x} } \text{.}
\end{align*}
Applying \eqref{eq.renorm_est} and \eqref{eq.distortion_conv} to the above proves the limit in \eqref{eql.limits_lem_final}.
The remaining bound for $\nabla_t \fol{Z}{y_2}$ follows from a similar estimate as the above for $\nabla_t \mf{d} Z$.
This completes the proof of \eqref{eql.limits_lem_final}, and hence \eqref{eq.limits_lem_conv}.

Finally, for \eqref{eq.limits_lem_conv_ex}, we proceed like the estimate for $\mf{d} \ul{H}$ in \eqref{eq.limits_lem_conv}:
\begin{align*}
\| \mf{d} \ul{H} \|_{ L^{2, \infty}_{x, t} (\fol{\Sigma}{y_1}, S_1) } &\lesssim \| \mf{d} v \|_{ L^{\infty, \infty}_{t, x} } ( 1 + \| \ul{H} \|_{ L^{2, \infty}_{x, t} } ) + ( \| \nabla^2 \mf{d} v \|_{ L^{2, \infty}_{x, t} (\fol{\Sigma}{y_1}, S_1) } + \| \nabla \mf{d} v \|_{ L^{q^\prime, \infty}_{x, t} } ) \\
&\qquad + ( \| \nabla \mf{d} v \|_{ L^{q^\prime, \infty}_{x, t} } + \| \mf{d} v \|_{ L^{\infty, \infty}_{t, x} } ) ( \| Z \|_{ L^{4, \infty}_{x, t} } + \| H \|_{ L^{4, \infty}_{x, t} } ) \text{.}
\end{align*}
The only difference is we restricted the norms to the region above $\fol{\Sigma}{y_1}$.
Although the second term on the right-hand side can no longer be controlled using \eqref{eq.distortion_conv}, because of the restriction to the shrinking region, this term will still go to zero as $y_1, y_2 \nearrow 1$ due to \eqref{eq.distortion_conv_ex}.
Thus, we have established \eqref{eq.limits_lem_conv_ex}, and hence Lemma \ref{thm.limits_lem}.

\subsection{The Bondi Energy} \label{sect:convergence.bondi}

We are now ready to establish the limit \eqref{eq.limit_M} involving the Hawking masses.
This is the most difficult limit, since we lack a tidy formula for how $M$ transforms under changes of foliations.
\footnote{Although such a formula can be derived, it is easier to avoid doing so.}
To work around this, we observe that a tidy transformation formula for $M$ does exist \emph{at $S_0$}; see \eqref{eq.cf_M_init}.

As usual, let $\fol{M}{y}_y$ denote the restriction of $\fol{M}{y}$ to $\fol{\Sigma}{y} = \fol{S}{y}_y$.
Recalling Definition \ref{def.conv_inf}, to show that the $\fol{M}{y}_y$'s converge in $L^1_x$, it suffices to prove
\begin{equation} \label{eq.limit_goal_M} \lim_{y_1, y_2 \nearrow 1} {\bf L}_M = 0 \text{,} \qquad {\bf L}_M = \int_{ S_0 } | \Phi_{ \fol{\Sigma}{y_2} }^\ast ( \fol{M}{y_2}_{y_2} ) - \Phi_{ \fol{\Sigma}{y_1} }^\ast ( \fol{M}{y_1}_{y_1} ) | d \epsilon_0 \text{.} \end{equation}
To convert \eqref{eq.limit_goal_M} into estimates that we have, we resort to the following lemma:

\begin{lemma} \label{thm.limit_lem_M}
The following estimate holds:
\begin{align}
\label{eq.limit_lem_M} {\bf L}_M &\lesssim \| \fol{\nabla}{y_2}_{ \fol{t}{y_2} } \fol{M}{y_2} - \fol{\nabla}{y_1}_{ \fol{t}{y_1} } \fol{M}{y_1} \|_{ L^{1, 1}_{t, x} } + \| \fol{\nabla}{y_2}_{ \fol{t}{y_2} } \fol{M}{y_2} \|_{ L^{1, 1}_{\fol{t}{y_2}, x} (\fol{\Sigma}{y_1}, \fol{\Sigma}{y_2}) } \\
\notag &\qquad + \| \fol{M}{y_2} - \fol{M}{y_1} \|_{ L^1_x (S_0) } \text{.}
\end{align}
\end{lemma}

\begin{proof}
See Appendix \ref{sect:estimates.limit}.
\end{proof}

Since the $\fol{\Sigma}{y}$'s are going to infinity, then by Proposition \ref{thm.renorm_cfol},
\[ \lim_{y_1, y_2 \nearrow 1} \| \fol{\nabla}{y_2}_{ \fol{t}{y_2} } \fol{M}{y_2} \|_{ L^{1, 1}_{\fol{t}{y_2}, x} (\fol{\Sigma}{y_1}, \fol{\Sigma}{y_2}) } = 0 \text{.} \]
Moreover, using \eqref{eq.cf_M_init} and \eqref{eq.distortion_conv},
\[ \lim_{y_1, y_2 \nearrow 1} \| \fol{M}{y_2} - \fol{M}{y_1} \|_{ L^1_x (S_0) } = \lim_{y_1, y_2 \nearrow 1} \| \lapl (v_{y_2} - v_{y_1}) \|_{ L^1_x (S_0) } = 0 \text{.} \]

Suppose in addition that
\begin{equation} \label{evol.M.corr} \lim_{y_1, y_2 \nearrow 1} \| \fol{\nabla}{y_2}_{ \fol{t}{y_2} } \fol{M}{y_2} - \fol{\nabla}{y_1}_{ \fol{t}{y_1} } \fol{M}{y_1} \|_{ L^{1, 1}_{t, x} } = 0 \text{.} \end{equation}
From this, we obtain \eqref{eq.limit_goal_M}, hence the $\fol{M}{y}_y$'s have a limit in $L^1_x$ as $y \nearrow 1$.
Since Theorem \ref{thm.nc_renorm} applies to every $\fol{v}{y}$-foliation of $\mc{N}$, courtesy of Proposition \ref{thm.renorm_cfol}, then \eqref{eq.limit_M} follows.
\footnote{Recall that norms with respect to different (renormalized) metrics and different foliations of $\mc{N}$ remain comparable due to the discussions in Sections \ref{sect:prelim.ncinf} and \ref{sect:prelim.cfol.small}.}
As a result, it remains only to establish \eqref{evol.M.corr}.

\subsubsection{Proof of \eqref{evol.M.corr}}

The strategy is to work with the evolution equation below for $\fol{M}{y}$, given in \cite[Proposition 4.2]{alex_shao:nc_inf} and valid for any $\fol{t}{y}$-foliation:
\begin{align}
\label{eq.structure_renorm_evd} \fol{\nabla}{y}_{ \fol{t}{y} } \fol{M}{y} &= - \frac{3}{2} ( \fol{\trace}{y} \fol{H}{y} ) \paren{ \fol{M}{y} + 2 m_{\rm S}} - 2 (1 - \fol{t}{y}) ( \fol{\gamma}{y} )^{ab} \fol{Z}{y}_a \fol{B}{y}_b \\
\notag &\qquad + 2 ( \fol{\gamma}{y} )^{ab} ( \fol{\gamma}{y} )^{cd} \fol{\hat{H}}{y}_{ac} \fol{\nabla}{y}_b \fol{Z}{y}_d - 2 (1 - \fol{t}{y}) ( \fol{\gamma}{y} )^{ab} ( \fol{\gamma}{y} )^{cd} \fol{\hat{H}}{y}_{ac} \fol{Z}{y}_b \fol{Z}{y}_d \\
\notag &\qquad + 2 ( \fol{\gamma}{y} )^{ab} \fol{Z}{y}_b \fol{\nabla}{y}_a ( \fol{\trace}{y} \fol{H}{y} ) + \frac{3}{2} ( \fol{\gamma}{y} )^{ab} [ (1 - \fol{t}{y}) \fol{\trace}{y} \fol{H}{y} + 2 ] \fol{Z}{y}_a \fol{Z}{y}_b \\
\notag &\qquad - \frac{1}{4} ( \fol{\gamma}{y} )^{ab} ( \fol{\gamma}{y} )^{cd} \brak{ \fol{\trace}{y} \fol{\ul{H}}{y} - 2 \paren{ 1 - \frac{2m_{\rm S}}{ \fol{s}{y} } } } \fol{\hat{H}}{y}_{ac} \fol{\hat{H}}{y}_{bd} \text{.}
\end{align}
Recalling the definition \eqref{eq.renorm_maf} and suppressing constant factors and all instances of $\fol{\gamma}{y}$, $\fol{s}{y}$, and $\fol{t}{y}$, we can rewrite \eqref{eq.structure_renorm_evd} schematically as
\begin{align}
\label{evol.again} \fol{\nabla}{y}_{ \fol{t}{y} } \fol{M}{y} &= \fol{H}{y} \cdot \fol{R}{y} + \fol{Z}{y} \cdot \fol{B}{y} + \fol{H}{y} \cdot \fol{\nabla}{y} \fol{Z}{y} + \fol{Z}{y} \cdot \fol{\nabla}{y} \fol{H}{y} + \fol{H}{y} \cdot \fol{Z}{y} \cdot \fol{Z}{y} \\
\notag &\qquad + \fol{H}{y} \cdot \fol{H}{y} \cdot \fol{\ul{H}}{y} + \fol{Z}{y} \cdot \fol{Z}{y} + \fol{H}{y} \cdot \fol{H}{y} + \fol{H}{y} \text{.}
\end{align}

For convenience, as in Section \ref{sect:convergence.diff}, we adopt the abbreviation
\footnote{Recall the conventions from Section \ref{sect:prelim.cfol} for identifying fields from different foliations of $\mc{N}$.}
\[ \mf{d} \nabla_t M = \fol{\nabla}{y_2}_{ \fol{t}{y_2} } \fol{M}{y_2} - \fol{\nabla}{y_1}_{ \fol{t}{y_1} } \fol{M}{y_1} \text{.} \]
As before, similar conventions will hold for other quantities, e.g., $\mf{d} H$ and $\mf{d} \nabla Z$.
We expand $\mf{d} \nabla_t M$ as the difference between the right-hand sides of \eqref{evol.again}, with $y = y_2$ and $y = y_1$.
Each term of this expansion will contain a factor that is a difference, e.g., $\mf{d} t$, $\mf{d} H$, $\mf{d} \nabla Z$.
These differences can be controlled using Lemma \ref{thm.limits_lem}.
The remaining factors can be controlled using Proposition \ref{thm.renorm_cfol_bound}.

For brevity, we adopt the following additional schematic notations:
\begin{itemize}
\item The symbol $\Psi^\prime$ will refer to any one of $\fol{\Psi}{y_1}$ or $\fol{\Psi}{y_2}$.
For example, we will apply this with $\Psi$ being $Z$, $B$, $\nabla H$, etc.

\item We use the symbol $\mc{A}$ to denote any one of $H$ or $Z$.

\item We use the symbol $\mc{R}$ to denote any one of $B$ and $R$.
\end{itemize}
With this in mind, the expansion of $\mf{d} \nabla_t M$ using \eqref{evol.again} yields the following bound:
\begin{align}
\label{subtract.evol} \| \mf{d} \nabla_t M \|_{ \mc{L}^{1, 1}_{t, x} } &\lesssim \| \mc{A}^\prime (\mf{d} \mc{R}) \|_{ L^{1, 1}_{t, x} } + \| (\mf{d} \mc{A}) \mc{R}^\prime \|_{ L^{1, 1}_{t, x} } + \| \mc{A}^\prime (\mf{d} \nabla \mc{A}) \|_{ L^{1, 1}_{t, x} } \\
\notag &\qquad + \| (\mf{d} \mc{A}) (\nabla \mc{A})^\prime \|_{ L^{1, 1}_{t, x} } + \| (\mf{d} \ul{H}) \mc{A}^\prime \mc{A}^\prime \|_{ L^{1, 1}_{t, x} } + \| (\mf{d} \mc{A}) \ul{H}^\prime \mc{A}^\prime \|_{ L^{1, 1}_{t, x} } \\
\notag &\qquad + \| (\mf{d} \mc{A}) \mc{A}^\prime \mc{A}^\prime \|_{ L^{1, 1}_{t, x} } + \| (\mf{d} \mc{A}) \mc{A}^\prime \|_{ L^{1, 1}_{t, x} } + \| \mf{d} \mc{A}^\prime \|_{ L^{1, 1}_{t, x} } + \mc{L} \text{,}
\end{align}
where $\mc{L}$ arises from terms involving either $\mf{d} t$ or $\mf{d} \gamma^{-1}$ (and hence no differences involving the Ricci or the curvature coefficients).
More specifically,
\begin{align*}
\mc{L} &= ( \| \mf{d} t \|_{ L^{\infty, \infty}_{t, x} } + \| \mf{d} \gamma^{-1} \|_{ L^{\infty, \infty}_{t, x} } ) \cdot \mc{L}^\prime \text{,} \\
\mc{L}^\prime &= \| \mc{A}^\prime \mc{R}^\prime \|_{ L^{1, 1}_{t, x} } + \| \mc{A}^\prime ( \nabla \mc{A} )^\prime \|_{ L^{1, 1}_{t, x} } + \| \ul{H}^\prime \mc{A}^\prime \mc{A}^\prime \|_{ L^{1, 1}_{t, x} } \\
&\qquad + \| \mc{A}^\prime \mc{A}^\prime \mc{A}^\prime \|_{ L^{1, 1}_{t, x} } + \| \mc{A}^\prime \mc{A}^\prime \|_{ L^{1, 1}_{t, x} } + \| \mc{A}^\prime \|_{ L^{1, 1}_{t, x} } \text{.}
\end{align*}

Using H\"older's inequality, Proposition \ref{thm.renorm_cfol_bound}, and \eqref{eq.limits_lem_conv_pre}, we obtain
\[ \lim_{y_1, y_2 \nearrow 1} \mc{L} = 0 \text{.} \]
As the estimates for these terms are straightforward and are easier than the remaining terms, we leave the details to the reader.

The other terms in \eqref{subtract.evol} are controlled similarly, using Proposition \ref{thm.renorm_cfol_bound} and Lemma \ref{thm.limits_lem}.
For instance, for the first two terms on the right-hand side of \eqref{subtract.evol}:
\begin{align*}
\lim_{y_1, y_2 \nearrow 1} \| \mc{A}^\prime (\mf{d} \mc{R}) \|_{ L^{1, 1}_{t, x} } &\lesssim \lim_{y_1, y_2 \nearrow 1} [ \| \mc{A}^\prime \|_{ L^{\infty, 2}_{x, t} } \| \mf{d} \mc{R} \|_{ L^{4/3, 2}_{x, t} } ] = 0 \text{,} \\
\lim_{y_1, y_2 \nearrow 1} \| (\mf{d} \mc{A}) \mc{R}^\prime \|_{ L^{1, 1}_{t, x} } &\lesssim \lim_{y_1, y_2 \nearrow 1} [ \| \mf{d} \mc{A} \|_{ L^{2, \infty}_{x, t} } \| \mc{R}^\prime \|_{ L^{2, 2}_{t, x} } ] = 0 \text{.}
\end{align*}
The third and fourth terms are bounded analogously.
For the next two terms:
\begin{align*}
\lim_{y_1, y_2 \nearrow 1} \| (\mf{d} \ul{H}) \mc{A}^\prime \mc{A}^\prime \|_{ L^{1, 1}_{t, x} } &\lesssim \lim_{y_1, y_2 \nearrow 1} [ \| \mf{d} \mc{H} \|_{ L^{4/3, \infty}_{x, t} } \| \mc{A}^\prime \|_{ L^{\infty, 2}_{x, t} }^2 ] = 0 \text{,} \\
\lim_{y_1, y_2 \nearrow 1} \| (\mf{d} \mc{A}) \ul{H}^\prime \mc{A}^\prime \|_{ L^{1, 1}_{t, x} } &\lesssim \lim_{y_1, y_2 \nearrow 1} [ \| \mf{d} \mc{A} \|_{ L^{2, \infty}_{x, t} } \| \ul{H}^\prime \|_{ L^{2, \infty}_{x, t} } \| \mc{A}^\prime \|_{ L^{\infty, 2}_{x, t} } ] = 0 \text{.}
\end{align*}
The remaining three terms are easier than the above and can be controlled in a similar manner.
This completes the proof of \eqref{evol.M.corr}.
  
\subsection{Rate of Mass Loss} \label{sect:convergence.remain}

Here, we complete the proof of Lemma \ref{thm.convergence} by establishing \eqref{eq.limit_ZHbar}.
As usual, let $\fol{Z}{y_i}_{y_i}$ and $\fol{\ul{H}}{y_i}_{y_i}$ refer to the restrictions of $\fol{Z}{y_i}$ and $\fol{\ul{H}}{y_i}$, respectively, to $\fol{\Sigma}{y_i}$.
By Definition \ref{def.conv_inf}, we must show
\begin{align}
\label{eq.limit_goal_ZHbar} \lim_{y_1, y_2 \nearrow 1} {\bf L}_Z &= \lim_{y_1, y_2 \nearrow 1} \int_{ S_0 } | \Phi_{ \fol{\Sigma}{y_2} }^\ast ( \fol{Z}{y_2}_{y_2} ) - \Phi_{ \fol{\Sigma}{y_1} }^\ast ( \fol{Z}{y_1}_{y_1} ) |^2 d \epsilon_0 = 0 \text{,} \\
\notag \lim_{y_1, y_2 \nearrow 1} {\bf L}_{\ul{H}} &= \lim_{y_1, y_2 \nearrow 1} \int_{ S_0 } | \Phi_{ \fol{\Sigma}{y_2} }^\ast ( \fol{\ul{H}}{y_2}_{y_2} ) - \Phi_{ \fol{\Sigma}{y_1} }^\ast ( \fol{\ul{H}}{y_1}_{y_1} ) |^2 d \epsilon_0 = 0 \text{.}
\end{align}
This is similar to the process in Section \ref{sect:convergence.bondi}, though it is easier due to \eqref{eq.cf_Z} and \eqref{eq.cf_Hbar}.
Because of these formulas, we can simplify the work by comparing $\fol{Z}{y_1}$ and $\fol{Z}{y_2}$ on $\fol{\Sigma}{y_1}$ rather than on $S_0$.
A similar argument holds as well for the $\fol{\ul{H}}{y_i}$'s.

For this, we use the following analogue of Lemma \ref{thm.limit_lem_M}:

\begin{lemma} \label{thm.limit_lem_ZHbar}
The following estimate holds:
\begin{align}
\label{eq.limit_lem_ZHbar} {\bf L}_Z &\lesssim \| \fol{Z}{y_2} - \fol{Z}{y_1} \|_{ L^{2, \infty}_{x, t} (\fol{\Sigma}{y_1}, S_1) }^2 + \| \fol{\nabla}{y_2}_{ \fol{t}{y_2} } \fol{Z}{y_2} \|_{ L^{2, 1}_{x, \fol{t}{y_2}} (\fol{\Sigma}{y_1}, \fol{\Sigma}{y_2}) }^2 \\
\notag &\qquad + \| \fol{H}{y_2} \|_{ L^{\infty, 2}_{x, \fol{t}{y_2}} (\fol{\Sigma}{y_1}, \fol{\Sigma}{y_2}) }^2 \| \fol{Z}{y_2} \|_{ L^{2, 1}_{x, \fol{t}{y_2}} (\fol{\Sigma}{y_1}, \fol{\Sigma}{y_2}) }^2 \text{,} \\
\notag {\bf L}_{\ul{H}} &\lesssim \| \fol{\ul{H}}{y_2} - \fol{\ul{H}}{y_1} \|_{ L^{2, \infty}_{x, t} (\fol{\Sigma}{y_1}, S_1) }^2 + \| \fol{\nabla}{y_2}_{ \fol{t}{y_2} } \fol{\ul{H}}{y_2} \|_{ L^{2, 1}_{x, \fol{t}{y_2}} (\fol{\Sigma}{y_1}, \fol{\Sigma}{y_2}) }^2 \\
\notag &\qquad + \| \fol{H}{y_2} \|_{ L^{\infty, 2}_{x, \fol{t}{y_2}} (\fol{\Sigma}{y_1}, \fol{\Sigma}{y_2}) }^2 \| \fol{\ul{H}}{y_2} \|_{ L^{2, 1}_{x, \fol{t}{y_2}} (\fol{\Sigma}{y_1}, \fol{\Sigma}{y_2}) }^2 \text{.}
\end{align}
\end{lemma}

\begin{proof}
See Appendix \ref{sect:estimates.limit}.
\end{proof}

By \eqref{eq.limits_lem_conv} and \eqref{eq.limits_lem_conv_ex},
\[ \lim_{y_1, y_2 \nearrow 1} [ \| \fol{Z}{y_2} - \fol{Z}{y_1} \|_{ L^{2, \infty}_{x, t} (\fol{\Sigma}{y_1}, S_1) } + \| \fol{\ul{H}}{y_2} - \fol{\ul{H}}{y_1} \|_{ L^{2, \infty}_{x, t} (\fol{\Sigma}{y_1}, S_1) } ] = 0 \text{.} \]
Furthermore, since the $\fol{\Sigma}{y}$'s are going to infinity, the remaining terms on the right-hand side of \eqref{eq.limit_lem_ZHbar}  will also converge to zero as $y_1, y_2 \nearrow 1$, by Proposition \ref{thm.renorm_cfol}.
Combining these observations results in \eqref{eq.limit_goal_ZHbar}.
Finally, since Theorem \ref{thm.nc_renorm} applies to every $\fol{v}{y}$-foliation of $\mc{N}$ by Proposition \ref{thm.renorm_cfol}, then \eqref{eq.limit_ZHbar} follows.

\section{Construction of the Distortion Functions} \label{sect:distortion}

In this section, we prove Lemma \ref{thm.distortion}, that is, we construct the family of distortion functions $\fol{v}{y}$ used to generate the family of asymptotically round spheres.
This is the first main component of the proof of Theorem \ref{the.thm}.

\subsection{Main Ideas} \label{sect:distortion.outline}

The construction of these $\fol{v}{y}$'s is a delicate exercise due to the low regularity of the metrics $\gamma_t$.
While it is not too difficult to obtain solutions $\fol{v}{y}$ satisfying \eqref{eq.asympt_round} (we shall see this amounts essentially to solving the uniformization equation, albeit with a rough background metric), it is far more delicate to solve it in a way that guarantees the properties \eqref{eq.distortion_bound}-\eqref{eq.distortion_conv_ex}.
These estimates are indispensable in guaranteeing that the curvature fluxes on $\mc{N}$ with respect to the new $\fol{s}{y}$-foliations are still bounded and small, and that the relevant physical quantities on the resulting asymptotically round spheres $(\fol{\Sigma}{y}, \fol{h}{y})$ converge as $y \nearrow 1$.

Heuristically, the requirement associated with \eqref{eq.asympt_round} is an elliptic equation over the limiting sphere $(S_1, \gamma_1)$. 
To see this, let us suppose we are given a distortion function $v$, and with it the corresponding $t^\prime$-foliation of $\mc{N}$.
The key observation is one of the (renormalized) structure equations found in \cite[Prop. 4.2]{alex_shao:nc_inf}.
More specifically, the renormalized \emph{Gauss equation}, applied to the $t^\prime$-foliation, gives
\begin{align}
\label{Gauss} \mc{K}^\prime - 1 &= -\frac{1}{2} \trace^\prime \ul{H}^\prime + s^{\prime -1} \gamma^{\prime ab} \nabla^\prime_a Z^\prime_b - s^{-1} M^\prime + \frac{1}{2} s^{\prime -1} \paren{ 1 - \frac{2 m_{\rm S}}{s^\prime} } \trace^\prime H^\prime \\
\notag &\qquad - \frac{1}{4} s^{\prime -1} \trace^\prime H^\prime \trace^\prime \ul{H}^\prime \text{.}
\end{align}
By collecting the terms that are asymptotically vanishing, we can write \eqref{Gauss} as
\[ \mc{K}^\prime = 1 - \frac{1}{2} \trace^\prime \ul{H}^\prime + (1 - t^\prime) \mc{E}^\prime \text{,} \]
where the ``error terms" $\mc{E}^\prime$ are uniformly small in the appropriate norms due to Theorem \ref{thm.nc_renorm}.
In particular, in the (weak) limit $t^\prime \nearrow 1$, we have on $S_1$ that
\[ \mc{K}^\prime_1 = 1 - \frac{1}{2} \trace^\prime \ul{H}^\prime_1 \text{.} \]

Note the obstacle preventing $\mc{K}^\prime$ from converging to $1$ at infinity is the presence of $\trace^\prime \ul{H}^\prime$; if this can be eliminated, then the $S^\prime_{t^\prime}$'s will become asymptotically round.
We will see that, using \eqref{eq.cf_Hbar}, we can write
\[ \trace^\prime \ul{H}^\prime = \trace \ul{H} + 2 \lapl v + 2 ( e^{2 v} - 1 ) + (1 - t) \mc{E}_\ast \text{,} \]
where $\mc{E}_\ast$ is uniformly small in the appropriate norms.
Since we wish for the left-hand side to vanish, then, in the limit $t \nearrow 1$, the equation we must solve is
\begin{equation} \label{thePDE} \Delta_{\gamma_1} v + e^{2 v} = 1 - \frac{1}{2} \trace \ul{H}_1 \text{,} \end{equation}
where $\lapl_{\gamma_1}$ is the Laplacian with respect to the limiting metric $\gamma_1$ (in the $t$-foliation).

\begin{remark}
Note that \eqref{thePDE} is analogous to the differential equation arising in the uniformization theorem, where one seeks a conformal transformation $\gamma_1 \mapsto e^{2 v} \gamma_1$ to obtain a round metric on the limiting sphere $S_1$. 
That the change of foliation induces a conformal transformation of the (renormalized) metric at infinity can be seen directly from \eqref{eq.cf_gammainv}; see also the discussion in the introduction.
\end{remark}

Nonetheless, the metric $\gamma_1$ at infinity is not regular enough to attack \eqref{thePDE} directly.
Also, one may prefer families of smooth spheres converging to infinity.
For these reasons, we instead obtain this $v$ in \eqref{thePDE} indirectly as a limit of smooth distortion functions $\fol{v}{y}$, for $y \in [0, 1)$, by solving approximations of \eqref{thePDE} on $\mc{N}$. 

\subsubsection{Proof Outline}

It is important to emphasize that for each $\fol{v}{y}$, the distortion function $\fol{v}{y}$ is constructed by solving an approximating elliptic equation \emph{on the sphere $(S_y, \gamma_y)$, in the $t$-foliation.}
The precise result is stated below:

\begin{lemma} \label{thm.to.solve}
There is a family of distortion functions $\fol{v}{y}$, $y \in [0, 1)$, such that:
\begin{itemize}
\item For a fixed $y$, this $\fol{v}{y}$ satisfies on $(S_y, \gamma_y)$ the equation
\begin{equation} \label{to.solve} \lapl \fol{v}{y} + e^{2 u} e^{2 \fol{v}{y}} = 1 - \frac{1}{2} \trace \ul{H} + (1 - t) E \text{,} \end{equation}
where $u$ and $E$ are smooth functions on $\mc{N}$ satisfying
\footnote{The role of $u$ here is to absorb the low regularity of the metrics $\gamma$; see Section \ref{sect:distortion.construct}.}
\begin{equation} \label{to.solve.u} \| u \|_{ L^\infty (S_t) } \lesssim (1 - t) \Gamma \text{,} \qquad \| E \|_{ B^{\infty, 0}_{t, x} } \lesssim \Gamma \text{.} \end{equation}

\item The $\fol{v}{y}$'s satisfy the properties \eqref{eq.distortion_bound}-\eqref{eq.distortion_conv_ex}.
\end{itemize}
\end{lemma}

\begin{proof}
See Section \ref{sect:distortion.construct}.
\end{proof}

We now show that the conclusions of Lemma \ref{thm.to.solve} imply the conclusions of Lemma \ref{thm.distortion}.
For each $\fol{v}{y}$, we consider the change of geodesic foliation associated with $\fol{v}{y}$, in particular the renormalized $\fol{t}{y}$-foliation of $\mc{N}$.
Letting $(\fol{\Sigma}{y}, \fol{h}{y})$ be the level sphere $(\fol{S}{y}_y, \fol{\gamma}{y}_y)$, then Proposition \ref{thm.renorm_cfol} and the bounds \eqref{eq.distortion_bound} for $\fol{v}{y}$ imply
\[ \| \fol{\mc{K}}{y} - 1 \|_{ H^{-1/2}_x (\fol{\Sigma}{y}, \fol{h}{y}) } \lesssim \| \fol{\trace}{y} \fol{\ul{H}}{y} \|_{ L^2_x (\fol{\Sigma}{y}, \fol{h}{y}) } + (1 - y) \Gamma \text{.} \]
Thus, to prove the second limit in \eqref{eq.asympt_round}, we need only show that
\begin{equation} \label{need.Hbartrace} \lim_{y \nearrow 1} \| \fol{\trace}{y} \fol{\ul{H}}{y} \|_{ L^2_x (\fol{\Sigma}{y}, \fol{h}{y}) } = 0 \text{.} \end{equation}

Now, we look at the transformation law \eqref{eq.cf_Hbar} for $\fol{\ul{H}}{y}$ and $\ul{H}$, and we separate all the terms on the right-hand side which vanish as $t \nearrow 1$.
Since $\mc{B}_{-1} = 1 + s^{-1} \mc{C}_{-1}$ and $t = 1 - s^{-1}$, then \eqref{eq.cf_Hbar} can be expressed in the form
\begin{align*}
\fol{\ul{H}}{y}_{ab} &= \ul{H}_{ab} + 2 \nabla_{ab} \fol{v}{y} + (e^{2 \fol{v}{y}} - 1) \gamma_{ab} + \gamma^{cd} \nabla_c \fol{v}{y} \nabla_d \fol{v}{y} \cdot \gamma_{ab} \\
\notag &\qquad - 2 \nabla_a \fol{v}{y} \nabla_b \fol{v}{y} + s^{-1} \bar{\mc{E}}_{ab} \text{.}
\end{align*}
Moreover, from \eqref{eq.renorm_est} and \eqref{eq.distortion_bound}, we can estimate $\| \bar{\mc{E}} \|_{ L^{2, \infty}_{x, t} } \lesssim \Gamma$.

Taking a trace of the above identity and recalling \eqref{eq.cf_gammainv} yields
\begin{equation} \label{eq.cf_Hbartrace} \fol{\trace}{y} \fol{\ul{H}}{y} = e^{-2v} \mc{B}_2 \trace \fol{\ul{H}}{y} = \mf{A} + s^{-1} \mc{E} \text{,} \end{equation}
where $\mf{A}$ and $\mc{E}$ satisfy
\[ \mf{A} = \trace \ul{H} + 2 \lapl \fol{v}{y} + 2 ( e^{2 \fol{v}{y}} - 1 ) \text{,} \quad \| \mc{E} \|_{ L^{2, \infty}_{x, t} } \lesssim \Gamma \text{.} \]
As a result,
\[ \| \trace^y \ul{H}^y \|_{ L^2_x (\fol{\Sigma}{y}, \fol{h}{y}) } \lesssim \| \mf{A} \|_{ L^{2, \infty}_{x, t} (\fol{\Sigma}{y}, \fol{h}{y}) } + \inf_{ P \in \fol{\Sigma}{y} } s^{-1} (P) \cdot \Gamma \text{.} \]
Since the $\fol{\Sigma}{y}$'s go to infinity, the last term on the right-hand side vanishes as $y \nearrow 1$.
Consequently, to prove \eqref{need.Hbartrace}, it remains to show
\begin{equation} \label{need.mfA} \lim_{y \nearrow 1} \| \mf{A} \|_{ L^2_x (\fol{\Sigma}{y}, \fol{h}{y}) } = 0 \text{.} \end{equation}

The key observation behind proving \eqref{need.mfA} is the equation \eqref{to.solve} satisfied by $\fol{v}{y}$ on $S_y$.
To take advantage of this, we first move our estimate to $S_y$:
\footnote{Here, we implicitly used that integral norms in the $t$- and $\fol{t}{y}$-foliations are comparable.}
\begin{align*}
\| \mf{A} \|_{ (\fol{\Sigma}{y}, \fol{h}{y}) } &\lesssim \| \mf{A} \|_{ L^2_x (S_y) } + \| \nabla_t (\trace \ul{H}) \|_{ L^{2, 1}_{x, t} (S_y, \fol{\Sigma}{y}) } + \| \nabla_t \nabla^2 \fol{v}{y} \|_{ L^{2, 1}_{x, t} (S_y, \fol{\Sigma}{y}) } \text{.}
\end{align*}
By \eqref{eq.renorm_est} and \eqref{eq.distortion_bound}, the last two terms vanish as $y \nearrow 1$.
Furthermore, by \eqref{to.solve},
\[ \mf{A} = 2 (1 - t) E + 2 e^{2 \fol{v}{y}} (1 - e^{2 u}) \]
on $S_y$.
Since $u$ converges uniformly to $0$ as $t \nearrow 1$ by \eqref{to.solve.u},
\begin{align*}
\| \mf{A} \|_{ L^2_x (S_y) } &\lesssim (1 - y) \| E \|_{ L^2_x (S_y) } + \| 1 - e^{2 u} \|_{ L^\infty (S_y) } \rightarrow 0 \text{,}
\end{align*}
as $y \nearrow 1$.
Thus, we have established \eqref{need.Hbartrace}, and the second limit in \eqref{eq.asympt_round} follows.

Finally, letting $\fol{\omega}{y}$ denote the volume form associated with $\fol{h}{y}$, we have
\begin{align*}
\operatorname{Area} (\fol{\Sigma}{y}, \fol{h}{y}) &= \int_{ \fol{\Sigma}{y} } \fol{\mc{K}}{y} d \fol{\omega}{y} - \int_{ \fol{\Sigma}{y} } ( \fol{\mc{K}}{y} - 1 ) d \fol{\omega}{y} = 4 \pi - \int_{ \fol{\Sigma}{y} } ( \fol{\mc{K}}{y} - 1 ) d \fol{\omega}{y} \text{.}
\end{align*}
The second limit in \eqref{eq.asympt_round} implies that the last term on the right-hand side vanishes as $y \nearrow 1$.
This yields the first limit in \eqref{eq.asympt_round} and proves Lemma \ref{thm.distortion}.

\subsection{Proof of Lemma \ref{thm.to.solve}} \label{sect:distortion.construct}

The proof of Lemma \ref{thm.to.solve}, i.e., the formal construction of the $\fol{v}{y}$'s, can be divided into a three-step process.
The first two steps essentially amount to preliminary smoothings of the metrics $\gamma_t$, while the last step involves carefully chosen uniformizations of the smoothed metrics.

The goal of the first two steps is to reduce matters to solving \eqref{to.solve} over a metric with $L^\infty$-bounds on its Gauss curvature.
This is accomplished by two conformal transformations which absorb the lower regularity terms in \eqref{Gauss}.
The first conformal transformation, which comes from \cite{shao:stt}, smoothes the Gauss curvature from $H^{-1/2}$ to $B^0$.
\footnote{Having $B^0$ instead of $L^2$ is important, as it ultimately results in $L^\infty$-bounds for $\nabla \fol{v}{y}$.}
The second step adapts an idea from \cite{bie_zip:stb_mink} to further smooth the curvature to $L^\infty$.
Finally, at the third step, we proceed with a uniformization, adapting an argument of Christodoulou and Klainerman in \cite{chr_kl:stb_mink}.

In the end, the $\fol{v}{y}$'s are obtained via a composition of \emph{only the last two steps}.
In particular, the impact of the first (and also least regular) smoothing vanishes at infinity and can be discarded.
\footnote{This is due to the factor $s^{-1}$ in front of $\operatorname{div} Z$ in \eqref{Gauss}.}
The remainder of the proof of Lemma \ref{thm.to.solve} is dedicated to deriving estimates and convergence properties for the $\fol{v}{y}$'s.
In particular, we obtain $L^\infty$-bounds for the $\nabla \fol{v}{y}$'s, which are essential for the main result.

\subsubsection{Step 1: The Initial Smoothing}

The first technical issue that one faces is the irregularity of the $\gamma_t$'s; in particular, the Gauss curvatures of the $\gamma_t$'s lie only in $H^{-1/2}$.
\footnote{This causes a substantial number of issues, e.g., for elliptic estimates; see \cite{kl_rod:cg, shao:stt, wang:cg, wang:cgp}.}
Thus, in this first step, we apply a conformal smoothing of the $\gamma_t$'s in order to obtain a new family $\bar{\gamma}_t$, whose Gauss curvatures $\bar{\mc{K}}_t$ are uniformly bounded in $L^2$.
This was the same process that was employed in \cite[Sect. 6]{shao:stt} in order to derive elliptic estimates for various symmetric Hodge operators in Besov norms.
That this is possible rests on the observation that the least smooth term in the right-hand side of \eqref{Gauss} is an exact divergence: $(1 - t) \gamma^{ab} \nabla_a Z_b$.

As in \cite[Sect. 6.4]{shao:stt}, we define the function $u$ on $\mc{N}$ satisfying, for every $\tau \in [0, 1)$,
\begin{equation} \label{eqn.u} \lapl u_\tau = s^{-1} \gamma_\tau^{ab} \nabla_a (Z_\tau)_b \text{,} \qquad \int_{ S_\tau } u_\tau d \epsilon_\tau = 0 \text{,} \end{equation}
where $\epsilon_\tau$ denotes the volume form associated with $\gamma_\tau$.
In other words, $u_\tau$ is the unique mean-free function solving the above Poisson equation on $(S_\tau, \gamma_\tau)$.
Define next a new family of metrics $\bar{\gamma}_\tau = e^{2 u_\tau} \gamma_\tau$ on the $S_\tau$'s.
Then, from standard formulas, the Gauss curvatures $\bar{\mc{K}}_\tau$ of the $(S_\tau, \bar{\gamma}_\tau)$'s are given by
\begin{align}
\label{eq.K_bar} \bar{\mc{K}} &= e^{-2 u} ( \mc{K} - \lapl u ) = e^{-2 u} \brak{ 1 - \frac{1}{2} \trace \ul{H} + (1 - t) E } \text{,} \\
\notag E &= - M + \frac{1}{2} [ 1 - 2 m_{\rm S} (1 - t) ] \trace \ul{H} - \frac{1}{4} \trace H \trace \ul{H} \text{.}
\end{align}

To control $u$, we apply elliptic estimates on the Poisson equation \eqref{eqn.u}, along with existing bounds for $Z$.
\footnote{While such estimates are immediate for regular background metrics, they are very delicate for the rough metrics under consideration here.
In particular, we resort to estimates in \cite{shao:stt}.}
Using a variant of \cite[Prop. 6.10]{shao:stt}, we obtain
\footnote{More specifically, we apply the proof of \cite[Eq. (6.18)]{shao:stt} individually to each $S_t$ and take advantage of the factor $s^{-1} = 1 - t$ in front of the divergence of $Z$.}
\begin{equation} \label{eq.est_u} \| u \|_{ L^\infty_x (S_t) } + \| \nabla u \|_{ L^4_x (S_t) } \lesssim (1 - t) \Gamma \text{.} \end{equation}
Moreover, \cite[Prop. 6.10]{shao:stt} implies
\[ \| \bar{\mc{K}} - 1 \|_{ L^2_x (S_\tau) } \leq \| \bar{\mc{K}} - e^{-2u} \|_{ L^2_x (S_\tau) } + \| e^{-2u} - 1 \|_{ L^2_x (S_\tau) } \lesssim \Gamma \text{.} \]
We also note that by \eqref{eq.renorm_est} and the definition of $E$ in \eqref{eq.K_bar},
\footnote{To control the term $\trace H \cdot \trace \ul{H}$ in Besov norms, we use \cite[Thm. 3.6]{shao:stt}, along with \eqref{eq.renorm_est}.}
\begin{equation} \label{Eerror.est} \| E \|_{ B^{\infty, 0}_{t, x} } + \| E \|_{ L^{2, \infty}_{x, t} } \lesssim \Gamma \text{.} \end{equation}
In particular, the above choices of $u$ and $E$ satisfy \eqref{to.solve.u}.

\begin{remark}
We note that the conformal factors $u_t$ \emph{do not} have the smoothness required of the desired $\fol{v}{y}$'s and hence \emph{will not} be built into the $\fol{v}{y}$'s.
Their purpose is to absorb the term of least regularity in $\mc{K}$, thus producing a more regular metric, from which we can construct the desired $\fol{v}{y}$'s.
Note in particular that as $t \nearrow 1$, the $\bar{\mc{K}}_t$'s converge to $\bar{\mc{K}}_1 = \mc{K}_1 = 1 - \frac{1}{2} \trace \ul{H}_1$ at $S_1$ \emph{in $L^2_x$}.
\end{remark}

\subsubsection{Step 2: Further Smoothing}

In the next step, we generate the first part of the $\fol{v}{y}$'s.
To accomplish this task, we apply yet another conformal transformation, directly inspired by Bieri, \cite{bie_zip:stb_mink}.
Its purpose is to reduce matters to solving an analogue of \eqref{to.solve}, but with the right-hand side lying in $L^\infty_x$ rather than in $L^2_x$. 
Throughout, we let $\bar{\nabla}$ and $\bar{\lapl}$ be the Levi-Civita connection and Laplacian relative to $\bar{\gamma}$, and we let $| S_y |$ be the area of $S_y$, \emph{relative to the metric $\gamma_y$}.
Moreover, given a smooth function $f$ on $S_y$, we let $\mc{A}_y (f)$ denote its $\gamma_y$-average:
\[ \mc{A}_y (f) = | S_y |^{-1} \int_{ S_y } f d \epsilon_y \text{.} \]

Next, we solve (uniquely) on $(S_y, \gamma_y)$ the Poisson equation
\begin{equation} \label{poisson} \lapl \fol{v}{y}_1 = e^{2 u_y} \bar{\mc{K}}_y - \mc{A}_y ( e^{2 u_y} \bar{\mc{K}}_y ) \text{,} \qquad \int_{ S_y } \fol{v}{y}_1 d \epsilon_y = 0 \text{.} \end{equation}
Recalling the value of $\bar{\mc{K}}$, we can expand the equation as
\[ \lapl \fol{v}{y}_1 = - \frac{1}{2} [ \trace \ul{H}_y - \mc{A}_y ( \trace \ul{H}_y ) ] + (1 - y) [ E_y - \mc{A}_y (E_y) ] \text{.} \]
Applying \eqref{eq.renorm_est} and \eqref{Eerror.est} yields
\begin{equation} \label{poisson.est} \| \lapl \fol{v}{y}_1 \|_{ B^0_x (S_y) } \lesssim \| \trace \ul{H} \|_{ B^0_x (S_y) } + (1 - y) \| E \|_{ B^0_x (S_y) } \lesssim \Gamma \text{.} \end{equation}
Note in particular that the second derivative of $\fol{v}{y}_1$ is bounded \emph{in the Besov norm}.
This is a crucial point, as it will allow us to control $\nabla \fol{v}{y}_1$ in $L^\infty_x$.

The above argument defined $\fol{v}{y}_1$ only on $S_y$.
Next, we extend each $\fol{v}{y}_1$ to all of $\mc{N}$ by requiring it to be constant along every null generator of $\mc{N}$; in other words, we require $\nabla_t \fol{v}{y}_1 \equiv 0$.
In order to derive the full complement of estimates for the $\fol{v}{y}_1$'s from \eqref{poisson.est}, we resort to standard elliptic, Sobolev embedding, and transport estimates.
The only caveat here is the low regularity of our setting, which forces us to apply the tools developed in \cite{shao:stt}.
To avoid distracting from our main construction, we defer the details of these estimates to the appendices.

\begin{lemma} \label{thm.est_v1}
For any $y \in [0, 1)$, the following estimates hold:
\begin{align}
\label{eq.est_v1} \| \nabla_t \nabla^2 \fol{v}{y}_1 \|_{ L^{2, 2}_{t, x} } + \| \nabla_t \nabla \fol{v}{y}_1 \|_{ L^{\infty, 2}_{x, t} } &\lesssim \Gamma^2 \text{,} \\
\notag \| \nabla^2 \fol{v}{y}_1 \|_{ B^{\infty, 0}_{t, x} \cap L^{2, \infty}_{x, t} } + \| \nabla \fol{v}{y}_1 \|_{ L^{\infty, \infty}_{t, x} } + \| \fol{v}{y}_1 \|_{ L^{\infty, \infty}_{t, x} } &\lesssim \Gamma \text{.}
\end{align}
Moreover, for any $q \in (2, \infty)$ and $p = \frac{2q}{q + 2} \in (1, 2)$,
\begin{align}
\label{eq.est_v1_cauchy} \lim_{y_1, y_2 \nearrow 1} [ \| \nabla_t \nabla^2 ( \fol{v}{y_2}_1 - \fol{v}{y_1}_1 ) \|_{ L^{p, 2}_{x, t} } + \| \nabla_t \nabla ( \fol{v}{y_2}_1 - \fol{v}{y_1}_1 ) \|_{ L^{q, 2}_{x, t} } ] &= 0 \text{,} \\
\notag \lim_{y_1, y_2 \nearrow 1} [ \| \nabla^2 ( \fol{v}{y_2}_1 - \fol{v}{y_1}_1 ) \|_{ L^{p, \infty}_{x, t} } + \| \nabla ( \fol{v}{y_2}_1 - \fol{v}{y_1}_1 ) \|_{ L^{q, \infty}_{x, t} } + \| \fol{v}{y_2}_1 - \fol{v}{y_1}_1 \|_{ L^{\infty, \infty}_{t, x} } ] &= 0 \text{,} \\
\notag \lim_{y_1, y_2 \nearrow 1} \| \nabla^2 ( \fol{v}{y_2}_1 - \fol{v}{y_1}_1 ) \|_{ L^{2, \infty}_{x, t} ( \fol{S}{y_1}_{y_1}, S_1) } &= 0 \text{.}
\end{align}
\end{lemma}

\begin{proof}
See Appendix \ref{sect:estimates.v1}.
\end{proof}

Finally, defining the metric
\[ \ddot{\gamma}_y = e^{2 \fol{v}{y}_1} \bar{\gamma}_y \]
on $S_y$, we find that its curvature $\ddot{\mc{K}}_y$ satisfies
\begin{equation} \label{new.bound} \ddot{\mc{K}}_y = e^{-2 \fol{v}{y}_1} ( \bar{\mc{K}}_y - \bar{\lapl} \fol{v}{y}_1 ) = e^{-2 \fol{v}{y}_1} e^{-2 u} \mc{A}_y ( e^{2 u} \bar{\mc{K}}_y ) \text{.} \end{equation}
$\ddot{\mc{K}}_y$ is uniformly bounded, independently of $y$, since by \eqref{eq.est_u} and \eqref{eq.est_v1},
\begin{align}
\label{eq.K_dotdot} \| \ddot{\mc{K}}_y - 1 \|_{ L^\infty_x (S_y) } &\lesssim \| e^{-2 (\fol{v}{y}_1 + u)} - 1 \|_{ L^\infty_x (S_y) } + \| \mc{A}_y ( \trace \ul{H} ) \|_{ L^\infty_x (S_y) } \\
\notag &\qquad + \| \mc{A}_y ( E ) \|_{ L^\infty_x (S_y) } \\
\notag &\lesssim \Gamma \text{.}
\end{align}

\subsubsection{Step 3: Uniformization of the Smoothed Metrics}

Now that the metrics have been smoothed as to have $L^\infty_x$-curvature, we can proceed to the third and final step of the proof of Lemma \ref{thm.to.solve}: we construct functions $\fol{v}{y}_2$ solving
\begin{equation} \label{final.unif} \ddot{\lapl}_y \fol{v}{y}_2 + e^{2 \fol{v}{y}_2} = e^{-2 (\fol{v}{y}_1 + u)} \mc{A}_y ( e^{2 u} \bar{\mc{K}}_y ) = \ddot{\mc{K}}_y \text{,} \end{equation}
\emph{on $S_y$}, where $\ddot{\lapl}_y$ denotes the Laplacian with respect to the metric $\ddot{\gamma}_y$ on $S_y$.
In other words, we solve the uniformization problem on $(S_y, \ddot{\gamma}_y)$.

While the uniformization problem itself is classical, in our current situation, we must also ensure that these $\fol{v}{y}_2$'s are uniformly small (that is, controlled by $\Gamma$) and converge appropriately as $y \nearrow 1$.
The difficulties behind these additional constraints arise from the lack of uniqueness of solutions of \eqref{final.unif} due to the conformal group on the sphere.
For this task, we resort to the subsequent abstract lemma:

\begin{lemma} \label{thm.unif}
Let $h$ be a Riemannian metric on $\Sph^2$, whose Gauss curvature satisfies
\begin{equation} \label{eq.unif_ass} \| \mc{K}_h - 1 \|_{ L^\infty_x (\Sph^2) } \lesssim \Gamma \text{.} \end{equation}
If $\Gamma$ is sufficiently small, then there exists a smooth function $v: \Sph^2 \rightarrow \R$, with
\begin{equation} \label{eq.unif_est} \| v \|_{ L^\infty_x (\Sph^2) } \lesssim \Gamma \text{,} \end{equation}
such that $\mathring{h} = e^{2 v} h$ is the round metric, with Gauss curvature identically equal to $1$.
Furthermore, $v$ can be chosen to depend continuously on the pair $(h, \mc{K}_h)$.
\end{lemma}

\begin{proof}
See Appendix \ref{sect:distortion.unif}.
\end{proof}

Applying Lemma \ref{thm.unif} to each $(S_y, \ddot{\gamma}_y)$, $y \in [0, 1)$, we obtain functions $\fol{v}{y}_2$ on $S_y$ satisfying \eqref{final.unif} as well as the following estimate:
\begin{equation} \label{eq.est_v2_pre} \| \fol{v}{y}_2 \|_{ L^\infty_x (S_y) } \lesssim \Gamma \text{,} \qquad y \in [0, 1) \text{.} \end{equation}
Like for the $\fol{v}{y}_1$'s, we extend the $\fol{v}{y}_2$'s to $\mc{N}$ by the condition $\nabla_t \fol{v}{y}_2 \equiv 0$.

Since the $\gamma_t$'s converge as $t \nearrow 1$ (see Section \ref{sect:prelim.ncinf}), since the $\fol{v}{y}_1$'s converge as $y \nearrow 1$ by \eqref{eq.est_v1_cauchy}, and since $u$ converges to zero at infinity by \eqref{eq.est_u}, it follows that $\ddot{\gamma}_y$, restricted to $S_y$, also converges (uniformly) as $y \nearrow 1$.
Furthermore, from \eqref{new.bound}, along with Corollary \ref{thm.nc_renorm.limit.rc}, \eqref{eq.est_u}, and \eqref{eq.est_v1_cauchy}, we see that $\ddot{\mc{K}}_y$, again restricted to $S_y$, converges uniformly as $y \nearrow 1$.
Consequently, by the continuous dependence statement in Lemma \ref{thm.unif}, the $\fol{v}{y}_2$'s must also converge as $y \nearrow 1$.
In other words, as functions defined on all of $\mc{N}$, the $\fol{v}{y}_2$'s satisfy the Cauchy property
\begin{equation} \label{eq.est_v2_cauchy_pre} \lim_{y_1, y_2 \nearrow \infty} \| \fol{v}{y_2}_2 - \fol{v}{y_1}_2 \|_{ L^{\infty, \infty}_{t, x} } = 0 \text{.} \end{equation}

Finally, rewriting \eqref{final.unif} as
\begin{equation} \label{eqn.v_2} \Delta \fol{v}{y}_2 = e^{2 (u + \fol{v}{y}_1)} \ddot{\Delta} \fol{v}{y}_2 = e^{2 (u + \fol{v}{y}_1)} ( \ddot{\mc{K}} - e^{2 \fol{v}{y}_2} ) \text{,} \end{equation}
and applying \eqref{eq.est_v2_pre}, \eqref{eq.est_v2_cauchy_pre}, and the usual elliptic, embedding, and transport estimates (at low regularities, via \cite{shao:stt}), we derive the full set of bounds for the $\fol{v}{y}_2$'s.
The proof, given in Appendix \ref{sect:estimates.v2}, is analogous to that for the $\fol{v}{y}_1$'s.

\begin{lemma} \label{thm.est_v2}
For any $y \in [0, 1)$, the following estimates hold:
\begin{align}
\label{eq.est_v2} \| \nabla_t \nabla^2 \fol{v}{y}_2 \|_{ L^{2, 2}_{t, x} } + \| \nabla_t \nabla \fol{v}{y}_2 \|_{ L^{\infty, 2}_{x, t} } &\lesssim \Gamma^2 \text{,} \\
\notag \| \nabla^2 \fol{v}{y}_2 \|_{ B^{\infty, 0}_{t, x} \cap L^{2, \infty}_{x, t} } + \| \nabla \fol{v}{y}_2 \|_{ L^{\infty, \infty}_{t, x} } + \| \fol{v}{y}_2 \|_{ L^{\infty, \infty}_{t, x} } &\lesssim \Gamma \text{.}
\end{align}
Moreover, for any $q \in (2, \infty)$ and $p = \frac{2q}{q + 2} \in (1, 2)$,
\begin{align}
\label{eq.est_v2_cauchy} \lim_{y_1, y_2 \nearrow 1} [ \| \nabla_t \nabla^2 ( \fol{v}{y_2}_2 - \fol{v}{y_1}_2 ) \|_{ L^{p, 2}_{x, t} } + \| \nabla_t \nabla ( \fol{v}{y_2}_2 - \fol{v}{y_1}_2 ) \|_{ L^{q, 2}_{x, t} } ] &= 0 \text{,} \\
\notag \lim_{y_1, y_2 \nearrow 1} [ \| \nabla^2 ( \fol{v}{y_2}_2 - \fol{v}{y_1}_2 ) \|_{ L^{p, \infty}_{x, t} } + \| \nabla ( \fol{v}{y_2}_2 - \fol{v}{y_1}_2 ) \|_{ L^{q, \infty}_{x, t} } + \| \fol{v}{y_2}_2 - \fol{v}{y_1}_2 \|_{ L^{\infty, \infty}_{t, x} } ] &= 0 \text{,} \\
\notag \lim_{y_1, y_2 \nearrow 1} \| \nabla^2 ( \fol{v}{y_2}_2 - \fol{v}{y_1}_2 ) \|_{ L^{2, \infty}_{x, t} (\fol{S}{y_1}_{y_1}, S_1) } &= 0 \text{.}
\end{align}
\end{lemma}

\begin{proof}
See Appendix \ref{sect:estimates.v2}.
\end{proof}

\subsubsection{The Distortion Functions $\fol{v}{y}$}

Finally, we complete the proof of Lemma \ref{thm.to.solve} by combining the three steps described above.
Defining our desired distortion functions by $\fol{v}{y} = \fol{v}{y}_1 + \fol{v}{y}_2$, we see on $(S_y, \gamma_y)$ that
\begin{align*}
\bar{\lapl} \fol{v}{y} + e^{2 \fol{v}{y}} &= \bar{\lapl} \fol{v}{y}_1 + e^{2 \fol{v}{y}_1} ( \ddot{\lapl} \fol{v}{y}_2 + e^{2 \fol{v}{y}_2} ) \\
&= \bar{\mc{K}}_y - e^{-2 u} \mc{A}_y ( e^{2 u} \bar{\mc{K}}_y ) + e^{2 \fol{v}{y}_1} [ e^{-2 (\fol{v}{y}_1 + u)} \mc{A}_y ( e^{2 u} \bar{\mc{K}}_y ) ] \\
&= e^{-2 u} \left( 1 - \frac{1}{2} \trace \ul{H} + s^{-1} E \right) \text{,}
\end{align*}
where we also noted that $\bar{\lapl} = e^{-2 u} \lapl$.
Consequently,
\[ \lapl \fol{v}{y} + e^{2 u} e^{2 \fol{v}{y}} = e^{2 u} ( \bar{\lapl} \fol{v}{y} + e^{2 \fol{v}{y}} ) = 1 - \frac{1}{2} \trace \ul{H} + (1 - t) E \text{,} \]
which proves \eqref{to.solve}.
Furthermore, combining Lemmas \ref{thm.est_v1} and \ref{thm.est_v2} immediately yields \eqref{eq.distortion_bound}, \eqref{eq.distortion_conv}, and \eqref{eq.distortion_conv_ex}, completing the proof of Lemma \ref{thm.to.solve}.

\appendix

\section{Changes of Foliations} \label{sect:cfol}

In this Appendix, we prove the change of foliations formulas from Section \ref{sect:prelim.cfol}.

\subsection{Proof of Proposition \ref{thm.cfol.phys}} \label{sect:cfol.phys}

First of all, we observe that the conjugate null vector fields $\ul{L}^\prime$ and $\ul{L}$, for the $s^\prime$- and $s$-foliations, respectively, satisfy
\begin{equation} \label{eq.cf_Lbar} \ul{L}^\prime = e^{-v} \ul{L} + (s - 1)^2 \mind^{ab} e^{-v} \nasla_a v \nasla_b v \cdot L + 2 (s - 1) e^{-v} \gras v \text{,} \end{equation}
where $\gras v$ is the $\mind$-gradient of $v$, i.e., $\gras^a v = \mind^{ab} \nasla_b v$.
To see this, one can directly compute that right-hand side $\ul{L}^\prime$ of \eqref{eq.cf_Lbar} satisfies
\[ g ( \ul{L}^\prime, \ul{L}^\prime ) \equiv 0 \text{,} \qquad g ( \ul{L}^\prime, L^\prime ) \equiv -2 \text{,} \qquad g (\ul{L}^\prime, X^\prime ) \equiv 0 \text{.} \]

Furthermore, for convenience, we define the coefficients
\[ \mc{L}_a = (s - 1) \nasla_a v \text{,} \qquad \mc{M} = (s - 1)^2 \mind^{ab} \nasla_a v \nasla_b v \text{,} \]
which show up in the formulas \eqref{eq.cf_X} and \eqref{eq.cf_Lbar}.
Also, we let $e_a$ and $e^\prime_a$ denote the frame elements in the $s$ and $s^\prime$-foliations corresponding to the index $a$.

\subsubsection{Ricci Coefficients}

For \eqref{eq.cf_chi}, we have, by the definitions of $\chi$ and $\chi^\prime$,
\begin{align*}
\chi^\prime_{ab} &= g ( D_{ e_a^\prime } L^\prime, e_b^\prime ) = g ( D_{ e_a + \mc{L}_a L } ( e^v L ), e_b + \mc{L}_b L ) = g ( D_{ e_a } ( e^v L ), e_b ) = e^v \chi_{ab} \text{,}
\end{align*}
where we used that $L$ is normal to $\mc{N}$.
For \eqref{eq.cf_zeta}, we do similar computations:
\begin{align*}
\zeta^\prime_a &= \frac{1}{2} g ( D_{ e_a + \mc{L}_a L } ( e^v L ), e^{-v} \ul{L} + e^{-v} \mc{M} L + 2 (s - 1) e^{-v} \gras v ) \\
&= \frac{1}{2} g ( D_{ e_a } L, \ul{L} ) + \frac{1}{2} \nasla_a v \cdot g (L, \ul{L}) + (s - 1) \cdot g ( D_{ e_a } L, \gras v ) \\
&= \zeta_a + (s - 1) \mind^{bc} \nasla_b v \cdot \chi_{ac} - \nasla_a v \text{.}
\end{align*}

The process for $\ul{\chi}^\prime$ and $\ul{\chi}$ is similar, but longer:
\begin{align*}
\ul{\chi}^\prime_{ab} &= g ( D_{ e_a + \mc{L}_a L } ( e^{-v} \ul{L} ), e_b + \mc{L}_b L ) + g ( D_{ e_a + \mc{L}_a L } ( e^{-v} \mc{M} L ), e_b + \mc{L}_b L ) \\
&\qquad + 2 g ( D_{ e_a + \mc{L}_a L } [ (s - 1) e^{-v} \gras v ], e_b + \mc{L}_b L ) \\
&= I_1 + I_2 + I_3 \text{.}
\end{align*}
The simplest term to handle is $I_2$:
\[ I_2 = g ( D_{e_a} ( e^{-v} \mc{M} L ), e_b ) = e^{-v} (s - 1)^2 \mind^{cd} \nasla_c v \nasla_d v \cdot \chi_{ab} \text{.} \]
Next, for $I_1$, we expand:
\begin{align*}
I_1 &= g ( D_{e_a} ( e^{-v} \ul{L} ), e_b ) + \mc{L}_a \cdot g ( D_L ( e^{-v} \ul{L} ), e_b ) + \mc{L}_b \cdot g ( D_{e_a} ( e^{-v} \ul{L} ), L ) \\
&\qquad + \mc{L}_a \mc{L}_b \cdot g ( D_L ( e^{-v} \ul{L} ), L ) \\
&= e^{-v} \ul{\chi}_{ab} - 2 e^{-v} \mc{L}_a \cdot \zeta_b + 2 e^{-v} \mc{L}_b \nasla_a v - 2 e^{-v} \mc{L}_b \zeta_a \\
&= e^{-v} \ul{\chi}_{ab} + 2 (s - 1) e^{-v} \nasla_a v \nasla_b v - 2 (s - 1) e^{-v} ( \nasla_a v \cdot \zeta_b + \nasla_b v \cdot \zeta_a ) \text{.}
\end{align*}
Finally, for $I_3$:
\begin{align*}
I_3 &= 2 (s - 1) \cdot g ( D_{e_a} ( e^{-v} \gras v ), e_b ) + 2 e^{-v} \mc{L}_a \cdot g ( D_L [ (s - 1) \gras v ], e_b ) \\
&\qquad + 2 (s - 1) \mc{L}_b \cdot g ( D_{e_a} ( e^{-v} \gras v), L ) + 2 (s - 1) e^{-v} \mc{L}_a \mc{L}_b \cdot g ( D_L \gras v, L ) \\
&= - 2 (s - 1) e^{-v} \nasla_a v \nasla_b v + 2 (s - 1) e^{-v} \cdot g ( D_{e_a} \gras v, e_b ) + 2 e^{-v} \mc{L}_a \nasla_b v \\
&\qquad + 2 (s - 1) e^{-v} \mc{L}_a \cdot g ( D_L \gras v, e_b ) + 2 (s - 1) e^{-v} \mc{L}_b \cdot g ( D_{e_a} \gras v, L ) \\
&= - 2 (s - 1) e^{-v} \nasla_a v \nasla_b v + 2 (s - 1) e^{-v} \nasla_{ab} v + 2 e^{-v} \mc{L}_a \nasla_b v \\
&\qquad + 2 (s - 1) e^{-v} \mc{L}_a \cdot \nasla_s \nasla_b v - 2 (s - 1) e^{-v} \mind^{cd} \mc{L}_b \nasla_c v \cdot \chi_{ad} \text{,}
\end{align*}
where we used that $\nasla_a$ and $\nasla_s$ are the projections of the corresponding spacetime covariant derivatives onto the $\mc{S}_s$'s.
Since $v$ is $s$-independent, \eqref{eq.comm} yields
\[ \nasla_s \nasla_b v = - \mind^{cd} \chi_{bc} \nasla_d v \text{,} \]
and it follows that
\begin{align*}
I_3 &= 2 (s - 1) e^{-v} \nasla_{ab} v - 2 (s - 1)^2 e^{-v} \mind^{cd} \nasla_c v ( \nasla_a v \cdot \chi_{bd} + \nasla_b v \cdot \chi_{ad} ) \text{.}
\end{align*}
Combining $I_1$, $I_2$, and $I_3$ yields \eqref{eq.cf_chibar}.

\subsubsection{Curvature Coefficients}

Next, we establish \eqref{eq.cf_alpha}-\eqref{eq.cf_betabar}, which involve the curvature components.
For \eqref{eq.cf_alpha}, we have
\begin{align*}
\alpha^\prime_{ab} &= e^{2v} \cdot R ( L, e_a + \mc{L}_a L, L e_b + \mc{L}_b L ) = e^{2v} \cdot R ( L, e_a, L, e_b ) = e^{2v} \alpha_{ab} \text{.}
\end{align*}
Similarly, for $\beta^\prime$ and $\beta$, we compute
\begin{align*}
\beta^\prime_a &= \frac{1}{2} e^v \cdot R ( L, \ul{L} + \mc{M} L + 2 (s - 1) \gras v, L, e_a + \mc{L}_a L ) \\
&= \frac{1}{2} e^v \cdot R ( L, \ul{L}, L, e_a ) + (s - 1) e^v \cdot R ( L, \gras v, L, e_a ) \text{,}
\end{align*}
from which \eqref{eq.cf_beta} follows.
Moreover, \eqref{eq.cf_rho} is a consequence of the identities
\begin{align*}
\rho^\prime &= \frac{1}{4} R ( L, \ul{L} + 2 (s - 1) \gras v, L, \ul{L} + 2 (s - 1) \gras v ) \\
&= \frac{1}{4} R ( L, \ul{L}, L, \ul{L} ) + (s - 1) \cdot R ( L, \ul{L}, L, \gras v) + (s - 1)^2 \cdot R ( L, \gras v, L, \gras v ) \text{.}
\end{align*}

Next, let $\{ e_1, e_2 \}$ be a positively oriented orthonormal frame on the $\mc{S}_s$'s.
Then,
\begin{align*}
\sigma &= \frac{1}{4} {}^\star R ( L, \ul{L}, L, \ul{L} ) = - \frac{1}{2} R ( e_1, e_2, L, \ul{L} ) \text{,}
\end{align*}
by the definition of the Hodge dual; an analogous identity holds for $\sigma^\prime$.
Therefore,
\begin{align*}
\sigma^\prime &= -\frac{1}{2} R ( e_1 + \mc{L}_1 L, e_2 + \mc{L}_2 L, L, \ul{L} + \mc{M} L + 2 (s - 1) \gras v ) \\
&= - \frac{1}{2} R ( e_1, e_2, L, \ul{L} ) - \frac{1}{2} R ( L, \mc{L}_1 e_2 - \mc{L}_2 e_1, L, \ul{L} ) - (s - 1) \cdot R ( e_1, e_2, L, \gras v ) \\
&\qquad - (s - 1) \cdot R ( L, \mc{L}_1 e_2 - \mc{L}_2 e_1, L, \gras v ) \\
&= I_1 + I_2 + I_3 + I_4 \text{.}
\end{align*}
First, $I_1$ is simply $\sigma$.
Recalling the definition of $\mc{L}_a$, then
\[ I_2 = - (s - 1) \vind^{ab} \nasla_a v \cdot \beta_b \text{,} \qquad I_4 = - (s - 1)^2 \vind^{ac} \mind^{bd} \nasla_a v \nasla_b v \cdot \alpha_{cd} \text{.} \]
For $I_3$, we expand and use that $\operatorname{Ric} \equiv 0$:
\begin{align*}
I_3 &= (s - 1) \nasla_1 v \cdot R ( e_2, e_1, L, e_1 ) - (s - 1) \nasla_2 v \cdot R ( e_1, e_2, L, e_2 ) \\
&= \frac{1}{2} (s - 1) \nasla_1 v \cdot R ( e_2, L, L, \ul{L} ) - \frac{1}{2} (s - 1) \nasla_2 v \cdot R ( e_1, L, L, \ul{L} ) \\
&= - (s - 1) \vind^{ab} \nasla_a v \cdot \beta_b \text{.}
\end{align*}
Combining all the above results in \eqref{eq.cf_sigma}.

Finally, for $\ul{\beta}^\prime$ and $\ul{\beta}$, we once again expand:
\begin{align*}
\ul{\beta}^\prime_a &= \frac{1}{2} e^{-v} \cdot R ( \ul{L}, L, \ul{L} + \mc{M} L + 2 (s - 1) \gras v, e_a ) \\
&\qquad + \frac{1}{2} e^{-v} \mc{L}_a \cdot R ( \ul{L}, L, \ul{L} + \mc{M} L + 2 (s - 1) \gras v, L ) \\
&\qquad + (s - 1) e^{-v} \cdot R ( \gras v, L, \ul{L} + \mc{M} L + 2 (s - 1) \gras v, e_a ) \\
&\qquad + (s - 1) e^{-v} \mc{L}_a \cdot R ( \gras v, L, \ul{L} + \mc{M} L + 2 (s - 1) \gras v, L ) \\
&= J_1 + J_2 + J_3 + J_4 \text{.}
\end{align*}
We can then expand $J_1$ as
\begin{align*}
J_1 &= \frac{1}{2} e^{-v} \cdot R ( \ul{L}, L, \ul{L}, e_a ) + \frac{1}{2} e^{-v} \mc{M} \cdot R ( \ul{L}, L, L, e_a ) \\
&\qquad + (s - 1) e^{-v} \cdot R ( \ul{L}, L, \gras v, e_a ) \\
&= e^{-v} \ul{\beta}_a - (s - 1)^2 e^{-v} \gamma^{bc} \nasla_b v \nasla_c v \cdot \beta_a + 2 (s - 1) e^{-v} {}^\star \nasla_a v \cdot \sigma \text{.}
\end{align*}
Similar computations yield
\begin{align*}
J_2 &= 2 (s - 1) e^{-v} \nasla_a v \cdot \rho + 2 (s - 1)^2 e^{-v} \mind^{bc} \nasla_a v \nasla_b v \cdot \beta_c \text{,} \\
J_4 &= 2 (s - 1)^2 e^{-v} \mind^{bc} \nasla_a v \nasla_b v \cdot \beta_c + 2 (s - 1)^3 e^{-v} \mind^{bd} \mind^{ce} \nasla_a v \nasla_b v \nasla_c v \cdot \alpha_{de} \text{.}
\end{align*}
Finally, for the remaining term $J_3$, we decompose
\begin{align*}
J_3 &= (s - 1) e^{-v} \cdot R ( \gras v, L, \ul{L}, e_a ) + (s - 1) e^{-v} \mc{M} \cdot R ( \gras v, L, L, e_a ) \\
&\qquad + 2 (s - 1)^2 e^{-v} \cdot R ( \gras v, L, \gras v, e_a ) \\
&= J_{31} + J_{32} + J_{33} \text{.}
\end{align*}
The simplest term is $J_{32}$:
\[ J_{32} = - (s - 1)^3 e^{-v} \mind^{bd} \mind^{ce} \nasla_b v \nasla_c v \nasla_e v \cdot \alpha_{ad} \text{.} \]
Let $e_{a\star}$ denote the frame element which is not $e_a$ (i.e., $e_{a\star} = e_2$ if $e_a = e_1$, and vice versa).
With this notation, we can expand $J_{31}$ and $J_{33}$ as
\begin{align*}
J_{31} &= - (s - 1) e^{-v} \nasla_a v \cdot R ( L, e_a, \ul{L}, e_a ) - (s - 1) e^{-v} \nasla_{a\star} v \cdot R ( L, e_{a\star}, \ul{L}, e_a ) \\
&= - \frac{1}{4} (s - 1) e^{-v} \nasla_a v \cdot R ( L, \ul{L}, \ul{L}, L ) - \frac{1}{2} (s - 1) e^{-v} \nasla_{a\star} v \cdot R ( L, \ul{L}, e_{a\star}, e_a ) \\
&= (s - 1) e^{-v} \nasla_a v \cdot \rho - (s - 1) e^{-v} {}^\star \nasla_a v \cdot \sigma \text{,} \\
J_{33} &= 2 (s - 1)^2 e^{-v} \nasla_{a\star} v \nasla_a v \cdot R ( e_a, L, e_{a\star}, e_a ) \\
&\qquad + 2 (s - 1)^2 e^{-v} \nasla_{a\star} v \nasla_{a\star} v \cdot R ( e_{a\star}, L, e_{a\star}, e_a ) \\
&= 2 (s - 1)^2 e^{-v} \nasla_{a\star} v \nasla_a v \cdot \beta_{a\star} - 2 (s - 1)^2 e^{-v} \nasla_{a\star} v \nasla_{a\star} v \cdot \beta_a \\
&= 2 (s - 1)^2 e^{-v} \vind^{bc} {}^\star \nasla_a v \nasla_b v \cdot \beta_c \text{.}
\end{align*}
Finally, combining all the above, we obtain \eqref{eq.cf_betabar}.

\subsubsection{Covariant Derivatives}

It remains to prove the formula \eqref{eq.cf_nasla} for changes of covariant derivatives.
For this, we work in terms of corresponding coordinates transported from the initial sphere $\mc{S} = \mc{S}_1$, as described in Sections \ref{sect:prelim.fol} and \ref{sect:prelim.cfol}.

Let $\Samma^c_{ab}$ and $\Samma^{\prime c}_{ab}$ denote the Christoffel symbols for $\mind$ and $\mind^\prime$, with respect to these coordinates.
Since the coordinate vector fields are related via \eqref{eq.cf_X}, then
\begin{align*}
\Samma^{\prime c}_{ab} &= \frac{1}{2} \mind^{cd} ( \partial^\prime_a \mind^\prime_{db} + \partial^\prime_b \mind^\prime_{da} - \partial^\prime_d \mind^\prime_{ab} ) \\
&= \Samma^c_{ab} + \frac{1}{2} (s - 1) \mind^{cd} ( \nasla_a v \cdot L \mind_{db} + \nasla_b v \cdot L \mind_{da} - \nasla_d v \cdot L \mind_{ab} ) \\
&= \Samma^c_{ab} + (s - 1) \mind^{cd} ( \nasla_a v \cdot \chi_{db} + \nasla_b v \cdot \chi_{da} - \nasla_d v \cdot \chi_{ab} ) \text{,}
\end{align*}
since $\mf{L}_s \mind = 2 \chi$.
As a result, we see that
\begin{align*}
\nasla^\prime_a \Psi^\prime_{ u_1 \dots u_l } &= \partial^\prime_a \Psi^\prime_{ u_1 \dots u_l } - \sum_{i = 1}^l \Samma^{\prime c}_{a u_i} \Psi^\prime_{ u_1 \hat{c}_i u_l } \\
&= \partial_a \Psi^\prime_{u_1 \dots u_l} + (s - 1) \nasla_a v \cdot \mf{L}_s \Psi^\prime_{u_1 \dots u_l} - \sum_{i = 1}^l \Samma^c_{a u_i} \Psi^\prime_{u_1 \hat{c}_i u_l} \\
&\qquad - (s - 1) \mind^{cd} ( \nasla_a v \cdot \chi_{d u_i} + \nasla_{u_i} v \cdot \chi_{da} - \nasla_d v \cdot \chi_{a u_i} ) \Psi^\prime_{u_1 \hat{c}_i u_l} \text{,}
\end{align*}
where the notation $u_1 \hat{c}_i u_l$ is defined in the same manner as in \eqref{eq.comm}.
Recalling the definitions of $\nasla$ and $\nasla_s$ results in \eqref{eq.cf_nasla}, as desired.

\subsection{Proofs of Proposition \ref{thm.cfol.renorm} and Corollary \ref{thm.cfol.init}} \label{sect:cfol.renorm}

First of all, since
\[ H = \chi - s^{-1} \mind \text{,} \qquad H^\prime = \chi^\prime - s^{\prime -1} \mind^\prime \text{,} \]
\eqref{eq.cf_H} follows from \eqref{eq.cf_mind}, \eqref{eq.cf_chi}, and \eqref{eq.cf_sk}.
Similarly, for \eqref{eq.cf_Z}, we use \eqref{eq.cf_zeta}:
\footnote{Note that $\nasla$ and $\nabla$ act identically on scalar functions.}
\begin{align*}
Z^\prime_a &= s^\prime \zeta^\prime_a = e^{-v} s \mc{B}_1 [ s^{-1} Z_a + ( s - 1 ) s^{-2} \gamma^{bc} \nabla_b v \cdot H_{ac} - s^{-1} \nabla_a v ] \text{.}
\end{align*}
This immediately implies \eqref{eq.cf_Z}.
Next, since
\[ \ul{H} = s^{-1} \ul{\chi} + s^{-2} ( 1 - s^{-1} m_{\rm S} ) \mind \text{,} \qquad \ul{H}^\prime = s^{\prime -1} \ul{\chi}^\prime + s^{\prime -2} ( 1 - s^{\prime -1} m_{\rm S} ) \mind^\prime \text{,} \]
then \eqref{eq.cf_Hbar} follows from \eqref{eq.cf_mind}, \eqref{eq.cf_chibar}, and \eqref{eq.cf_sk}.

Continuing on to the curvature components, by \eqref{eq.cf_alpha} and \eqref{eq.cf_sk}, we have
\begin{align*}
A^\prime_{ab} &= s^{\prime 2} \alpha^\prime_{ab} = e^{-2v} s^2 \mc{B}_2 e^{2v} \alpha_{ab} \text{,}
\end{align*}
from which \eqref{eq.cf_A} follows.
Similarly, by \eqref{eq.cf_beta} and \eqref{eq.cf_sk},
\begin{align*}
B^\prime_a &= s^3 e^{-3v} \mc{B}_3 [ e^v \beta_a + s^{-2} (s - 1) \gamma^{bc} e^v \nabla_b v \cdot \alpha_{ac} ] \text{,}
\end{align*}
and \eqref{eq.cf_B} follows.
Finally, the identities \eqref{eq.cf_R} and \eqref{eq.cf_Bbar} are consequences of analogous computations using \eqref{eq.cf_rho}, \eqref{eq.cf_sigma}, and \eqref{eq.cf_betabar}.

For \eqref{eq.cf_nabla}, we again consider Christoffel symbols with respect to corresponding transported coordinates.
Let $\Gamma^c_{ab}$ and $\Gamma^{\prime c}_{ab}$ denote these Christoffel symbols, with respect to $\gamma$ and $\gamma^\prime$, respectively.
Since $\gamma$ is a rescaling of $\mind$ by a constant factor on each $S_t$, then $\Gamma^c_{ab}$ is equal to the corresponding Christoffel symbol $\Samma^c_{ab}$ with respect to $\mind$.
By similar reasoning, $\Gamma^{\prime c}_{ab} = \Samma^{\prime c}_{ab}$ as well, hence
\begin{align*}
\Gamma^{\prime c}_{ab} &= \Gamma^c_{ab} + s^{-2} (s - 1) \gamma^{cd} ( \nabla_a v \cdot H_{db} + \nabla_b v \cdot H_{da} - \nabla_d v \cdot H_{ab} ) \\
&\qquad + s^{-1} (s - 1) ( \nabla_a v \cdot \delta^c_b + \nabla_b v \cdot \delta^c_a - \gamma^{cd} \nabla_d v \cdot \gamma_{ab} ) \text{,}
\end{align*}
where we used the relation between $\Samma^{\prime c}_{ab}$ and $\Samma^c_{ab}$ within the proof of \eqref{eq.cf_nasla}.
Thus,
\begin{align*}
\nabla^\prime_a \Psi^\prime_{ u_1 \dots u_l } &= \partial^\prime_a \Psi^\prime_{ u_1 \dots u_l } - \sum_{i = 1}^l \Gamma^{\prime c}_{a u_i} \Psi^\prime_{ u_1 \hat{c}_i u_l } \\
&= \partial_a \Psi^\prime_{u_1 \dots u_l} + s^{-2} (s - 1) \nabla_a v \cdot \mf{L}_t \Psi^\prime_{u_1 \dots u_l} - \sum_{i = 1}^l \Gamma^c_{a u_i} \Psi^\prime_{u_1 \hat{c}_i u_l} \\
&\qquad - s^{-2} (s - 1) \gamma^{cd} \sum_{i = 1}^l ( \nabla_a v \cdot H_{d u_i} + \nabla_{u_i} v \cdot H_{da} - \nabla_d v \cdot H_{a u_i} ) \Psi^\prime_{u_1 \hat{c}_i u_l} \\
&\qquad + s^{-1} (s - 1) \sum_{i = 1}^l ( \nabla_{u_i} v \cdot \Psi^\prime_{u_1 \hat{a}_i u_l} - \gamma^{cd} \nabla_d v \cdot \gamma_{a u_i} \Psi^\prime_{u_1 \hat{c}_i u_l} ) \\
&\qquad + l s^{-1} (s - 1) \nabla_a v \cdot \Psi^\prime_{u_1 \dots u_l} \text{,}
\end{align*}
where we also applied \eqref{eq.cf_X}.
Recalling the definitions of $\nabla$ and $\nabla_t$ yields \eqref{eq.cf_nabla}.

\subsubsection{Initial Values}

Finally, to prove Corollary \ref{thm.cfol.init}, we assume now that $t = 0$.
Then, the identities \eqref{eq.cf_H_init}-\eqref{eq.cf_Hbar_init} follow immediately from \eqref{eq.cf_H}-\eqref{eq.cf_Hbar}.
Moreover, \eqref{eq.cf_H_deriv_init} follows from \eqref{eq.cf_H_init}, since $\gamma_0 = \gamma^\prime_0$, and hence $\nabla^\prime$ and $\nabla$ behave identically on $S_0 = S_0^\prime$.
Since Proposition \ref{thm.cfol.renorm} also implies that when $t = 0$,
\[ R^\prime = R \text{,} \qquad \hat{H}^\prime_{ab} = e^v \hat{H}_{ab} \text{,} \qquad \hat{\ul{H}}^\prime_{ab} = e^{-v} \hat{\ul{H}}_{ab} \text{,} \qquad \gamma^{\prime ab} \nabla^\prime_a Z^\prime_b = \gamma^{ab} \nabla_a Z_b + \lapl v \text{,} \]
combining this with the definition \eqref{eq.renorm_maf} yields \eqref{eq.cf_M_init}.

\section{Some Proofs of Estimates} \label{sect:estimates}

In this appendix, we prove some estimates needed in the proof of Theorem \ref{the.thm} which are more technical, in the sense that they require more machinery from \cite{alex_shao:nc_inf, shao:stt}.

\subsection{Additional Definitions and Results} \label{sect:estimates.norms}

For a few estimates, we will require one more addition to the general formalism described in Section \ref{sect:prelim.fol}; for simplicitly, we deal exclusively with the renormalized setting.
As in \cite{alex_shao:nc_inf, shao:stt}, for a fixed $t_0 \in [0, 1)$, we define $\cint_{t_0}^t \Psi$ to be the \emph{definite covariant ($t$-)integral from $S_{t_0}$}, i.e., the (unique) horizontal tensor field $\Psi$ which vanishes on $S_{t_0}$ and satisfies $\nabla_t \cint_{t_0}^t \Psi = \Psi$.

We will also require some additional norms, used throughout \cite{alex_shao:nc_inf}.
All definitions will be with respect to a renormalized system $(\mc{N}, \gamma)$.
\begin{itemize}
\item Define the $N^{1 i}_{t, x}$-norm on horizontal tensor fields to be the first-order Sobolev norm on $\mc{N}$, along with a measure of ``initial data":
\[ \| \Psi \|_{ N^{1 i}_{t, x} } = \| \nabla_t \Psi \|_{ L^{2, 2}_{t, x} } + \| \nabla \Psi \|_{ L^{2, 2}_{t, x} } + \| \Psi \|_{ L^{2, 2}_{t, x} } + \| \Psi \|_{ H^{1/2}_x (S_0) } \text{.} \]

\item The $N^{0 \star}_{t, x}$-norm is defined
\[ \| \Psi \|_{ N^{0 \star}_{t, x} } = \inf \{ \| \Phi \|_{ N^{1 i}_{t, x} } \mid \nabla_t \Phi = \Psi \} \text{,} \]
and measures the smallest $N^{1 i}_{t, x}$-norm of any covariant $t$-antiderivative of $\Psi$.

\item Finally, the ``sum" norm, $N^{0 \star}_{t, x} + B^{2, 0}_{t, x}$, measures the ``smallest" way in which a horizontal tensor field can be decomposed into a sum in $N^{0 \star}_{t, x}$ and $B^{2, 0}_{t, x}$:
\[ \| \Psi \|_{ N^{0 \star}_{t, x} + B^{2, 0}_{t, x} } = \inf \{ \| \Psi_1 \|_{ N^{0 \star}_{t, x} } + \| \Psi_2 \|_{ B^{2, 0}_{t, x} } \mid \Psi_1 + \Psi_2 = \Psi \} \text{.} \]
\end{itemize}
For detailed discussions behind these norms, see \cite[Sect. 3.3]{alex_shao:nc_inf}.

These decomposition norms enter our analysis via the main theorem of \cite{alex_shao:nc_inf}.
Indeed, there are some additional estimates in \cite[Thm. 5.1]{alex_shao:nc_inf} featuring these decomposition norms, which were omitted from Theorem \ref{thm.nc_renorm}:

\begin{proposition} \label{thm.nc_renorm_ex}
Assume the hypotheses of Theorem \ref{thm.nc_renorm}.
Then, in addition to \eqref{eq.renorm_est} and \eqref{eq.renorm_est_curv}, we have the following estimates:
\begin{equation} \label{eq.renorm_est_ex} \| \nabla H \|_{ N^{0 \star}_{t, x} + B^{2, 0}_{t, x} } + \| \nabla Z \|_{ N^{0 \star}_{t, x} + B^{2, 0}_{t, x} } + \| \nabla_t \ul{H} \|_{ N^{0 \star}_{t, x} + B^{2, 0}_{t, x} } \lesssim \Gamma \text{.} \end{equation}
\end{proposition}

\begin{proof}
See \cite[Thm. 5.1]{alex_shao:nc_inf}.
\end{proof}

The inequality \eqref{eq.renorm_est_ex}, in particular that for $\nabla H$, will be useful in the context of the following integrated product estimate from \cite[Cor. 3.10]{alex_shao:nc_inf}; see \cite{shao:stt} for details.

\begin{proposition} \label{thm.renorm_est_ex}
Assume the hypotheses of Theorem \ref{thm.nc_renorm}.
Then,
\begin{equation} \label{eq.trace_est_ex} \| \cint_0^t ( \Phi \otimes \Psi ) \|_{ B^{\infty, 0}_{t, x} } \lesssim \| \Phi \|_{ N^{0 \star}_{t, x} + B^{2, 0}_{t, x} } \| \Psi \|_{ N^{1 i}_{t, x} \cap L^{\infty, 2}_{x, t} } \text{.} \end{equation}
\end{proposition}

\subsection{Proof of Propositions \ref{thm.renorm_cfol} and \ref{thm.renorm_cfol_bound}} \label{sect:estimates.cfol.small}

We begin with the last inequality in \eqref{eq.renorm_cfol_bound}.
Considering $A^\prime$, $B^\prime$, $R^\prime$, $B^\prime$ as $t$-horizontal tensor fields (see Section \ref{sect:prelim.cfol}), then we must bound the right-hand sides of \eqref{eq.cf_A}-\eqref{eq.cf_Bbar} in the $L^{2, 2}_{t, x}$-norm.
This is a direct application of \eqref{eq.renorm_est}, \eqref{eq.renorm_cfol_ass}, and H\"older's inequality.
For example,
\[ \| B^\prime \|_{ L^{2, 2}_{t, x} } \lesssim \| B \|_{ L^{2, 2}_{t, x} } + \| A \|_{ L^{2, 2}_{t, x} } \lesssim \Gamma \text{,} \]
where we applied \eqref{eq.cf_B}.
Note that we used \eqref{eq.renorm_cfol_ass} and the smallness of $\Gamma$ in order to uniformly bound various instances of $v$ and $\nabla v$ within \eqref{eq.cf_B}.
The remaining coefficients $A^\prime$, $R^\prime$, $\ul{B}^\prime$ can be similarly controlled.

Recalling \eqref{eq.renorm_cfol_ass} and taking into account the discussions in Section \ref{sect:prelim.cfol.small} following Proposition \ref{thm.renorm_cfol}, we see that iterated Lebesgue norms with respect to the $t$- and $t^\prime$-foliations of $\mc{N}$ (applied to corresponding horizontal fields) are comparable.
Thus, from the $t$-foliation estimates on $A^\prime$, $B^\prime$, $R^\prime$, and $\ul{B}^\prime$, we conclude
\[ \| A^\prime \|_{ L^{2, 2}_{t^\prime, x} (\gamma^\prime) } + \| B^\prime \|_{ L^{2, 2}_{t^\prime, x} (\gamma^\prime) } + \| R^\prime \|_{ L^{2, 2}_{t^\prime, x} (\gamma^\prime) } + \| \ul{B}^\prime \|_{ L^{2, 2}_{t^\prime, x} (\gamma^\prime) } \lesssim \Gamma \text{.} \]
In other words, \eqref{eq.renorm_ass_flux} also holds in the $t^\prime$-foliation.

Next, we show that \eqref{eq.renorm_ass_init} also remains true in the $t^\prime$-foliation.
The keys are to note that $\gamma^\prime_0$ and $\gamma_0$ are identical, and to use the formulas in Proposition \ref{thm.cfol.init} to express $H^\prime$, $Z^\prime$, $\ul{H}^\prime$, and $M^\prime$ on $S^\prime_0 = S_0$ in terms of the $t$-foliation.
First, by \eqref{eq.cf_H_init}, along with \eqref{eq.renorm_ass_init} and \eqref{eq.renorm_cfol_ass}, we can estimate
\[ \| \trace^\prime H^\prime \|^\prime_{ L^\infty_x (S^\prime_0, \gamma^\prime_0) } \lesssim \| \trace H \|_{ L^\infty_x (S_0, \gamma_0) } + \| e^v - 1 \|_{ L^\infty_x (S_0, \gamma_0) } \lesssim \Gamma \text{.} \]
Furthermore, by \eqref{eq.renorm_cfol_ass}, along with the product estimates of \cite[Cor. 3.7]{shao:stt},
\begin{align*}
\| H^\prime \|_{ H^{1/2}_x (S^\prime_0, \gamma^\prime_0) } &\lesssim \| e^v H \|_{ H^{1/2}_x (S_0, \gamma_0) } + \| e^v - 1 \|_{ H^{1/2}_x (S_0, \gamma_0) } \\
&\lesssim \| H \|_{ H^{1/2}_x (S_0, \gamma_0) } + \| e^v - 1 \|_{ H^1_x (S_0, \gamma_0) } \text{.}
\end{align*}
Since the $H^1_x$-norm, defined in Section \ref{sect:prelim.norms}, is equivalent to the standard norm,
\footnote{See \cite[Sect. 2.3]{shao:stt}.}
\[ \| F \|_{ H^1_x (S_0) } \simeq \| \nabla F \|_{ L^2_x (S_0) } + \| F \|_{ L^2_x (S_0) } \text{,} \]
then \eqref{eq.renorm_ass_init} and \eqref{eq.renorm_cfol_ass} imply that
\[ \| H^\prime \|_{ H^{1/2}_x (S^\prime_0, \gamma^\prime_0) } \lesssim \Gamma \text{.} \]

By similar estimates using \eqref{eq.cf_Z_init}-\eqref{eq.cf_M_init}, we derive
\begin{align*}
\| Z^\prime \|_{ H^{1/2}_x (S^\prime_0) } &\lesssim \| Z \|_{ H^{1/2}_x (S_0) } + \| \nabla v \|_{ H^1_x (S_0) } \lesssim \Gamma \text{,} \\
\| \ul{H}^\prime \|_{ B^0_x (S^\prime_0) } &\lesssim \| e^{-v} \ul{H} \|_{ B^0_x (S_0) } + \| e^{-v} - 1 \|_{ H^1_x (S_0) } \lesssim \Gamma \text{,} \\
\| \nabla^\prime ( \trace^\prime H^\prime ) \|_{ B^0_x (S^\prime_0) } &\lesssim \| \nabla ( \trace H ) \|_{ B^0_x (S_0) } + \| \trace H \|_{ B^0_x (S_0) } + \| \nabla v \|_{ H^1_x (S_0) } \lesssim \Gamma \text{,} \\
\| M^\prime \|_{ B^0_x (S^\prime_0) } &\lesssim \| M \|_{ B^0_x (S_0) } + \| \lapl v \|_{ B^0_x (S_0) } \lesssim \Gamma \text{.}
\end{align*}
where we used \eqref{eq.renorm_ass_init}, \eqref{eq.renorm_cfol_ass}, \cite[Cor. 3.7]{shao:stt}, and the observation that the $B^0_x$-norm is bounded by the $H^1_x$-norm (see \cite[Prop. 2.2]{shao:stt}).
The preceding estimates imply that \eqref{eq.renorm_ass_init} indeed holds true with respect to the $t^\prime$-foliation.

Thus, with $\Gamma$ sufficiently small, we that the hypotheses, and hence the conclusions, of Theorem \ref{thm.nc_renorm} hold with respect to the $t^\prime$-foliation.
This completes the proof of Proposition \ref{thm.renorm_cfol}.
Appealing once again to the comparability of Lebesgue norms in the $t$- and $t^\prime$-foliations yields \eqref{eq.renorm_cfol_bound}, which proves Proposition \ref{thm.renorm_cfol_bound}.

\subsection{Proof of Lemmas \ref{thm.limit_lem_M} and \ref{thm.limit_lem_ZHbar}} \label{sect:estimates.limit}

We proceed like in \cite[Sect. 5.2]{alex_shao:nc_inf}.
While the basic ideas are simple, some extra care must be taken to state them correctly.
Fix an arbitrary bounded vector field $\bar{X}$ on $S_0$, and define the following:
\begin{itemize}
\item Extend $\bar{X}$ to a $t$-horizontal vector field $X$ on $\mc{N}$ by \emph{equivariant transport}, that is, by the condition $\mf{L}_t X \equiv 0$.
\footnote{In particular, note that $\Phi^\ast_{ S_t } ( X_t ) = \bar{X}$ for any $t$.}

\item Similarly, for each $y \in [0, 1)$, we extend $\bar{X}$ as a $\fol{t}{y}$-horizontal vector field $\fol{X}{y}$ on $\mc{N}$ by the analogous condition $\mf{L}_{ \fol{t}{y} } \fol{X}{y} \equiv 0$.
\end{itemize}
Observe in addition that, for the same reasons as for corresponding transported coordinate vector fields, $\fol{X}{y}$ and $X$ are related via \eqref{eq.cf_X}.

\subsubsection{Proof of \eqref{eq.limit_lem_ZHbar}}

For the first inequality in \eqref{eq.limit_lem_ZHbar}, it suffices to show
\begin{align*}
{\bf L}_Z^\prime &= \int_{ S_0 } | [ \Phi_{ \fol{\Sigma}{y_2} }^\ast ( \fol{Z}{y_2}_{y_2} ) - \Phi_{ \fol{\Sigma}{y_1} }^\ast ( \fol{Z}{y_1}_{y_1} ) ] (\bar{X}) |^2 d \epsilon_0 \\
&= \int_{ S_0 } | \Phi_{ \fol{\Sigma}{y_2} }^\ast [ \fol{Z}{y_2}_{y_2} ( \fol{X}{y_2}_{y_2} ) ] - \Phi_{ \fol{\Sigma}{y_1} }^\ast [ \fol{Z}{y_1}_{y_1} (\fol{X}{y_1}_{y_1} ) ] |^2 d \epsilon_0 \text{.}
\end{align*}
is controlled by the right-hand side of this inequality (with constant also depending on $X$).
As we are comparing the $\fol{Z}{y_i}$'s on different spheres $\fol{\Sigma}{y_i}$'s, the first step is to pull $\fol{Z}{y_2}$ from $\fol{\Sigma}{y_2}$ to $\fol{\Sigma}{y_1}$.
Consider points $P_i \in \fol{\Sigma}{y_i}$ which lie on a common null generator of $\mc{N}$.
Since $\mf{L}_{ \fol{t}{y_2} } \fol{X}{y_2} = 0$, it follows that
\footnote{Although the integral on the right-hand side is, technically, a covariant integral as defined in Appendix \ref{sect:estimates.norms}, since we are dealing with scalar quantities, this coincides with the usual integral.}
\[ \fol{Z}{y_2} ( \fol{X}{y_2} ) |_{ P_1 } = \fol{Z}{y_2} ( \fol{X}{y_2} ) |_{ P_2 } - [ \cint_{ \fol{t}{y_2} (P_1) }^{ \fol{t}{y_2} } ( \mf{L}_{ \fol{t}{y_2} } \fol{Z}{y_2} ) ( \fol{X}{y_2} ) ] |_{ P_2 } \text{.} \]
Moreover, since the $\fol{X}{y}$'s and $X$ are related via \eqref{eq.cf_X}, it follows that
\[ \fol{Z}{y_2} ( \fol{X}{y_2} ) |_{ P_1 } = \fol{Z}{y_2} ( \fol{X}{y_1} ) |_{ P_1 } \text{,} \]
where on the right-hand side, we treated $\fol{Z}{y_2}$ as a $\fol{t}{y_1}$-horizontal field.

Thus, combining the above and keeping in mind the comparability of all the renormalized metrics involved (see the discussion in Section \ref{sect:prelim.cfol.small}), we obtain that
\begin{align*}
{\bf L}_Z^\prime &\lesssim \int_{ S_0 } | \Phi_{ \fol{\Sigma}{y_1} }^\ast [ ( \fol{Z}{y_2}_{ \fol{\Sigma}{y_1} } - \fol{Z}{y_1}_{y_1} ) (\fol{X}{y_1}_{y_1} ) ] |^2 d \epsilon_0 + \| \mf{L}_{ \fol{t}{y_2} } \fol{Z}{y_2} \|_{ L^{2, 1}_{x, \fol{t}{y_2}} (\fol{\Sigma}{y_1}, \fol{\Sigma}{y_2}) } = I_1 + I_2 \text{,}
\end{align*}
where $\fol{Z}{y_2}_{ \fol{\Sigma}{y_1} }$ denotes the restricted of $\fol{Z}{y_2}$ to $\fol{\Sigma}{y_1}$, treated as a $\fol{t}{y_1}$-horizontal vector field.
Again, due to the comparability of all the renormalized metrics,
\[ I_1 \lesssim \| \fol{Z}{y_2} - \fol{Z}{y_1} \|_{ L^2_x (\fol{\Sigma}{y_1}, \fol{\gamma}{y_1}_{y_1}) }^2 \lesssim \| \fol{Z}{y_2} - \fol{Z}{y_1} \|_{ L^{2, \infty}_{x, t} (\fol{\Sigma}{y_1}, S_1) }^2 \text{.} \]
Furthermore, from the definition of $\nabla_{ \fol{t}{y_2} }$ and by H\"older's inequality, we can estimate
\[ I_2 \lesssim \| \fol{\nabla}{y_2}_{ \fol{t}{y_2} } \fol{Z}{y_2} \|_{ L^{2, 1}_{x, \fol{t}{y_2}} (\fol{\Sigma}{y_1}, \fol{\Sigma}{y_2}) }^2 + \| \fol{H}{y_2} \|_{ L^{\infty, 2}_{x, \fol{t}{y_2}} (\fol{\Sigma}{y_1}, \fol{\Sigma}{y_2}) }^2 \| \fol{Z}{y_2} \|_{ L^{2, 1}_{x, \fol{t}{y_2}} (\fol{\Sigma}{y_1}, \fol{\Sigma}{y_2}) }^2 \text{.} \]
This proves the first inequality in \eqref{eq.limit_lem_ZHbar}.
The remaining bound in \eqref{eq.limit_lem_ZHbar} is similarly proved by contracting the $\fol{\ul{H}}{y}$'s with two equivariantly transported vector fields.

\subsubsection{Proof of \eqref{eq.limit_lem_M}}

In this case, one has an additional convenience: since the $\fol{M}{y}$'s are scalar, we need not involve contractions with other vector fields.
First, by an analogous argument as for ${\bf L}_Z$, we obtain
\begin{align*}
{\bf L}_M &\lesssim \int_{ S_0 } | \Phi_{ \fol{\Sigma}{y_1} }^\ast ( \fol{M}{y_2}_{ \fol{\Sigma}{y_1} } - \fol{M}{y_1}_{y_1} ) | d \epsilon_0 + \| \mf{L}_{ \fol{t}{y_2} } \fol{M}{y_2} \|_{ L^{1, 1}_{\fol{t}{y_2}, x} (\fol{\Sigma}{y_1}, \fol{\Sigma}{y_2}) } = J_1 + J_2 \text{.}
\end{align*}
Since $\nabla_{ \fol{t}{y_2} }$ and $\mf{L}_{ \fol{t}{y_2} }$ act identically on scalar fields,
\[ J_2 = \| \fol{\nabla}{y_2}_{ \fol{t}{y_2} } \fol{M}{y_2} \|_{ L^{1, 1}_{\fol{t}{y_2}, x} (\fol{\Sigma}{y_1}, \fol{\Sigma}{y_2}) } \text{.} \]

To handle $J_1$, we pull $\fol{M}{y_2} - \fol{M}{y_1}$ from $\fol{\Sigma}{y_1}$ to $S_0$.
If $P_0 \in S_0$ and $P_1 \in \fol{\Sigma}{y_1}$ lie on the same null generator of $\mc{N}$, then as before,
\begin{align*}
\fol{M}{y_1} |_{ P_1 } &= \fol{M}{y_1} |_{ P_0 } - [ \cint_0^{ \fol{t}{y_1} } \fol{\nabla}{y_1}_{ \fol{t}{y_1} } \fol{M}{y_1} ] |_{ P_1 } \text{,} \\
\fol{M}{y_2} |_{ P_1 } &= \fol{M}{y_2} |_{ P_0 } - [ \cint_0^{ \fol{t}{y_2} } \fol{\nabla}{y_2}_{ \fol{t}{y_2} } \fol{M}{y_2} ] |_{ P_1 } \text{.}
\end{align*}
Therefore, we can bound
\[ J_1 \lesssim \int_{ S_0 } | \fol{M}{y_2}_0 - \fol{M}{y_1}_0 | d \epsilon_0 + \| \fol{\nabla}{y_2}_{ \fol{t}{y_2} } \fol{M}{y_2} - \fol{\nabla}{y_1}_{ \fol{t}{y_1} } \fol{M}{y_1} \|_{ L^{1, 1}_{t, x} } \text{.} \]
Combining the above completes the proof of \eqref{eq.limit_lem_M}.

\subsection{Transport Estimates} \label{sect:estimates.transport}

In Section \ref{sect:distortion.construct}, a common step is to solve for a function, say $v$, on a level sphere $(S_y, \gamma_y)$ and to then extend $v$ to $\mc{N}$ by the condition $\nabla_t v \equiv 0$.
If $v$ is bounded on $S_y$, then $v$ is trivially bounded on all of $\mc{N}$.
However, this becomes less trivial for covariant derivatives of $v$, since the connections $\nabla$ now depend on the metrics $\gamma_t$.
Here, we prove some properties stating that, in the appropriate norms, this change of metric will not affect the estimates.

\begin{lemma} \label{thm.est_v}
Assume that the hypotheses of Theorem \ref{thm.nc_renorm} hold.
Let $v$ be a smooth function on $\mc{N}$ satisfying $\nabla_t v \equiv 0$, i.e., $v$ is constant on the null generators of $\mc{N}$.
In addition, fix $q \in (2, \infty]$, $p = \frac{2q}{q + 2} \in (1, 2]$, and $y \in [0, 1)$, and assume
\begin{equation} \label{eq.est_v_ass} \| \nabla^2 v \|_{ L^p_x (S_y) } + \| \nabla v \|_{ L^q_x (S_y) } + \| v \|_{ L^\infty_x (S_y)} \lesssim D \text{,} \end{equation}
for some constant $D$.
Then, the following estimates hold for $v$ on all of $\mc{N}$: 
\begin{align}
\label{eq.est_v} \| \nabla_t \nabla^2 v \|_{ L^{p, 2}_{x, t} } + \| \nabla_t \nabla v \|_{ L^{q, 2}_{x, t} } &\lesssim \Gamma D \text{,} \\
\notag \| \nabla^2 v \|_{ L^{p, \infty}_{x, t} } + \| \nabla v \|_{ L^{q, \infty}_{x, t} } + \| v \|_{ L^{\infty, \infty}_{t, x} } &\lesssim D \text{.}
\end{align}
\end{lemma}

\begin{proof}
The $L^{\infty, \infty}_{t, x}$-bound for $v$ is trivial, while the $L^{q, \infty}_{x, t}$-bound for $\nabla v$ follows immediately from \cite[Prop. 4.12]{shao:stt}, since $v$ is scalar.
Furthermore, since
\[ \nabla_t \nabla_a v = - \gamma^{cd} H_{ac} \nabla_d v \]
by \eqref{eq.comm}, then
\[ \| \nabla_t \nabla v \|_{ L^{q, 2}_{x, t} } \lesssim \| H \|_{ L^{\infty, 2}_{x, t} } \| \nabla v \|_{ L^{q, \infty}_{x, t} } \lesssim \Gamma D \text{.} \]

The estimates for $\nabla^2 v$ are derived analogously, although we must perform the steps manually rather than rely on \cite{shao:stt}.
First, applying \eqref{eq.comm} twice yields
\footnote{Recall the second fundamental form $k$ in the renormalized setting is precisely $H$.}
\begin{align*}
\nabla_t \nabla_{ab} v_r &= - \gamma^{cd} \nabla_a ( H_{bc} \nabla_d v_r ) - \gamma^{cd} H_{ac} \nabla_{db} v_r - \gamma^{cd} (\nabla_b H_{ac} - \nabla_c H_{ab}) \nabla_d v_r \\
&= - \gamma^{cd} ( H_{ac} \nabla_{db} v_r + H_{bc} \nabla_{da} v_r ) - \gamma^{cd} ( \nabla_a H_{bc} + \nabla_b H_{ac} - \nabla_c H_{ab} ) \nabla_d v_r \text{.}
\end{align*}
As a result, for each $\tau \in [0, 1)$ and $x \in \Sph^2$, we can bound
\begin{align*}
| \nabla^2 v | |_{ (\tau, x) } &\lesssim | \nabla^2 v | |_{ (y, x) } + | \cint^t_y ( H \otimes \nabla^2 v ) | |_{ (\tau, x) } + | \cint^t_y ( \nabla H \otimes \nabla v ) | |_{ (\tau, x) } \\
&\lesssim | \nabla^2 v | |_{ (y, x) } + \left| \int_y^\tau | H | | \nabla^2 v | |_{ (\tau^\prime, x) } d \tau^\prime \right| + \left| \int_y^\tau | \nabla H | | \nabla v | |_{ (\tau^\prime, x) } d \tau^\prime \right| \text{.}
\end{align*}
Taking a supremum over $\tau$ and then an $L^p$-norm over $x$ (and applying \eqref{eq.renorm_est}) yields
\begin{align*}
\| \nabla^2 v \|_{ L^{p, \infty}_{x, t} } &\lesssim \| \nabla^2 v \|_{ L^p_x (S_y) } + \| H \|_{ L^{\infty, 1}_{x, t} } \| \nabla^2 v \|_{ L^{p, \infty}_{x, t} } + \| \nabla H \|_{ L^{2, 1}_{x, t} } \| \nabla v \|_{ L^{q, \infty}_{x, t} } \\
&\lesssim D + \Gamma \| \nabla^2 v \|_{ L^{p, \infty}_{x, t} } + \Gamma D \text{.}
\end{align*}
Since $\Gamma$ is small, we obtain the desired estimate for $\nabla^2 v$.
Finally, we can bound
\begin{align*}
\| \nabla_t \nabla^2 v \|_{ L^{p, 2}_{x, t} } \lesssim \| H \|_{ L^{\infty, 2}_{x, t} } \| \nabla^2 v \|_{ L^{p, \infty}_{x, t} } + \| \nabla H \|_{ L^{2, 2}_{t, x} } \| \nabla v \|_{ L^{q, \infty}_{x, t} } \lesssim \Gamma D \text{.}
\end{align*}
This completes the proof of \eqref{eq.est_v}.
\end{proof}

\begin{remark}
One also can prove variants of Lemma \ref{thm.est_v}, applying over only a portion of $\mc{N}$.
In particular, given any spherical cut $\Sigma$ of $\mc{N}$, by following through most of the proof of Lemma \ref{thm.est_v}, one obtains the estimate
\begin{align}
\label{eq.est_v_partial} \| \nabla^2 v \|_{ L^{p, \infty}_{x, t} (S_y, \Sigma) } &\lesssim \| \nabla^2 v \|_{ L^p_x (S_y) } + \| H \|_{ L^{\infty, 1}_{x, t} (S_y, \Sigma) } \| \nabla^2 v \|_{ L^{p, \infty}_{x, t} (S_y, \Sigma) } \\
\notag &\qquad + \| \nabla H \|_{ L^{2, 1}_{x, t} (S_y, \Sigma) } \| \nabla v \|_{ L^{q, \infty}_{x, t} (S_y, \Sigma) } \text{.}
\end{align}
Note one can also take $\Sigma = S_1$ in \eqref{eq.est_v_partial}.
\end{remark}

We also require the following variant of Lemma \ref{thm.est_v}.

\begin{lemma} \label{thm.est_v_ex}
Assume that the hypotheses of Theorem \ref{thm.nc_renorm} hold.
Let $v$ be a smooth function on $\mc{N}$ satisfying $\nabla_t v \equiv 0$.
Fix $y \in [0, 1)$, and assume
\begin{equation} \label{eq.est_v_ex_ass} \| \nabla^2 v \|_{ B^0_x (S_y) } + \| \nabla v \|_{ L^\infty_x (S_y) } + \| v \|_{ L^\infty_x (S_y)} \lesssim D \text{,} \end{equation}
for some constant $D$.
Then, the following estimates hold for $v$ on all of $\mc{N}$:
\begin{equation} \label{eq.est_v_ex} \| \nabla^2 v \|_{ B^{\infty, 0}_{t, x} \cap L^{2, \infty}_{x, t} } + \| \nabla v \|_{ L^{\infty, \infty}_{t, x} } + \| v \|_{ L^{\infty, \infty}_{t, x} } \lesssim D \text{.} \end{equation}
\end{lemma}

\begin{proof}
By Lemma \ref{thm.est_v}, the only estimate left to prove is the $B^{\infty, 0}_{t, x}$-bound for $\nabla^2 v$.
The first step is to obtain a Besov estimate for $\nabla^2 v$ at $S_0$:
\footnote{One hidden step in the estimate below is the equivalence of the $B^0_x$-norms on the various $S_t$'s.  This can be shown using special $t$-parallel frames; see \cite[Prop. 5.2]{shao:stt} and \cite[Sect. 3.5]{shao:stt}.}
\begin{align*}
\| \nabla^2 v \|_{ B^0_x (S_0) } &\lesssim \| \nabla^2 v \|_{ B^0_x (S_y) } + \| \cint^t_0 \nabla_t \nabla^2 v \|_{ B^{\infty, 0}_{t, x} } \\
&\lesssim D + \| \cint^t_0 ( H \otimes \nabla^2 v ) \|_{ B^{\infty, 0}_{t, x} } + \| \cint^t_0 ( \nabla H \otimes \nabla v ) \|_{ B^{\infty, 0}_{t, x} } \text{.}
\end{align*}
Applying the integrated product estimate from \cite[Thm. 5.2]{alex_shao:nc_inf} with \eqref{eq.renorm_est} yields
\footnote{Alternatively, one can use \eqref{eq.trace_est_ex} to arrive at the same result.}
\[ \| \cint^t_0 ( H \otimes \nabla^2 v ) \|_{ B^{\infty, 0}_{t, x} } \lesssim ( \| \nabla H \|_{ L^{2, 2}_{t, x} } + \| H \|_{ L^{\infty, 2}_{x, t} } ) \| \nabla^2 v \|_{ B^{2, 0}_{t, x} } \lesssim \Gamma \| \nabla^2 v \|_{ B^{\infty, 0}_{t, x} } \text{.} \]

Next, applying \eqref{eq.renorm_est_ex} and \eqref{eq.trace_est_ex}, we obtain
\begin{align*}
\| \cint^t_0 ( \nabla H \otimes \nabla v ) \|_{ B^{\infty, 0}_{t, x} } &\lesssim \| \nabla H \|_{ N^{0 \star}_{t, x} + B^{2, 0}_{t, x} } \| \nabla v \|_{ N^{1i}_{t, x} \cap L^{\infty, 2}_{t, x} } \\
&\lesssim \Gamma [ \| \nabla_t \nabla v \|_{ L^{2, 2}_{t, x} } + \| \nabla^2 v \|_{ L^{2, 2}_{t, x} } + \| \nabla v \|_{ L^{2, 2}_{t, x} } + \| \nabla v \|_{ H^{1/2}_x (S_0) } ] \text{.}
\end{align*}
Applying Lemma \ref{thm.est_v} yields
\[ \| \cint^t_0 ( \nabla H \otimes \nabla v ) \|_{ B^{\infty, 0}_{t, x} } \lesssim \Gamma ( \| \nabla_t \nabla v \|_{ L^{\infty, 2}_{x, t} } + \| \nabla^2 v \|_{ L^{2, \infty}_{x, t} } + \| \nabla v \|_{ L^{\infty, \infty}_{t, x} } ) \lesssim \Gamma D \text{.} \]
Combining the above, it follows that
\[ \| \nabla^2 v \|_{ B^0_x (S_0) } \lesssim D + \Gamma \| \nabla^2 v \|_{ B^{\infty, 0}_{t, x} } \text{.} \]

We can now go from $S_0$ to any $S_t$.
By a similar process as above,
\begin{align*}
\| \nabla^2 v \|_{ B^{\infty, 0}_{t, x} } &\lesssim \| \nabla^2 v \|_{ B^0_x (S_0) } + \| \cint^t_0 ( H \otimes \nabla^2 v ) \|_{ B^{\infty, 0}_{t, x} } + \| \cint^t_0 ( \nabla H \otimes \nabla v ) \|_{ B^{\infty, 0}_{t, x} } \\
&\lesssim D + \Gamma \| \nabla^2 v \|_{ B^{\infty, 0}_{t, x} } + \Gamma D \text{.}
\end{align*}
Recalling that $\Gamma$ is very small completes the proof of \eqref{eq.est_v_ex}.
\end{proof}

\subsection{Proof of Lemma \ref{thm.est_v1}} \label{sect:estimates.v1}

To control the full second derivative of $\fol{v}{y}_1$, we appeal to the Hodge estimates of \cite{shao:stt}.
More specifically, consider the Hodge operator
\[ \mc{D}_1 \xi = \gamma^{ab} \nabla_a \xi_b - i \epsilon^{ab} \nabla_a \xi_b \text{,} \]
defined on horizontal $1$-forms.
\footnote{See \cite[Sect. 2.1]{alex_shao:nc_inf} or \cite[Sect. 2.1]{shao:stt} for details; see also \cite{chr_kl:stb_mink, kl_rod:cg}.}
In particular, since $\mc{D}_1 \nabla \fol{v}{y}_1 = \lapl \fol{v}{y}_1$, applying the Hodge estimates of \cite[Sect. 6.2]{shao:stt} results in the bound
\[ \| \nabla^2 \fol{v}{y}_1 \|_{ L^2_x (S_y) } + \| \nabla \fol{v}{y}_1 \|_{ L^2_x (S_y) } \lesssim \| \lapl \fol{v}{y}_1 \|_{ L^2_x (S_y) } \lesssim \Gamma \text{.} \]
Furthermore, by Poincar\'e's inequality,
\footnote{This is, in fact, a special case of the Hodge estimates in \cite[Prop. 6.5]{shao:stt}, with operator $\mc{D}_1^\ast$.
See the remark following \cite[Prop. 6.5]{shao:stt} for further details.}
\[ \| \fol{v}{y}_1 \|_{ L^2_x (S_y) } \lesssim \| \nabla \fol{v}{y}_1 \|_{ L^2_x (S_y) } \lesssim \Gamma \text{.} \]

Similar elliptic estimates hold for Besov norms; by \cite[Thm. 6.11]{shao:stt}, we have
\footnote{By the usual manipulations described in \cite{shao:stt} (i.e., considering a foliation with an equivariantly transported horizontal metric), \cite[Thm. 6.11]{shao:stt} is also applicable to single spheres.}
\[ \| \nabla^2 \fol{v}{y}_1 \|_{ B^0_x (S_y) } \lesssim \| \lapl \fol{v}{y}_1 \|_{ B^0_x (S_y) } \lesssim \Gamma \text{.} \]
Combining this with the $L^\infty$-embeddings in \cite[Prop. 2.7, Thm. 6.11]{shao:stt} yields
\begin{equation} \label{sharp.v1} \| \nabla^2 \fol{v}{y}_1 \|_{ B^0_x (S_y) } + \| \nabla \fol{v}{y}_1 \|_{ L^\infty_x (S_y) } + \| \fol{v}{y}_1 \|_{ L^\infty_x (S_y) } \lesssim \Gamma \text{.} \end{equation}
To extend \eqref{sharp.v1} to all of $\mc{N}$, we must deal with the changing geometries of the $\gamma_t$'s.
For this, we take advantage of the transport equation $\nabla_t \fol{v}{y}_1 \equiv 0$ and apply Lemmas \ref{thm.est_v} and \ref{thm.est_v_ex} to $\fol{v}{y}_1$.
This yields all the estimates in \eqref{eq.est_v1}.

\subsubsection{Proof of \eqref{eq.est_v1_cauchy}}

For any $x \in \Sph^2$, we have
\begin{align*}
\lapl ( \fol{v}{y_2}_1 - \fol{v}{y_1}_1 ) |_{ (y_1, x) } &= \lapl \fol{v}{y_2}_1 |_{ (y_2, x) } - \lapl \fol{v}{y_1}_1 |_{ (y_1, x) } - \cint_{y_1}^t \nabla_t \lapl \fol{v}{y_2}_1 d \tau |_{ (y_2, x) } \\
&= - \frac{1}{2} [ \trace \ul{H} |_{ (y_2, x) } - \trace \ul{H} |_{ (y_1, x) } ] + \frac{1}{2} [ \mc{A}_{y_2} ( \trace \ul{H} ) - \mc{A}_{y_1} ( \trace \ul{H} ) ] \\
&\qquad + (1 - y_2) [ E |_{ (y_2, x) } - \mc{A}_{y_2} (E) ] \\
&\qquad - (1 - y_1) [ E |_{ (y_1, x) } - \mc{A}_{y_1} (E) ] - \int_{y_1}^{y_2} \nabla_t \lapl \fol{v}{y_2}_1 |_{ (\tau, x) } d \tau \\
&= ( I_1 + I_2 + I_3 + I_4 + I_5 ) |_x \text{,}
\end{align*}
where we recalled \eqref{poisson}.
The next step is to take the $L^2_x$-norm over $x$ (while recalling that all such norms over the $(S_t, \gamma_t)$'s are equivalent).

For $I_1$, since $\trace \ul{H}$ has an $L^2_x$-limit at $S_1$, it follows that
\[ \lim_{y_1, y_2 \nearrow 1} \int_{ \Sph^2 } |I_1|^2 = 0 \text{.} \]
The same holds for $I_2$ for similar reasons, along with the fact that the areas of the $S_t$'s converge to a limit as $t \nearrow 1$.
The terms $I_3$ and $I_4$ are easier, since by \eqref{Eerror.est},
\[ \lim_{y_1, y_2 \nearrow 1} \int_{ \Sph^2 } ( |I_3|^2 + |I_4|^2 ) \lesssim \lim_{y_1, y_2 \nearrow 1} [1 - \min (y_1, y_2) ]^2 \| E \|_{ L^{2, \infty}_{x, t} }^2 = 0 \text{.} \]
Finally, for $I_5$, we apply \eqref{eq.est_v1}:
\[ \lim_{y_1, y_2 \nearrow 1} \int_{ \Sph^2 } |I_5|^2 \lesssim \lim_{y_1, y_2 \nearrow 1} \| \nabla_t \nabla^2 \fol{v}{y_2}_1 \|_{ L^{2, 1}_{x, t} (S_{y_1}, S_{y_2}) }^2 = 0 \text{.} \]

As a result, we have shown that
\[ \lim_{y_1, y_2 \nearrow 1} \| \lapl ( \fol{v}{y_2}_1 - \fol{v}{y_1}_1 ) \|_{ L^2 (S_{y_1}) } = 0 \text{.} \]
Furthermore, by elliptic estimates (see \cite[Sect. 6.2]{shao:stt}),
\[ \lim_{y_1, y_2 \nearrow 1} [ \| \nabla^2 ( \fol{v}{y_2}_1 - \fol{v}{y_1}_1 ) \|_{ L^2 (S_{y_1}) } + \| \nabla ( \fol{v}{y_2}_1 - \fol{v}{y_1}_1 ) \|_{ L^2 (S_{y_1}) } ] = 0 \text{.} \]
Since $\fol{v}{y_1}_1$ and $\fol{v}{y_2}_1$ are mean-free on $(S_{y_1}, \gamma_{y_1})$ and $(S_{y_2}, \gamma_{y_2})$, respectively, then
\[ \fol{v}{y_2}_1 - \fol{v}{y_1}_1 + \mc{A}_{y_2} (\fol{v}{y_2}_1) - \mc{A}_{y_1} (\fol{v}{y_2}_1) \]
has zero mean on $(S_{y_1}, \gamma_{y_1})$.
Thus, it follows from the Poincar\'e inequality (see the remark immediately after \cite[Prop. 6.5]{shao:stt}) that
\begin{align*}
\lim_{y_1, y_2 \nearrow 1} \| \fol{v}{y_2}_1 - \fol{v}{y_1}_1 \|_{ L^2_x (S_{y_1}) } &\lesssim \lim_{y_1, y_2 \nearrow 1} \| \nabla (\fol{v}{y_2}_1 - \fol{v}{y_1}_1) \|_{ L^2_x (S_{y_1}) } \\
&\qquad + \lim_{y_1, y_2 \nearrow 1} \| \mc{A}_{y_2} (\fol{v}{y_2}_1) - \mc{A}_{y_1} (\fol{v}{y_2}_1) \|_{ L^2_x } \text{.}
\end{align*}
The last term on the right-hand side vanishes, since the areas of the $S_t$'s converge as $t \nearrow 1$.
By standard Sobolev estimates (see \cite[Prop. 2.7]{shao:stt}), we obtain
\[ \lim_{y_1, y_2 \nearrow 1} [ \| \nabla^2 (\fol{v}{y_2}_1 - \fol{v}{y_1}_1) \|_{ L^2_x (S_{y_1}) } + \| \nabla (\fol{v}{y_2}_1 - \fol{v}{y_1}_1) \|_{ L^q_x (S_{y_1}) } + \| \fol{v}{y_2}_1 - \fol{v}{y_2}_1 \|_{ L^\infty_x (S_{y_1}) } ] = 0 \text{.} \]
Applying \eqref{eq.est_v}, with $p = \frac{2q}{q + 2} < 2$, yields the first two limits in \eqref{eq.est_v1_cauchy}.

For the final limit, we first expand:
\begin{align*}
\| \nabla^2 (\fol{v}{y_2}_1 - \fol{v}{y_1}_1) \|_{ L^{2, \infty}_{x, t} (\fol{S}{y_1}_{y_1}, S_1) } &\lesssim \| \nabla^2 (\fol{v}{y_2}_1 - \fol{v}{y_1}_1) \|_{ L^{2, \infty}_{x, t} (S_{y_1}, \fol{S}{y_1}_{y_1}) } \\
&\qquad + \| \nabla^2 (\fol{v}{y_2}_1 - \fol{v}{y_1}_1) \|_{ L^{2, \infty}_{x, t} (S_{y_1}, S_1) } \\
&= J_1 + J_2 \text{.}
\end{align*}
Applying \eqref{eq.est_v_partial}, with $q = \infty$, yields
\begin{align*}
J_2 &\lesssim \| \nabla^2 (\fol{v}{y_2}_1 - \fol{v}{y_1}_1) \|_{ L^2_x (S_{y_1}) } + \| H \|_{ L^{\infty, 1}_{x, t} (S_{y_1}, S_1) } \| \nabla^2 (\fol{v}{y_2}_1 - \fol{v}{y_1}_1) \|_{ L^{2, \infty}_{x, t} (S_{y_1}, S_1) } \\
&\qquad + \| \nabla H \|_{ L^{2, 1}_{x, t} (S_{y_1}, S_1) } ( \| \nabla \fol{v}{y_2}_1 \|_{ L^{\infty, \infty}_{t, x} } + \| \nabla \fol{v}{y_1}_1 \|_{ L^{\infty, \infty}_{t, x} } ) \\
&= J_{2, 1} + J_{2, 2} + J_{2, 3} \text{.}
\end{align*}
By \eqref{eq.renorm_est}, $J_{2, 2}$ can be absorbed into the left-hand side, while the preceding arguments show that $J_{2, 1} \rightarrow 0$ as $y_1, y_2 \nearrow 1$.
For $J_{2, 3}$, we apply \eqref{eq.renorm_est} and \eqref{eq.est_v1}:
\[ \lim_{y_1, y_2 \nearrow 1} J_{2, 3} \lesssim \lim_{y_1, y_2 \nearrow 1} (1 - y_1)^\frac{1}{2} \| \nabla H \|_{ L^{2, 2}_{t, x} } \Gamma = 0 \text{.} \]
The remaining term $J_1$ is controlled analogously, completing the proof of \eqref{eq.est_v1_cauchy}.

\subsection{Proof of Lemma \ref{thm.est_v2}} \label{sect:estimates.v2}

Applying the $L^2$-estimates for the Hodge operators from \cite[Prop. 6.4]{shao:stt} to \eqref{eqn.v_2} and recalling \eqref{eq.est_u} and \eqref{eq.est_v1}, we obtain
\begin{align*}
\| \nabla^2 \fol{v}{y}_2 \|_{ L^2_x (S_y) } + \| \nabla \fol{v}{y}_2 \|_{ L^2_x (S_y) } &\lesssim \| \ddot{\mc{K}} - 1 \|_{ L^2_x (S_y) } + \| e^{2 \fol{v}{y}_2} - 1 \|_{ L^2_x (S_y) } \lesssim \Gamma \text{.}
\end{align*}
Next, applying \cite[Cor. 3.7]{shao:stt}, \eqref{eq.est_u}, \eqref{eq.est_v1}, and the above to \eqref{eqn.v_2} yields
\begin{align*}
\| \lapl \fol{v}{y}_2 \|_{ B^0_x (S_y) } &= [ \| \nabla (u + \fol{v}{y}_1) \|_{ L^2_x (S_y) } + \| u + \fol{v}{y}_1 \|_{ L^\infty_x (S_y) } ] \| \ddot{\mc{K}} - e^{2 \fol{v}{y}_2} \|_{ B^0_x (S_y) } \\
&\lesssim \| \ddot{\mc{K}} - 1 \|_{ B^0_x (S_y) } + \| e^{2 \fol{v}{y}_2} - 1 \|_{ B^0_x (S_y) } \text{.}
\end{align*}
Recalling the explicit formula \eqref{new.bound} for $\ddot{\mc{K}}$, then
\begin{align*}
\| \lapl \fol{v}{y}_2 \|_{ B^0_x (S_y) } &\lesssim \| \ddot{\mc{K}} - 1 \|_{ H^1_x (S_y) } + \| e^{2 \fol{v}{y}_2} - 1 \|_{ H^1_x (S_y) } \\
&\lesssim [ \| \nabla (u + \fol{v}{y}_1) \|_{ L^2_x (S_y) } + \| \ddot{\mc{K}} - 1 \|_{ L^\infty_x (S_y) } ] \\
\notag &\qquad + ( \| \nabla \fol{v}{y}_2 \|_{ L^2_x (S_y) } + \| \fol{v}{y}_2 \|_{ L^\infty_x (S_y) } ) \\
&\lesssim \Gamma \text{.}
\end{align*}

Thus, by \cite[Thm. 6.11]{shao:stt},
\[ \| \nabla^2 \fol{v}{y}_2 \|_{ B^0_x (S_y) } + \| \nabla \fol{v}{y}_2 \|_{ L^\infty_x (S_y) } \lesssim \| \lapl \fol{v}{y}_2 \|_{ B^0_x (S_y) } \lesssim \Gamma \text{.} \]
Combining this with \eqref{eq.est_v2_pre} yields the full set of estimates for $\fol{v}{y}_2$ on $(S_y, \gamma_y)$:
\begin{equation} \label{sharp.v2} \| \nabla^2 \fol{v}{y}_2 \|_{ B^0_x (S_y) } + \| \nabla \fol{v}{y}_2 \|_{ L^\infty_x (S_y) } + \| \fol{v}{y}_2 \|_{ L^\infty_x (S_y) } \lesssim \Gamma \text{.} \end{equation}
Applying Lemmas \ref{thm.est_v} and \ref{thm.est_v_ex} to \eqref{sharp.v2} yields \eqref{eq.est_v2}.

\subsubsection{Proof of \eqref{eq.est_v2_cauchy}}

This is similar to the proof of \eqref{eq.est_v1_cauchy}.
First, for $x \in \Sph^2$,
\begin{align*}
\lapl ( \fol{v}{y_2}_2 - \fol{v}{y_1}_2 ) |_{ (y_1, x) } &= \lapl \fol{v}{y_2}_2 |_{ (y_2, x) } - \lapl \fol{v}{y_1}_2 |_{ (y_1, x) } - \cint_{y_1}^t \nabla_t \lapl \fol{v}{y_2}_2 d \tau |_{ (y_2, x) } \\
&= [ e^{2 (u + \fol{v}{y_2}_1)} \ddot{\mc{K}} |_{ (y_2, x) } - e^{2 (u + \fol{v}{y_1}_1)} \ddot{\mc{K}} |_{ (y_1, x) } ] - \cint_{y_1}^t \nabla_t \lapl \fol{v}{y_2}_2 d \tau |_{ (y_2, x) } \\
&\qquad - [ e^{2 (\fol{v}{y_2}_1 + \fol{v}{y_2}_2)} |_{ (y_2, x) } - e^{2 (\fol{v}{y_1}_1 + \fol{v}{y_1}_2)} |_{ (y_1, x) } ] \\
&\qquad - (e^{2u} - 1) e^{2 (\fol{v}{y_2}_1 + \fol{v}{y_2}_2)} |_{ (y_2, x) } + (e^{2u} - 1) e^{2 (\fol{v}{y_1}_1 + \fol{v}{y_1}_2)} |_{ (y_1, x) } \\
&= ( I_1 + I_2 + I_3 + I_4 + I_5 ) |_x \text{,}
\end{align*}
where we recalled the equation \eqref{eqn.v_2}.
For $I_2$, we apply \eqref{eq.est_v2},
\[ \lim_{y_1, y_2 \nearrow 1} \int_{ \Sph^2 } |I_2|^2 \lesssim \lim_{y_1, y_2 \nearrow 1} \| \nabla_t \nabla^2 \fol{v}{y_2}_2 \|_{ L^{2, 1}_{x, t} (S_{y_1}, S_{y_2}) }^2 = 0 \text{,} \]
while for $I_3$, we apply \eqref{eq.est_v1_cauchy} and \eqref{eq.est_v2_cauchy_pre},
\[ \lim_{y_1, y_2 \nearrow 1} \int_{ \Sph^2 } |I_3|^2 \lesssim \lim_{y_1, y_2 \nearrow 1} ( \| \fol{v}{y_2}_1 - \fol{v}{y_1}_1 \|_{ L^{\infty, \infty}_{t, x} } + \| \fol{v}{y_2}_2 - \fol{v}{y_1}_2 \|_{ L^{\infty, \infty}_{t, x} } ) = 0 \text{.} \]
$I_4$ and $I_5$ can be controlled using \eqref{eq.est_u}, \eqref{eq.est_v1}, and \eqref{eq.est_v2}:
\[ \lim_{y_1, y_2 \nearrow 1} \int_{ \Sph^2 } ( |I_4|^2 + |I_5|^2 ) \lesssim \lim_{y_1, y_2 \nearrow 1} [ \| u \|_{ L^\infty (S_{y_2}) } + \| u \|_{ L^\infty (S_{y_1}) } ] = 0 \text{.} \]
For $I_1$, we expand the definitions of $\ddot{\mc{K}}$ and $\bar{\mc{K}}$ using \eqref{eq.K_bar} and \eqref{new.bound}:
\begin{align*}
I_1 |_x &= \mc{A}_{y_2} ( e^{2 u} \bar{\mc{K}} ) - \mc{A}_{y_1} ( e^{2 u} \bar{\mc{K}} ) \\
&= - \frac{1}{2} [ \mc{A}_{y_2} ( \trace \ul{H} ) - \mc{A}_{y_1} ( \trace \ul{H} ) ] + (1 - y_2) \mc{A}_{y_2} (E) - (1 - y_1) \mc{A}_{y_1} (E) \text{.}
\end{align*}
As discussed within Appendix \ref{sect:estimates.v1}, each term on the right-hand side converges to $0$ as $y_1, y_2 \nearrow 1$.
Therefore, combining the above, we obtain
\[ \lim_{y_1, y_2 \nearrow \infty} \| \lapl ( \fol{v}{y_2}_2 - \fol{v}{y_1}_2 ) \|_{ L^2_x (S_{y_1}) } = 0 \text{.} \]

Combining \eqref{eq.est_v2_cauchy_pre}, $L^2$-Hodge estimates (see \cite[Prop. 6.4]{shao:stt}), and Sobolev embedding estimates (see \cite[Prop. 2.7]{shao:stt}), we have for any $2 < q < \infty$ that
\[ \lim_{y_1, y_2 \nearrow 1} [ \| \nabla^2 (\fol{v}{y_2}_2 - \fol{v}{y_1}_2) \|_{ L^2_x (S_{y_1}) } + \| \nabla (\fol{v}{y_2}_2 - \fol{v}{y_1}_2) \|_{ L^q_x (S_{y_1}) } + \| \fol{v}{y_2}_2 - \fol{v}{y_2}_2 \|_{ L^\infty_x (S_{y_1}) } ] = 0 \text{.} \]
Using \eqref{eq.est_v}, we derive the first two limits in \eqref{eq.est_v2_cauchy}.
The final estimate in \eqref{eq.est_v2_cauchy} is proved in precisely the same manner as the analogous estimate for the $\fol{v}{y}_1$'s.

\section{Proof of Lemma \ref{thm.unif}} \label{sect:distortion.unif}

Here, we sketch one proof for the uniformization result in Lemma \ref{thm.unif}.
For this, we adopt a modification of the argument found in \cite[Sect. 2.4]{chr_kl:stb_mink}; in particular, we break the conformal invariance for the $2$-sphere by explicitly constructing our conformal factor $v$.
\footnote{In \cite[Sect. 2.4]{chr_kl:stb_mink}, the authors constructed uniformizing factors that were shown to be \emph{bounded}.
However, a more refined construction is better suited for observing \emph{smallness}.}
As in \cite{chr_kl:stb_mink}, the key will be to first transform $h$ into the flat metric via a conformal factor that is close to that for the stereographic projection.

\subsubsection*{Normal Coordinates}

Since $\mc{K}_h$ is uniformly near $1$ by \eqref{eq.unif_ass}, standard estimates (see \cite{ch_eb:comp}) imply the diameter $D$ and injectivity radius $\mc{R}$ of $(\Sph^2, h)$ satisfy
\[ \pi - \varepsilon \leq D \leq \pi + \varepsilon \text{,} \qquad \pi - \varepsilon \leq \mc{R} \leq \pi + \varepsilon \text{,} \qquad \varepsilon \lesssim \Gamma \text{.} \]
Thus, given a point $P \in \Sph^2$, we can consider normal polar coordinates $(\lambda_P, \varphi_P)$ in an open geodesic ball $B_P$ of radius $\pi - \varepsilon$ about $P$, so that $h$ takes the form
\[ h = d \lambda_P^2 + R^2 (\lambda_P, \varphi_P) \cdot d \varphi_P^2 \text{.} \]

\begin{remark}
In the case that $h$ is the round metric, with $\mc{K}_h \equiv 1$:
\begin{itemize}
\item If $P$ corresponds to the north pole of the sphere, then $(\lambda_P, \varphi_P)$ corresponds precisely to the spherical coordinates $(\theta, \phi)$.

\item If $P$ is the south pole of the sphere, then $(\lambda_P, \varphi_P)$ corresponds to $(\pi - \theta, \phi)$.
\end{itemize}
\end{remark}

The mean curvatures of the level circles of $\lambda_P$ are given by
\[ \mc{H}_P = R_P^{-1} \cdot \partial_{ \lambda_P } R_P \text{.} \]
Recall (see, e.g. \cite[Sect. 2.4]{chr_kl:stb_mink}) that $\mc{H}_P$ satisfies the Riccati equations
\begin{equation} \label{eq.riccati} \partial_{ \lambda_P } \mc{H}_P = - \mc{H}_P^2 - \mc{K}_h \text{,} \qquad \lim_{ \lambda_P \searrow 0 } ( \mc{H}_P - \lambda_P^{-1} ) = 0 \text{.} \end{equation}
In particular, if $\mc{K}_h$ is a positive constant $k > 0$, then
\[ \mc{H}_P = \mc{H}_{P, k} = \sqrt{k} \cdot \cot \frac{ \lambda_P }{ \sqrt{k} } \text{.} \]
Moreover, since $1 - \varepsilon^\prime \leq \mc{K} \leq 1 + \varepsilon^\prime$ for some $\varepsilon^\prime \lesssim \Gamma$, then standard comparison arguments using \eqref{eq.riccati} result in the bounds
\[ \mc{H}_{P, 1 + \varepsilon^\prime} \leq \mc{H}_P \leq \mc{H}_{P, 1 - \varepsilon^\prime} \text{.} \]
From this, it follows that
\begin{equation} \label{eq.bound_h} \lambda_P^{-1} | \mc{H}_P - \mc{H}_{P, 1} | \lesssim \Gamma \text{.} \end{equation}

In addition, we define the functions
\[ W_{P, n} = - 2 \log \sin \frac{\lambda_P}{2} \text{,} \qquad W_{P, s} = - 2 \log \cos \frac{\lambda_P}{2} \text{.} \]
Note that whenever $h$ is the round metric: if $P$ is the north/south pole on $\Sph^2$, then $W_{P, n}$/$W_{P, s}$ (resp.) is precisely the conformal factor
\[ (\theta, \phi) \mapsto - 2 \log \sin \frac{\theta}{2} \]
associated with the stereographic projection from $\Sph^2$ onto $\R^2$.
\footnote{More precisely, the specific stereographic projection we use here is that from the unit sphere about the origin in $\R^3$ onto the plane $z = -1$ in $\R^3$.}
Moreover, letting $\Delta_h$ denote the Laplacian associated with $h$, then $W_{P, n}$ and $W_{P, s}$ satisfy
\[ \Delta_h W_{P, n} = 1 - \cot \frac{\lambda_P}{2} \cdot ( \mc{H}_P - \mc{H}_{P, 1} ) \text{,} \qquad \Delta_h W_{P, s} = 1 + \tan \frac{\lambda_P}{2} \cdot ( \mc{H}_P - \mc{H}_{P, 1} ) \text{,} \]
and hence by \eqref{eq.K_dotdot} and \eqref{eq.bound_h}, we can estimate
\begin{equation} \label{eq.bound_W} | \mc{K}_h - \Delta_h W_{P, n} | \lesssim \Gamma \text{,} \qquad | \mc{K}_h - \Delta_h W_{P, s} | \lesssim \Gamma \text{,} \qquad \lambda_P \leq \frac{2 \pi}{3} \text{.} \end{equation}

\subsubsection*{Construction of the Uniformizing Factor}

We are now prepared to construct the desired factor $v$.
Fix first a pair of points $N, S \in \Sph^2$ such that
\[ \pi - \varepsilon \leq d (N, S) \leq \pi + \varepsilon \text{.} \]
The idea is to treat $N$ and $S$ as the eventual north and south poles, and to approximate the conformal factor for the stereographic projection using the functions $W_{N, n}$ and $W_{N, s}$.
Fixing a smooth cutoff function
\[ \phi: \Sph^2 \rightarrow [0, 1] \text{,} \qquad \phi = \begin{cases} 1 & \text{on the geodesic ball } B_\frac{\pi}{3} (N) \text{,} \\ 0 & \text{on the geodesic ball } B_\frac{\pi}{3} (S) \text{,} \end{cases} \]
we make the following initial guess for the approximate stereographic factor:
\footnote{In particular, we require two normal coordinate systems in our construction, since normal coordinates degenerate as one approaches the injectivity radius.}
\[ w_0 = \phi \cdot W_{N, n} + (1 - \phi) \cdot W_{S, s} \text{.} \]
Note that when $h$ is round, $w_0$ is precisely the stereographic conformal factor.

The actual conformal factor to transform $h$ to the flat metric will differ from $w_0$ by an error term.
To determine this error, we consider the function
\begin{align*}
f &= \mc{K}_h - \Delta_h w_0 = f_0 + f_1 + f_2 \text{,} \\
f_0 &= \phi \cdot ( \mc{K}_h - \Delta_h W_{N, n} ) + (1 - \phi) \cdot ( \mc{K}_h - \Delta_h W_{S, s} ) \text{,} \\
f_1 &= \partial_{\lambda_N} \phi \cdot \cot \frac{\lambda_N}{2} + \partial_{\lambda_S} \phi \cdot \tan \frac{\lambda_S}{2} \text{,} \\
f_2 &= 2 \Delta_h \phi \cdot \left[ \log \sin \frac{\lambda_N}{2} - \log \cos \frac{\lambda_S}{2} \right] \text{.}
\end{align*}
In particular, $f$ is bounded on all of $\Sph^2$.
Furthermore, the Gauss-Bonnet theorem and a divergence theorem argument as in \cite[Sect. 2.4]{chr_kl:stb_mink} imply that $f$ is mean-free.
As a result, we can solve the Poisson equation
\[ \Delta_h w_E = f \text{,} \qquad \int_{ \Sph^2 } w_E = 0 \text{.} \]
Defining now $w = w_E + w_0$, which satisfies on $\Sph^2 \setminus \{ N \}$ the equation $\Delta_h w = \mc{K}_h$, we see that $\tilde{h} := e^{2 w} h$ defines a flat metric on $\Sph^2 \setminus \{ N \}$.

Now that we are on the plane, we can return to the \emph{round} sphere by inverting the (standard) stereographic projection.
Letting $\tilde{d}_S$ denote the $\tilde{h}$-distance from $S$, the conformal factor associated with this inverse stereographic projection is
\[ \tilde{w} = - \log \left( 1 + \frac{ \tilde{d}_S^2 }{4} \right) \text{,} \]
where we treated $S$ as the origin in $\R^2$.
Therefore, if we define
\[ v = w + \tilde{w} \text{,} \]
then the metric $\mathring{h} = e^{2 v} h$ will be round, i.e., its curvature satisfies $\mc{K}_{ \mathring{h} } \equiv 1$.
In particular, $v$ satisfies the following nonlinear equations on $\Sph^2 \setminus \{ N \}$:
\begin{equation} \label{eq.unif_pde} \Delta_h v = \mc{K}_h - e^{2 v} \text{,} \qquad - \Delta_{ \mathring{h} } v = e^{2 v} \mc{K}_h - 1 \text{.} \end{equation}

\subsubsection*{Bounds on the Uniformization Factors}

Finally, we briefly sketch the proof of the bounds for $v$.
Note first that \eqref{eq.bound_W} immediately implies
\begin{equation} \label{eq.bound_fK} \| f_0 \|_{ L^\infty_x (\Sph^2) } \lesssim \Gamma \text{.} \end{equation}
For $f_1$ and $f_2$, we require the following observations:
\begin{itemize}
\item Both $f_1$ and $f_2$ are supported away from both $N$ and $S$.
\footnote{In particular, $\sin ( \lambda_N / 2)$, $\cos ( \lambda_N / 2)$, and the corresponding quantities for $\lambda_S$ are uniformly bounded from above and below in the supports of $f_1$ and $f_2$.}

\item Since $N$ and $S$ almost achieve the diameter of $(\Sph^2, h)$, it follows that $\lambda_S$ will be (uniformly) close to $\pi - \lambda_N$ in the supports of $f_1$ and $f_2$.

\item Moreover, when radial geodesics from $N$ and $S$ intersect in this region, they will point in almost opposite directions.
\footnote{This can be observed, e.g., using Toponogov's comparison theorem; see \cite{ch_eb:comp}.}
\end{itemize}
Combined, these observations imply that $f_1$ and $f_2$ are uniformly small.
A more careful quantitative analysis of this yields the estimates
\[ \| f_1 \|_{ L^\infty_x (\Sph^2) } + \| f_2 \|_{ L^\infty_x (\Sph^2) } \lesssim \Gamma \text{.} \]
This controls $f$ by $\Gamma$, and standard elliptic estimates now imply
\[ \| w_E \|_{ L^\infty_x (\Sph^2) } \lesssim \Gamma \text{.} \]

An analogue of the argument found in \cite[Sect. 2.4]{chr_kl:stb_mink} immediately yields that $w_0 + \tilde{w}$ is uniformly bounded.
To show smallness, however, we observe that $w_0$, as constructed, approximates the conformal factor for the stereographic projection, while $\tilde{w}$ is the (exact) conformal factor for the inverse stereographic projection.
A more careful accounting, using arguments similar to \cite[Sect. 2.4]{chr_kl:stb_mink} comparing $h$- and $\mathring{h}$-geodesics, yields the more precise estimate
\[ \| w_0 + \tilde{w} \|_{ L^\infty_x (\Sph^2 \setminus \{ N \}) } \lesssim \Gamma \text{.} \]

Collecting all the preceding estimates results in \eqref{eq.unif_est}; in particular, $v$ extends to a bounded function on $\Sph^2$.
Furthermore, using the nonlinear equation \eqref{eq.unif_pde} and the smoothness of $\mc{K}_h$, we can improve the regularity of $v$ and derive smoothness.

To show that $v$ depends continuously on $h$ and $\mc{K}_h$, we return to each step of its construction, and we observe that each of the components $w_E$ and $w_0 + \tilde{w}$ depends continuously on $h$ and $\mc{K}_h$.
To better sketch the main points of this argument, we let $h^\prime$ be another metric on $\Sph^2$ such that $h^\prime$ and its curvature $\mc{K}_{h^\prime}$ are uniformly very close to $h$ and $\mc{K}$, respectively.
Moreover, let $w_E^\prime$, $w_0^\prime$, $\tilde{w}^\prime$, and $v^\prime$ denote the various components obtained in the above process, but in terms of $h^\prime$.
\footnote{In particular, since $h^\prime$ is near $h$, we can choose the same pair of points $N$ and $S$ in the construction of $v^\prime$ from $h^\prime$ as for the construction of $v$ from $h$.}

The first point is that since $h^\prime$ is close to $h$, the normal coordinates $\lambda^\prime_N$ and $\lambda^\prime_S$ with respect to $h^\prime$ are similarly close to those for $h$, up to first derivatives.
\footnote{The cutoff functions $\phi$ and $\phi^\prime$ exclude the regions where these normal coordinates degenerate.}
From this, we can conclude that $w_0^\prime + \tilde{w}^\prime$, $f_1^\prime$, and $f_2^\prime$ lie uniformly close to $w_0 + \tilde{w}$, $f_1$, and $f_2$, respectively.
To similarly compare $f_0^\prime$ and $f_0$, we also require the closeness of curvatures.
Note that since $\mc{K}_{h^\prime}$ and $\mc{K}_h$ are close, the Riccati equation \eqref{eq.riccati} and its counterpart for $h^\prime$ imply that both $\mc{H}_N - \mc{H}_N^\prime$ and $\mc{H}_S - \mc{H}_S^\prime$ remain small.
Thus, by definition, $f_0^\prime$ must lie uniformly close to $f_0$.

Finally, to compare $w_E^\prime$ with $w_E$, we consider the linear elliptic equation
\begin{align*}
\Delta_h ( w_E^\prime - w_E ) &= \Delta_{h^\prime} w_E^\prime - \Delta_h w_E + ( \Delta_h - \Delta_{h^\prime} ) w_E^\prime \\
&= ( f_0^\prime - f_0 ) + ( f_1^\prime - f_1 ) + ( f_2^\prime - f_2 ) + ( \Delta_h - \Delta_{h^\prime} ) w_E^\prime \text{.} 
\end{align*}
The first three terms on the right-hand side will be small by the preceding discussion; since $(h^\prime, \mc{K}_{h^\prime})$ is close to $(h, \mc{K}_h)$, the difference of Laplacians will also be small.
Consequently, standard elliptic estimates imply that $w_E^\prime$ lies close to $w_E$.
Combining all the above, we conclude that $v^\prime - v$ is uniformly small.

\raggedright
\bibliographystyle{amsplain}
\bibliography{articles,books,misc}

\end{document}